\documentclass[10pt]{amsart}

\usepackage{mathtools,multicol}
\usepackage{kpfonts}
\usepackage{mathdots,enumerate}
\usepackage{multirow}
\usepackage{amssymb,amsmath,amsthm}
\usepackage{pst-node}
\usepackage{tikz-cd}
\usepackage{tikz}
\usepackage{enumitem}
\usepackage{tabularx}
\usepackage{float}

\newcounter{rownumber}
\numberwithin{equation}{section}

\parskip=\smallskipamount
\makeatletter 
\renewcommand\p@enumii{}
\makeatother

\theoremstyle{plain}
\newtheorem{theorem}{Theorem}[section]
\newtheorem{lemma}[theorem]{Lemma}
\newtheorem{corollary}[theorem]{Corollary}

\theoremstyle{definition}

\theoremstyle{remark}
\newtheorem{remark}[theorem]{Remark}

\def\id{\mathop{\rm id}\nolimits}
\def\dim{\mathop{\rm dim}\nolimits}

\def\Rea{\mathop{\rm Re}\nolimits}
\def\Ima{\mathop{\rm Im}\nolimits}
\def\Orb{\mathop{\rm Orb}\nolimits}

\def\Bun{\mathop{\rm Bun}\nolimits}

\begin{document}


\title[]
{On structures of normal forms of complex points of small $\mathcal{C}^{2}$-perturbations of real $4$-manifolds embedded in a complex $3$-manifold}
\author{Tadej Star\v{c}i\v{c}}
\address{Faculty of Education, University of Ljubljana, Kardeljeva Plo\v{s}\v{c}ad 16, 1000 Lju\-blja\-na, Slovenia}
\address{Institute of Mathematics, Physics and Mechanics, Jadranska
  19, 1000 Ljubljana, Slovenia}
\email{tadej.starcic@pef.uni-lj.si}
\subjclass[2010]{32V40,58K50,15A21}
\date{19 July 2022}

\keywords{CR manifolds, closure graphs, complex points, normal forms, perturbations} 

\begin{abstract}  
We extend our previous result on the behaviour of the quadratic part of a complex points of a small $\mathcal{C}^{2}$-perturbation of a real $4$-manifold embedded in a complex $3$-manifold. 
We describe the change of the structure of the quadratic normal form of a complex point.
It is an immediate consequence of a theorem clarifying how small perturbations can change the bundle of a pair of one arbitrary and one symmetric $2\times 2$ matrix with respect to 
an action of a certain linear group.
\end{abstract}

\maketitle

\vspace{-2mm}

\section{Introduction} \label{intro}

Let $ M$ be a smooth real $2n$-submanifold in $\mathbb{C}^{n+1}$. A point
$p\in M$ is called \textit{complex} when $T_pM$ is a complex subspace in $T_p X$; its complex dimension is equal to $n$.   
Locally, near a complex point $p\in M$ we can see $M$ as a graph 
(see e.g. \cite{TS2}):
\begin{equation}\label{BasForm1}
w=\overline{z}^TAz+\Rea (z^TBz)+o(|z|^2), \quad (w(p),z(p))=(0,0),\,\,\, A\in \mathbb{C}^{n\times n}, B\in \mathbb{C}^{n\times n}_S,
\end{equation}
in which $(z,w)=(z_1,z_2,\ldots,z_n,w)$ are suitable local coordinates on
$X$, and $\mathbb{C}^{n\times n}$, $\mathbb{C}^{n\times n}_S$ are sets of all $n\times n$ matrices and $n\times n$ symmetric matrices, respectively.
A complex point $p$ is 
quadratically flat, if the quadratic part of (\ref{BasForm1}) is real valued.

When $n=1$ complex points are well understood; see papers of Bishop
\cite{Bishop}, Kenig and Webster \cite{KW}, Moser and Webster \cite{MW},
Bedford and Klingenberg \cite{BK} 
and Forst\-ne\-ri\v c \cite{F}. They are quadratically flat and given locally by
$w=z\overline{z}+\frac{\gamma}{2} (z^2+\overline{z}^2)+o(|z|^2)$, $0\leq \gamma$, or $w=z^2+\overline{z}^2+o(|z|^2)$.  
For $n=2$ a relatively simple
description of complex points up to quadratic terms was obtained by Coffman \cite{Coff}; it includes two gene\-ric normal forms 
$w=\tau z_1\overline{z}_2+z_2\overline{z}_1+\frac{a}{2}(z_1^{2}+\overline{z}_1^{2})+b(z_1z_2+\overline{z}_1\overline{z}_2)+\frac{d}{2}(z_2^{2}+\overline{z}_2^{2})+o(|z|^{2})
$,
$|a|=1$, $b>0$, $d\in \mathbb{C}$, $\tau\in (0,1)$, and 
$w=z_1\overline{z}_1+e^{i\theta}z_2\overline{z}_2+\frac{a}{2}(z_1^{2}+\overline{z}_1^{2})+b(z_1z_2+\overline{z}_1\overline{z}_2)+\frac{d}{2}(z_2^{2}+\overline{z}_2^{2})+o(|z|^{2})$, $\theta\in (0,\pi)$, $a,d>0$, $b\in \mathbb{C}^{*}$ (apply $(A,B)=
\big(\begin{bsmallmatrix}
0 & 1 \\
\tau  & 0
\end{bsmallmatrix},
\begin{bsmallmatrix}
a & b \\
b  & d
\end{bsmallmatrix}\big)$ and $(A,B)=\big(\begin{bsmallmatrix}
1 & 0 \\
0  & e^{i\theta}
\end{bsmallmatrix},
\begin{bsmallmatrix}
a & b \\
b  & d
\end{bsmallmatrix}\big)$ to (\ref{BasForm1}) for $n=2$, respectively).
If $n>2$ the quadratically flat complex points were studied
by 
Slapar and the author \cite{ST}. 
We refer to the papers of Dolbeault, Tomassini and Zaitsev \cite{DTZ1} and Fang and Huang \cite{FH} for results on holomorphic flattenability of $CR$-nonminimal 
real analytic submanifolds near complex points.
Note that formal normal forms of $CR$-singularities were considered by Burcea \cite{Bur} and 
Gong and Stolovitch \cite{GongStolo1}, 
among others.

In this paper we continue the research started in our 
paper \cite{TS2}, in which we explained when 
the quadratic part of a complex point 
of a real $4$-manifold embedded in a complex $3$-manifold
can be transformed under small $\mathcal{C}^{2}$-perturbations 
to the quadratic part of another different complex point.
For instance, \cite[Corollary 3.8]{TS2} implies that no sufficiently small $\mathcal{C}^{2}$-deformation of $w=z_1\overline{z}_1+\frac{1}{2}z_1^{2}+\frac{1}{2}z_2^{2}+o(|z|^{2})$ near $(0,0)$ (with $(A,B)=
\big(\begin{bsmallmatrix}
1 & 0 \\
0  & 0
\end{bsmallmatrix},
\begin{bsmallmatrix}
0 & 0 \\
0  & 1
\end{bsmallmatrix}\big)$ in (\ref{BasForm1}) for $n=2$) can lead to 
$w=\frac{1}{2} z_1\overline{z}_2+z_2\overline{z}_1+z_1^{2}+\overline{z}_1^{2}+z_1z_2+\overline{z}_1\overline{z}_2+z_2^{2}+\overline{z}_2^{2}+o(|z|^{2})$ (with $(A,B)=
\big(\begin{bsmallmatrix}
0 & 1 \\
\frac{1}{2}  & 0
\end{bsmallmatrix},
\begin{bsmallmatrix}
2 & 1 \\
1  & 2
\end{bsmallmatrix}\big)$ in (\ref{BasForm1}) for $n=2$).

We now focus on the change of the type of a 
complex point, i.e. on the structu\-re of $(A,B)$ in (\ref{BasForm1}).
In particular, 
we provide the following result describing possible arbitrarily small $\mathcal{C}^{2}$-deformations to generic normal forms.

\begin{theorem}
Let $M$ be a real $4$-manifold in $\mathbb{C}^{3}$ and let $p\in M$ be a complex point given locally by the equation (\ref{BasForm1}) with $n=2$ and $(A,B)=(A_p,B_p)$.
Then $A_p\neq 
\begin{bsmallmatrix}
0 & 1 \\
\tau_p  & 0
\end{bsmallmatrix}$ with $\tau_p\in (0,1)$ (or $A_p\neq 
\begin{bsmallmatrix}
1 & 0 \\
0  & e^{i\theta_p}
\end{bsmallmatrix}$ with $\theta_p \in (0,\pi)$) if and only if there exists an arbitrarily small $\mathcal{C}^{2}$-perturbation $M'$ of $M$, and such that $M'$ has a complex point $p'$, arbitrarily close to $p$, and $p'$ is locally given by 
the equation (\ref{BasForm1}) with $n=2$ and $(A,B)=\big(\begin{bsmallmatrix}
0 & 1 \\
\tau'  & 0
\end{bsmallmatrix},
\begin{bsmallmatrix}
a' & b' \\
b'  & d'
\end{bsmallmatrix}\big)$
for some $\tau'\in (0,1)$, $|a'|=1$, $b'>0$, $d'\in \mathbb{C}$ (and $(A,B)=\big(\begin{bsmallmatrix}
1 & 0 \\
0  & e^{i\theta'}
\end{bsmallmatrix},
\begin{bsmallmatrix}
a' & b' \\
b'  & d'
\end{bsmallmatrix}\big)$ for some $\theta'\in (0,\pi)$, $a',d'>0$, $b'\in \mathbb{C}^{*}$).
\end{theorem}

\vspace{-2mm}

A more general situaton is considered in 
Theorem \ref{izrek} and Corollary \ref{pospert}. 
However, due to technical reasons, these results are precisely stated in Sec. \ref{Sec3} and proved in later sections.
A substantial difference in comparison to \cite{TS2} is that our problem now reduces to a system of nonlinear 
equations with larger set of parameters. In general it makes the analysis considerably more involved. 
We add that 
Theorem \ref{izrek} might be of independent interest in matrix analysis since it
clarifies how small perturbations can change the bundle 
of a pair 
$(A,B)\in \mathbb{C}^{2\times 2}\times \mathbb{C}^{2\times 2}_S$
with respect to 
transformations $(cP^{*}AP,P^{T}BP)$ with nonsingular matrix $P$ and $|c|=1$. 


%
%
%
%
%
\section{Normal forms in dimension $2$}\label{secNF}

Any holomorphic change of coordinates that
preserves (\ref{BasForm1}) for $n=2$ 
transforms (\ref{BasForm1}) into the equation that can by a slight abuse of notation be written as
\vspace{-2mm}
\begin{equation*}
w=\overline{z}^T\left(cP^{*}AP\right) z+\Rea \left(z^T (P^TBP) z\right)+o(|z|^2), 
\qquad P\in GL_2(\mathbb{C}),\quad c\in S^{1},
\end{equation*}
%
where $S^1$ and $ GL_2(\mathbb{C})$ are a unit circle and the group of invertible $2\times 2$ matrices, respectively.
Studying the quadratic part of a complex point 
thus means examining the action of $S^1\times GL_2(\mathbb{C})$ 
on $\mathbb{C}^{2\times 2}\times \mathbb{C}^{2\times 2}_S $ (see also \cite{Coff} and \cite[Sec. 3]{TS2}):
\vspace{-1mm}
\begin{equation}\label{aAB}
\Psi\colon\bigl((c,P),(A,B)\bigr)\mapsto (cP^{*}AP,P^TBP), \qquad P\in GL_2(\mathbb{C}),\quad c\in S^{1}. 
\end{equation}
\vspace{-1mm}
%
An orbit at $(A,B)\in \mathbb{C}^{2\times 2}\times \mathbb{C}^{2\times 2}_S $ with respect to (\ref{aAB}) is denoted by $\Orb_{\Psi}(A,B)$.

For some applications it is 
useful to have 
a stratification into \emph{bundles} of matrices, i.e. sets of all matrices having similar properties. This notion was first introduced by Arnold \cite[Section 30]{Arnold} for the action of similarity; two matrices are in the same bundle under similarity precisely when their Jordan canonical forms have the same structure (with bijection between the sets of distinct eigenvalues). For instance, matrices with all distinct eigenvalues form the generic bundle.

Three bundles with respect to the action (\ref{aAB}) can be formed according to the sign of 
$\det\begin{bsmallmatrix}
A & \overline{B}\\
B & \overline{A}
\end{bsmallmatrix}
$ for $(A,B)\in \mathbb{C}^{2\times 2}\times \mathbb{C}^{2\times 2}_S$
(see \cite[Sec. 4]{Coff}). Slapar \cite{S2} (see also \cite{S3}) proved that the bundles with nonvanishing determinant are connected components of $\mathbb{C}^{2\times 2}\times \mathbb{C}^{2\times 2}_S$ and 
showed that up to smooth isotopy complex points of a real $4$-submanifold in $\mathbb{C}^{3}$ are locally given either by $w=z_1\overline{z}_1+z_2\overline{z}_2$ or $w=z_1\overline{z}_1+\overline{z}_2^{2}$.

Our goal is to understand 
the change of normal forms of (\ref{aAB})
under small perturbations, thus we use the list \cite[Sec. 7,Table 1]{Coff} (see also \cite[Lemma 2.2]{TS2}) of normal forms for orbits under (\ref{aAB}), 
to form bundles so that they contain 
normal forms of similar structure. To be more precise, each such set of normal forms 
is parameterized by smooth maps 
$\Lambda\to \mathbb{C}^{2\times 2}$, 
$\lambda \mapsto A(\lambda)$ and 
$\Lambda\to \mathbb{C}^{2\times 2}_S$,  
$\lambda \mapsto B(\lambda)$, in which $\Lambda\subset \mathbb{R}^{k}$ is a parameter set, and we define
the bundle of $(A_0,B_0)=(A(\lambda_0),B(\lambda_0))$ for $\lambda_0\in \Lambda$ with respect to the action $\Psi$ in (\ref{aAB}) as:
\vspace{-1mm}
\begin{equation}\label{bun0}
\Bun_{\Psi}(A_0,B_0):=\bigcup_{\lambda\in\Lambda}\Orb_{\Psi}(A(\lambda),B(\lambda)). 
\end{equation}
\vspace{-1mm}
Elements of a bundle must behave similarly under small perturbations (Sec. \ref{Sec3}).

To simplify the notation, $a\oplus d$ denotes the diagonal matrix with $a$, $d$ on the diagonal, while the $2\times 2$ identity-matrix and the $2\times 2$ zero-matrix are $I_2$ and $0_2$. 

\vspace{-1mm}

\begin{lemma}\label{lemalist}
Bundles of the action (\ref{aAB}), represented by pairs 
$(A,B)$ given in Table \ref{table2}, are immersed submanifolds in $\mathbb{C}^{2\times 2}\times \mathbb{C}^{2\times 2}_S$ with dimensions noted in the first column.
%
\end{lemma}

\vspace{-4mm}

\begin{table}[H]
\caption{Bundles of the action (\ref{aAB}). Here $0<\tau <1$, $0<\theta <\pi$,\,\, $a,b,d>0$,\,\, $\zeta \in \mathbb{C}$,\,\, $\varphi\in \mathbb{R}$, $\zeta^{*}\in \mathbb{C}^{*}$ are the parameters.}
\begin{tabular}{p{4mm}|| c| c  ||c | c || c | c ||c | c}
\small{$\dim$} & $A$ & $B$ & $A$ & $B$ & $A$ & $B$ & $A$ & $B$\\
\hline
14
&
\multirow{ 12}{*}{
\small{$1\oplus e^{i\theta}$}
}
&
$
\begin{bsmallmatrix}
a & \zeta^{*} \\
\zeta^{*}  & d
\end{bsmallmatrix}
$, 
&
\multirow{ 12}{*}{
$
\begin{bsmallmatrix}
0 & 1 \\
\tau & 0
\end{bsmallmatrix}
$
}
&
$
\begin{bsmallmatrix}
e^{i\varphi} & b \\
b & \zeta
\end{bsmallmatrix}
$,
&
\multirow{ 12}{*}{
$
\begin{bsmallmatrix}
0 & 1 \\
1 & i
\end{bsmallmatrix}
$
}
&
&
\multirow{ 12}{*}{
$
\begin{bsmallmatrix}
0 & 1 \\
0 & 0
\end{bsmallmatrix}
$
}
&
\\
%
& &
$\scriptscriptstyle{-\zeta^{*}\sim \zeta^{*}}$   &
&$\scriptscriptstyle{\varphi+\pi\sim \varphi}$
&&&&\\
\cline{1-1}\cline{3-3}\cline{5-5}\cline{7-7}\cline{9-9}

12 &
&
$
\begin{bsmallmatrix}
0 & b \\
b & d
\end{bsmallmatrix}
$
&
&
$
\begin{bsmallmatrix}
0 & b \\
b & e^{i\varphi}
\end{bsmallmatrix}
$,
&
&
$
a\oplus \zeta
$
&
&
$
\begin{bsmallmatrix}
\zeta^{*} & b \\
b & 1
\end{bsmallmatrix}
$
\\
%
%
&
&
$
\begin{bsmallmatrix}
a & b \\
b & 0
\end{bsmallmatrix}
$
&
&
 $\scriptscriptstyle{\varphi+\pi\sim \varphi}$
&
&
&
&
\\

%
%
&& $a\oplus d$ && $
1\oplus \zeta
$  &&&&\\
\cline{1-1}\cline{3-3}\cline{5-5}\cline{7-7}\cline{9-9}

10 &
&
$
\begin{bsmallmatrix}
0 & b \\
b & 0
\end{bsmallmatrix}
$
&
&
$
0\oplus 1
$
&
&
$
\begin{bsmallmatrix}
0 & b \\
b & 0
\end{bsmallmatrix}
$
&
&
$
\begin{bsmallmatrix}
0 & b \\
b & 1
\end{bsmallmatrix}
$
\\

&
&
$
a\oplus 0
$
&
&
$
\begin{bsmallmatrix}
0 & b \\
b & 0
\end{bsmallmatrix}
$
&
&
&
&
$
a\oplus 1
$
\\

&
&
$
0\oplus d
$
&
&
&
&
&
&
$
\begin{bsmallmatrix}
1 & b \\
b & 0
\end{bsmallmatrix}
$
\\ 
\cline{1-1}\cline{3-3}\cline{5-5}\cline{7-7}\cline{9-9}

9 &&&&&
&
$
0\oplus d
$
&&\\
\cline{1-1}\cline{3-3}\cline{5-5}\cline{7-7}\cline{9-9}

8 &
&
$0_2$
&&
$0_2$
&
&
&
&
$
\begin{bsmallmatrix}
0 & b \\
b & 0
\end{bsmallmatrix}
$\\

 &&&&&&&& 
$
1\oplus 0
$\\
 
  &&&&&&&&
  $
0\oplus 1
$\\
\cline{1-1}\cline{3-3}\cline{5-5}\cline{7-7}\cline{9-9}

7  &&&&&&
$0_2$
&&\\
\cline{1-1}\cline{3-3}\cline{5-5}\cline{7-7}\cline{9-9}
6 &&&&&&&&$0_2$\\

%
%
\hline

$11$ & \multirow{8}{*}{$I_2$} & $a\oplus d$, $\scriptscriptstyle{a< d}$ &  \multirow{6}{*}{\small{$1\oplus -1$}} & $a\oplus d$, $\scriptscriptstyle{a< d}$ & 
\multirow{6}{*}{
$
\begin{bsmallmatrix}
0 & 1 \\
1 & 0
\end{bsmallmatrix}
$
}
& $1\oplus de^{i\theta}$ & 
\multirow{8}{*}{\small{$1\oplus 0
$}}
& \\
\cline{1-1}\cline{3-3}\cline{5-5}\cline{7-7}\cline{9-9}

10  &  &  &   &  &  &
 $
\begin{bsmallmatrix}
0 & b \\
b & 1
\end{bsmallmatrix}
$  
  & & $a \oplus 1$ \\
\cline{1-1}\cline{3-3}\cline{5-5}\cline{7-7}\cline{9-9}

9 & &  $ d I_2$ &  & $dI_2$ & &  & & \\

 &  &  $0\oplus d$  & &
 $
 \begin{bsmallmatrix}
 0 & b \\
 b & 0
 \end{bsmallmatrix}
$
 &  & & &
\\

&&&& $0\oplus d$ &&&&\\
\cline{1-1}\cline{3-3}\cline{5-5}\cline{7-7}\cline{9-9}

8 &&&&  &&
$1 \oplus 0 $
&&$0\oplus 1$\\

&&&&  &&&&$
\begin{bsmallmatrix}
0 & 1 \\
1 & 0
\end{bsmallmatrix}
$\\

\cline{1-1}\cline{3-7}\cline{9-9}

6 & & & & & \multirow{4}{*}{$0_2$} & $I_2$ & & $a\oplus 0$\\

\cline{1-1}\cline{3-3}\cline{5-5}\cline{7-7}\cline{9-9}
5 & & $0_2$ & & $0_2$ & &&  & \\
\cline{1-5}\cline{7-7}\cline{9-9}
4 &   \multicolumn{2}{c}{} &   \multicolumn{2}{c||}{} & & $1\oplus 0$ &  & $0_2$\\
\cline{7-9}
0 &   \multicolumn{2}{c}{} &   \multicolumn{2}{c||}{} & & $0_2$ &  \multicolumn{2}{c}{} \\
\end{tabular}
 \label{table2}
\end{table}
%
%
%


Note that we arranged orbits $\Orb_{\Psi}(1\oplus \sigma,d_0\oplus d)$ for $\sigma\in \{1,-1\}$, $d>0$, $d_0\in \{0,d\}$ into bundles $\Bun_{\Psi}(1\oplus \sigma,0\oplus d)=\cup_{d>0}\Orb_{\Psi}(1\oplus \sigma,0\oplus d)$ and $\Bun_{\Psi}(1\oplus \sigma,dI_2)=\cup_{d>0}\Orb_{\Psi}(1\oplus \sigma, d I_2)$, $\sigma\in \{1,-1\}$. 
Next, 
$\Orb_{\Psi}\big(\begin{bsmallmatrix}
0 & 1 \\
0  & 0
\end{bsmallmatrix},
\begin{bsmallmatrix}
\zeta & b \\
b  & 1
\end{bsmallmatrix}\big)
$ for $\zeta\in \mathbb{C}$, 
$b>0$
are split into 
bundles with representatives 
$
\big(\begin{bsmallmatrix}
0 & 1 \\
0  & 0
\end{bsmallmatrix},
\begin{bsmallmatrix}
\zeta^{*} & b \\
b  & 1
\end{bsmallmatrix}\big)
$  and 
$
\big(\begin{bsmallmatrix}
0 & 1 \\
0  & 0
\end{bsmallmatrix},
\begin{bsmallmatrix}
0 & b \\
b  & 1
\end{bsmallmatrix}\big)
$ for $\zeta^{*}\in \mathbb{C}^{*},b>0$.

\vspace{-1mm}

\begin{proof}[Sketch of the proof of Lemma \ref{lemalist}]
%
Fix $(A_0,B_0)\in \mathbb{C}^{2\times 2}\times \mathbb{C}_S^{2\times 2}$ from Table \ref{table2}
and define
\begin{equation}\label{actionL}
\Psi_{\Lambda}\colon S^1\times GL_2(\mathbb{C})\times \Lambda\to \mathbb{C}^{2\times 2}\times \mathbb{C}^{2\times 2}_S, \quad (c,P,\lambda)\mapsto \Psi\big(c,P,A(\lambda),B(\lambda)\big),
\end{equation}
where $\Psi_{\Lambda}(1,I_2,\lambda_0)=(A_0,B_0)$. For every $g\in S^1\times GL_2(\mathbb{C})$ the maps $\Phi^g \colon (A,B) \mapsto \Phi(g,(A,B))$ and $R_g\colon h\mapsto hg$ are automorphisms of $\mathbb{C}^{2\times 2}\times \mathbb{C}^{2\times 2}_S$ and $S^1\times GL_2(\mathbb{C})$, respectively, and we have $\Psi^g\circ \Psi_{\Lambda}=\Psi_{\Lambda} \circ(L_g\times \id_{\Lambda})$. Thus the rank of $d \Psi_{\Lambda}$ does not depend on $\lambda\in \Lambda$, $g\in S^1\times GL_2(\mathbb{C})$ 
and by the constant rank theorem (e.g. \cite[Theorem IV.5.8]{Boot}) the bundle $\Bun_{\Psi}(A_0,B_0)\subset \mathbb{C}^{2\times 2}\times \mathbb{C}_S^{2\times 2}$ is an immersed manifold. 

In a similar manner as tangent spaces of orbits in \cite[Lemma 2.2]{TS2} are computed, tangent spaces of bundles are obtained. 
We choose a path in $S^{1}\times GL_2(\mathbb{C})$: 
\[
\gamma\colon (-\delta,\delta)\to S^{1}\times GL_2(\mathbb{C}), \quad \gamma(t)=(e^{i\alpha t},I+tX), \qquad \alpha\in \mathbb{R}, X\in \mathbb{C}^{2\times 2}, \delta>0,
\]
and calculate:
\begin{align*}
&\frac{d}{dt}\Big|_{t=0}e^{i\alpha t}\big((I+tX)^*A(t)(I+tX)\big)
                                      =i\alpha A_0+\tfrac{d}{dt}\big|_{t=0}A(t)+  (X^*A_0+A_0X),\\
&\frac{d}{dt}\Big|_{t=0}\big((I+tX)^TB(t)(I+tX)\big) 
                                        =\tfrac{d}{dt}\big|_{t=0}B(t)+(X^TB_0+B_0X).
\end{align*}
Writing $X=\sum_{j,k=1}^2(x_{jk}+iy_{jk})E_{jk}$, where $E_{jk}$ is the elementary matrix with one in the $j$-th row and $k$-th column and zeros otherwise, we deduce that
\vspace{-2mm}
\begin{align*}
X^*A_0+A_0X &=  \sum_{j,k=1}^2(x_{jk}-iy_{jk})E_{kj}A_0+\sum_{j,k=1}^2(x_{jk}+iy_{jk})A_0E_{jk}\\
                             &=\sum_{j,k=1}^2x_{jk}(E_{kj}A_0+A_0E_{jk})+\sum_{j,k=1}^2y_{jk}i(-E_{kj}A_0+A_0E_{jk}),\\
\tfrac{d}{dt}\big|_{t=0}A(t)&=\beta_{21}E_{21}+\beta_{22}E_{22}  ,                           
\end{align*}
%
\vspace{-1mm}
\[
\beta_{22}=
\left\{
\begin{array}{ll}
\beta ie^{i\theta}, & A=1\oplus e^{i\theta}, 0<\theta<\pi\\
0, & \textrm{otherwise}
\end{array}
\right.,\quad
\beta_{21}=
\left\{
\begin{array}{ll}
\beta, & A=
\begin{bsmallmatrix}
0 & 1\\
\tau & 0
\end{bsmallmatrix},  0<\tau<1\\
0, & \textrm{otherwise}
\end{array}
\right.,\quad \beta\in \mathbb{R}.
\]
In a similar fashion we conclude that
\vspace{-1mm}
\begin{align*}
&X^TB_0+B_0X  
                             = \sum_{j,k=1}^2x_{jk}(E_{kj}B_0+B_0E_{jk})
                               +\sum_{j,k=1}^2y_{jk}i(E_{kj}B_0+B_0E_{jk}),\\
&\tfrac{d}{dt}\big|_{t=0}B(t)=\sum_{j,k=1}^2\gamma_{jk}E_{jk},\qquad
\gamma_{jk}=\left\{
\begin{array}{ll}
 z_{jk}, & B_{jk}(t)=(B_0)_{jk}+z_{jk}t, z_{jk}\in \mathbb{C}\\
i(B_0)_{jk}\omega_{jk}, & B_{jk}(t)=(B_0)_{jk}e^{i\omega_{jk}t}, \omega_{jk}\in \mathbb{R}\\
0, & \textrm{otherwise}
\end{array}
\right..                                                        
\end{align*}
%
Note that if $A_{jk}(t)$ (or $B_{jk}(t)$) is constant, then $\beta_{jk}=0$ ($\gamma_{jk}=0$).

In view of the identification 
$
\mathbb{R}^{8} \times \mathbb{R}^{6}\approx\mathbb{C}^{2\times 2}\times \mathbb{C}^{2\times 2}_S$
we denote ($j,k\in\{1,2\}$):
%
\vspace{-1mm}
\begin{align*}
&\widetilde{u}_{jk}\approx (0,E_{jk}), \qquad \qquad \widetilde{v}_{jk}\approx (0,iE_{jk}), \quad\qquad j\leq k \\
&u_{jk}\approx (E_{kj}A_0+  A_0E_{jk},E_{kj}B_0+B_0E_{jk}), \quad  v_{jk}\approx i(-E_{kj}A_0+A_0E_{jk},E_{kj}B_0+B_0E_{jk}),
\end{align*}
\begin{align*}
& w_1\approx \left\{
\begin{array}{ll}
(ie^{i\theta}E_{22},0), & A=1\oplus e^{i\theta}, 0<\theta<\pi\\
(E_{21},0), & A=
\begin{bsmallmatrix}
0 & 1\\
\tau & 0
\end{bsmallmatrix},  0<\tau<1\\
0, & \textrm{otherwise}
\end{array}
\right.,\qquad 
\begin{array}{l}
w_2\approx (iA,0),\\
w_{3} \approx(0,i(B_0)_{11}E_{11}),\\
w_{4} \approx (0,i(B_0)_{22}E_{22}).
\end{array}
\end{align*}
%
The tangent space of $\Bun_{\Psi}(A_0,B_0)$ can be seen as a linear space spanned by vectors $\{w_1,w_2\}\cup \{ u_{jk},v_{jk}\}_{j,k\in\{1,2\} }$ and a subset of vectors $\{w_{3},w_{4}\}\cup\{ \widetilde{u}_{jk},\widetilde{v}_{jk}\}_{j,k\in\{1,2\},j\leq k}$.
If $B_{jj}(t)=(B_0)_{jj}(\lambda_0)e^{i\omega_{jj}t}$ for $j\in \{1,2\}$, then $w_{j+2}$ is in the span, while for $B_{jk}(t)=(B_0)_{jk}+z_{jk}t$, $z_{jk}\neq 0$ vectors $\widetilde{u}_{jk},\widetilde{v}_{jk}$ are in the span.
It is straightforward 
to compute the dimensions; see \cite[Lemma 2.2]{TS2} for the details in the case of orbits. 
\end{proof}

%
%
%
%

\section{Change of the normal form under small perturbations}\label{Sec3}

In this section we study how small deformations of a pair of one arbitrary and one symmetric matrix can change its bundle under the action (\ref{aAB}). 
For the sake of clarity the notion \emph{closure graph} for bundles for an action is introduced; compare it with the closure graph for orbits in \cite{TS2}. Given an action $\Phi$, \emph{vertices} of its closure graph are pairwise disjoint bundles of orbits with respect to $\Phi$, and there is a \emph{path} from a vertex $\widetilde{\mathcal{V}}$ to a vertex $\mathcal{V}$ precisely when $\widetilde{\mathcal{V}}$ lies in the closure of $\mathcal{V}$. The path from $\widetilde{\mathcal{V}}$ to $\mathcal{V}$ is denoted by $\widetilde{\mathcal{V}}\to \mathcal{V}$. 
To simplify the notation we usually write $\widetilde{V}\to V$ for $\widetilde{V}\in \widetilde{\mathcal{V}}$, $V\in \mathcal{V}$ (instead of $\widetilde{\mathcal{V}}\to \mathcal{V}$).
We also require that if $\widetilde{V}\in\mathcal{\widetilde{V}}$ (hence $\Orb_{\Phi}(\widetilde{V})$) is contained in the closure of $\mathcal{V}$, 
then whole bundle $\widetilde{\mathcal{V}}$ must lie in the closure of $\mathcal{V}$; it does not hold in general. 
Closure graphs are reflexive and transitive.

When $\widetilde{V}\not\to V$ it is useful to know the distance from $\widetilde{V}$ to the bundle $\mathcal{V}\ni V$. It suffices to consider the distance from the normal form
of $\widetilde{V}$ (see e.g. \cite[Remark 3.2]{TS2}). 
We use the max norm $\|X\|=\max_{j,k\in \{1, 2\}}|x_{j,k}|$, $X=[x_{j,k}]_{j,k=1}^{2}\in \mathbb{C}^{2\times 2}$ to measure the distance between matrices. 

To emphasize the difference between the closure graphs for orbits and bundles we take look at the action of similarity on $\mathbb{C}^{2\times 2}$. Given $\lambda,\mu\in \mathbb{C}$ with $\lambda\neq \mu$ we have $\lambda\oplus \lambda \not\to \lambda\oplus \mu$ in the closure graph for orbits (eigenvalues depend continuously on the entries of the matrix), but $\lambda\oplus \lambda\to \lambda\oplus \mu$ in the closure graph for bundles (the bundle of $\lambda\oplus \mu$ is dense in $\mathbb{C}^{2\times 2}$). For a comprehensive theory on closure hierarchy 
of matrices under similarity we refer to \cite{EEK} and \cite{EJK}.

The action (\ref{aAB}) is closely related with the following two actions:
%
%
%
\begin{align}
\label{actionpsi1}
&\Psi_1\colon (c,P,A)\big)\mapsto cP^{*}AP, \quad P\in GL_2(\mathbb{C}), \,c\in S^{1},\,A\in \mathbb{C}^{2\times 2}\\
%
\label{actionpsi2}
&\Psi_2\colon (P,B)\mapsto P^TBP, \quad P\in GL_2(\mathbb{C}),\, B\in \mathbb{C}_S^{2\times 2}. 
\end{align}
Bundles under these actions are defined the same way as bundles for $\Psi$ in (\ref{bun0}).
The closure graph for (\ref{actionpsi2}) with trivial bundles (orbits) is simple (see \cite[Lemma 3.2]{TS2}); we add a few necessary conditions on its paths and prove them in Sec. \ref{proofL}. For closure graphs of all $2\times 2$ or $3\times 3$ matrices see \cite{FKS1}.

\begin{lemma}\label{lemapsi2}
The closure graph for the action (\ref{actionpsi2}) 
is 
\begin{equation}\label{closureTgraph}
0_2 \to 1\oplus 0\to I_2,
\end{equation}
in which $1\oplus 0 $ and $I_2$ 
correspond to bundles of symmetric matrices of rank $1$ and $2$.  
Furthermore, let $B=\begin{bsmallmatrix}
a & b\\
b & b
\end{bsmallmatrix}\in \mathbb{C}^{2\times 2}_S$, $\widetilde{B}=\begin{bsmallmatrix}
\widetilde{a} & \widetilde{b} \\
\widetilde{b}  & \widetilde{d} 
\end{bsmallmatrix}\in \mathbb{C}^{2\times 2}_S$,
$P=\begin{bsmallmatrix}
x & y\\
u & v
\end{bsmallmatrix}\in GL_2(\mathbb{C})$ and $F=\begin{bsmallmatrix}
\epsilon_1 & \epsilon_2\\
\epsilon_2 & \epsilon_4
\end{bsmallmatrix}\in \mathbb{C}^{2\times 2}_S$ be such that $P^{T} AP=\widetilde{B}+F$. 
Then the following statements hold: 
\begin{enumerate}
\item \label{lemapsi2a} If $\widetilde{B}$, $B$ are normal forms in (\ref{closureTgraph}) and such that $\widetilde{B}\not \to B$, then $\|F\|\geq 1$.
\item \label{lemapsi22} If $\widetilde{B}\to B$, then there exist 
\small 
$\epsilon_2',\epsilon_2''\in \mathbb{C}$, 
$|\epsilon_2'|,|\epsilon_2''|\leq \left\{\begin{array}{ll}
\tfrac{\|F\|(4\|\widetilde{B}\|+2+|\det \widetilde{B}|)}{|\det \widetilde{B}|}, & \det \widetilde{B}\neq 0\\
\sqrt{\|F\|(4\|\widetilde{B}\|+3)}, & \det \widetilde{B}= 0
\end{array}\right.$,
\normalsize
so that equations listed in the third column (and in the line corresponding to $B$) of Table \ref{table4} are valid. 
\end{enumerate}
\end{lemma}
\vspace{-3mm}
\small
\begin{table}[hbt!]
\caption{
Necessary conditions on $B$ and $P$ (given that $P^{T} BP=\widetilde{B}+F$).
}
\begin{tabular}{|c|c|l|}
\hline
 & $B$ &                \\
\hline
\refstepcounter{rownumber}
\label{p3bd}  D\ref{p3bd}  & 
$\begin{bsmallmatrix}
0 & b \\
b &  d
\end{bsmallmatrix}$
& 
$u(i(-1)^{l}\sqrt{\det\widetilde{B}}+\widetilde{b}+\epsilon_2')=v(\widetilde{a}+\epsilon_4)$, \,\, $v(-i(-1)^{l}\sqrt{\det\widetilde{B}}+\widetilde{b}+\epsilon_2'')=u(\widetilde{d}+\epsilon_4)$ \\
		\hline
\refstepcounter{rownumber}
\label{p3ba}  D\ref{p3ba}  & 
$\begin{bsmallmatrix}
a & b \\
b &  0
\end{bsmallmatrix}$
& 
$y(i(-1)^{l}\sqrt{\det\widetilde{B}}+\widetilde{b}+\epsilon_2')=x(\widetilde{d}+\epsilon_4)$, \,\, $x(-i(-1)^{l}\sqrt{\det\widetilde{B}}+\widetilde{b}+\epsilon_2'')=y(\widetilde{a}+\epsilon_1)$ \\
		\hline	
\refstepcounter{rownumber}
\label{p4}  D\ref{p4}  & 
$\begin{bsmallmatrix}
0 & b \\
b &  0
\end{bsmallmatrix}$
& 
$2bvx=i(-1)^{l}\sqrt{\det\widetilde{B}}+\widetilde{b}+\epsilon_2'$, \,\, $2buy=(-i(-1)^{l}\sqrt{\det\widetilde{B}}+\widetilde{b}+\epsilon_2''$ \\
		\hline
\refstepcounter{rownumber}
\label{p5d}  D\ref{p5d}  & 
$0\oplus d$
& 
$u(\widetilde{b}+\epsilon_2)=v(\widetilde{a}+\epsilon_4)$, \,\, $v(\widetilde{b}+\epsilon_2)=u(\widetilde{d}+\epsilon_4)$ \\
		\hline		
\refstepcounter{rownumber}
\label{p5a}  D\ref{p5a}  & 
$a\oplus 0$
& 
$y(\widetilde{b}+\epsilon_2)=x(\widetilde{d}+\epsilon_4)$, \,\, $x(\widetilde{b}+\epsilon_2)=y(\widetilde{a}+\epsilon_1)$  \\
		\hline			
	\end{tabular}
 \label{table4}
\end{table}	
\normalsize

%
By making a more detailed analysis as in \cite[Lemma 3.4]{TS2} 
(see also
\cite[Theorem 2.2]{FKS2})
we get the closure graph for bundles under the action (\ref{actionpsi1}) along with necessary conditions related to its paths; 
the proof is given in Sec. \ref{proofL}. 

\begin{lemma}\label{lemapsi1}
The closure graph for bundles under the action (\ref{actionpsi1}) is drawn in Figure \ref{cgraph1}. 
It contains six vertices corresponding to bundles (orbits) with normal forms $0_2$, $1\oplus 0$, $I_2$, $1\oplus -1$, $\begin{bsmallmatrix}
0 & 1\\
0 & 0
\end{bsmallmatrix}$, $\begin{bsmallmatrix}
0 & 1\\
1 & i
\end{bsmallmatrix}$, and two vertices for bundles with normal forms 
$1\oplus e^{i\theta}$ for 
$\theta\in (0,\pi)$ 
and  
$\begin{bsmallmatrix}
0 & 1\\
\tau & 0
\end{bsmallmatrix}$ for $\tau\in (0,1)$.

%
%
\vspace{-3mm}
\small
\begin{figure}[hbt!]
\begin{equation*}
\begin{tikzcd}[row sep=1.1em]
1\oplus e^{i\theta}   &      &   
\begin{bsmallmatrix}
0 & 1\\
\tau & 0
\end{bsmallmatrix}       & \qquad 8 \\
 & 
\begin{bsmallmatrix}
0 & 1\\
1 & i
\end{bsmallmatrix}
\arrow[ur, ""] \arrow[ul, ""] & 
 & \qquad 7\\
&  &
\begin{bsmallmatrix}
0 & 1\\
0 & 0
\end{bsmallmatrix}
\arrow[uu, ""]
& \qquad 6 \\
I_2 \arrow[uuu, ""] & 
1\oplus -1\approx
\begin{bsmallmatrix}
0 & 1\\
1 & 0
\end{bsmallmatrix}
\arrow[uu, ""] & 
& \qquad 5 \\
& 
1\oplus 0
\arrow[u, ""] \arrow[ul, ""] \arrow[uur,bend right= 10, ""]  & & \qquad 4\\
 & 
0_2
\arrow[u, ""]  & & \qquad 0
\end{tikzcd}
\end{equation*}
\caption{The closure graph for the action (\ref{actionpsi1}).
} 
\label{cgraph1}
\end{figure}
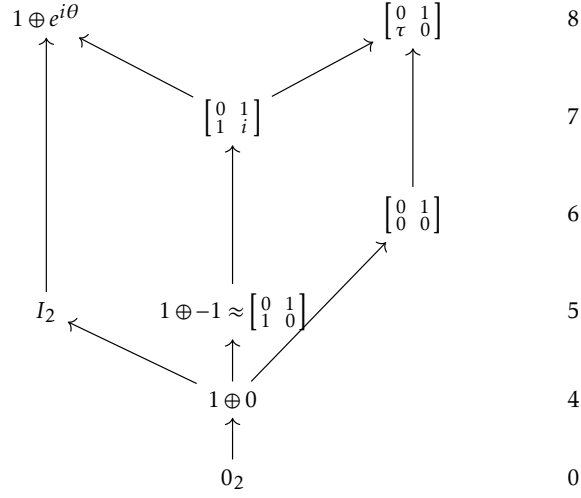
\normalsize
Furthermore, let $\widetilde{A}$, $A$ be normal forms in Figure \ref{cgraph1}, and let $E=cP^{*} AP-\widetilde{A}$ for some $c\in S^{1}$, $P=\begin{bsmallmatrix}
x & y\\
u & v
\end{bsmallmatrix}\in GL_2(\mathbb{C})$, $E\in \mathbb{C}^{2\times 2}$ with $\|E\|<1$.
Then the following statements hold:
\begin{enumerate}
\item \label{lemapsi1a} If $\widetilde{A}\not\to A$, then there exists a constant $\mu>0$ such that $\|E\|\geq \mu$.
\item \label{lemapsi11} If $\widetilde{A}\to A$, then there is a constant $\nu>0$ such that the moduli of expressions
listed in the fourth column (and in the line corresponding to $\widetilde{A}$, $A$) of Table \ref{table3} are bounded by $\nu\sqrt{\|E\|}$. (If $\widetilde{A}\in GL_2(\mathbb{C})$ then also $\|E\|\leq \tfrac{|\det \widetilde{A}|}{8\|\widetilde{A}\|+4}$ is assumed.) 
\end{enumerate}
\end{lemma}
\setcounter{rownumber}{0}
\small
\begin{table}[hbt!]
\caption{Necessary conditions on $A,P,c$ (given that $cP^{*} AP=\widetilde{A}+E$).
}
\begin{tabular}{|c|c|c|l|l|}
\hline
 &	$\widetilde{A}$ & $A$ &              &    \\
		\hline
\refstepcounter{rownumber}
\label{r32}  C\ref{r32}  & 
$\alpha \oplus 0$
&
$1 \oplus e^{i\theta}$
& 
 $|x|^2+e^{i\theta}|u|^2-c^{-1}\alpha$, $|y|^2+e^{i\theta}|v|^2$ & $\scriptstyle{\alpha\in\{0,1\}}$, $\scriptstyle{0< \theta<\pi}$\\
& & & $(\sin\theta) \overline{u}v$, $(\sin\theta) \overline{x}y$, $\overline{x}y+(\cos \theta) \overline{u}v$  &  \\
\hline
\refstepcounter{rownumber}
\label{r33}  C\ref{r33}  & 
$
\begin{bsmallmatrix}
0 & 1 \\
1 &  \omega
\end{bsmallmatrix}
$
&
$1 \oplus e^{i\theta}$
& 
 $|x|^2-|u|^2$, $|y|^2-|v|^2$, $\overline{x}y-\overline{u}v-(-1)^{k}$, $\sin\theta$ & $\scriptstyle{k\in \mathbb{Z}}$; $\scriptstyle{0< \theta<\pi}$,  $\scriptstyle{\omega\in \{0,i\}}$\\
  &   &   & & $\scriptstyle{\textrm{or } \theta=\pi,\omega=0}$\\
		\cline{4-5}
 &  & & 
$(\sin \theta) |v|^2-1$, $(\sin \theta)|u|^{2}$, $k=0$ &   $\omega=i$\\
		\hline		
\refstepcounter{rownumber}
\label{r4}  C\ref{r4} &
$\alpha\oplus 0$
&
$
\begin{bsmallmatrix}
0 & 1 \\
\tau &  0
\end{bsmallmatrix}
$
&
$\scriptstyle{(1+\tau)\Rea(\overline{x}u)+i(1-\tau)\Ima(\overline{x}u)-\frac{\alpha}{c}}$,$(1-\tau)\overline{x}v$ & $\scriptstyle{0\leq \tau \leq 1,\alpha\in\{0,1\}}$,\\ 
&&& 
$\Rea(\overline{y}v)$, $(1-\tau)\Ima(\overline{y}v)$,$(1-\tau)\overline{u}y$, $\overline{x}v+\overline{u}y$  & 
\\
\cline{4-5}
&&& $c^{-1}-(-1)^{k}$, $2\Rea(\overline{x}u)- (-1)^{k}\alpha$
& $\scriptstyle{\tau=\alpha=1,\|E\|\leq \frac{1}{2}}$
\\
\hline
\refstepcounter{rownumber}
\label{r44}  C\ref{r44} &
$
\begin{bsmallmatrix}
0 & 1 \\
1 &  \omega
\end{bsmallmatrix}
$
&
$
\begin{bsmallmatrix}
0 & 1 \\
\tau &  0
\end{bsmallmatrix}
$
&
$\Rea(\overline{x}u)$, $(1-\tau)\Ima(\overline{x}u)$, $1-\tau$,  & $\scriptstyle{0< \tau\leq 1,\omega\in\{0,i\}}$\\
&&& $(1+\tau)\Rea(\overline{y}v)+i(1-\tau)\Ima(\overline{y}v)-(-1)^{k}\omega$ & $\scriptstyle{k\in \mathbb{Z}}$ \\
&&&  $\overline{x}v+\overline{u}y-(-1)^{k}$   &\\
\hline
\refstepcounter{rownumber}
\label{r7}  C\ref{r7}  &
$1\oplus -1$
&
$
\begin{bsmallmatrix}
0 & 1 \\
\tau & 0
\end{bsmallmatrix}
$ 
&        $2\Rea(\overline{y}v)-(-1)^{k}$, $2\Rea(\overline{x}u)+(-1)^{k}$, $1-\tau$ & $\scriptstyle{0<\tau\leq 1, k\in \mathbb{Z}}$
\\
&    &	    & $(1-\tau)\Ima(\overline{y}v)$, $(1-\tau)\Ima(\overline{x}u)$, $\overline{x}v+\overline{u}y$   &  
\\
\hline	
\refstepcounter{rownumber}
\label{r10}  C\ref{r10}  &	
$\alpha\oplus 0$
&
$
\begin{bsmallmatrix}
0 & 1 \\
1 & i
\end{bsmallmatrix}
$ 
&
$\Rea(\overline{y}u)$, $2\Rea(\overline{x}u)+i|u|^{2}-\frac{\alpha}{c}$, & $\scriptstyle{\alpha\in \{0,1\}}$   \\
& 	&   &  $\overline{x}v+\overline{u}y$, $\overline{u}v$, $v^{2}$  &   \\
\hline
 \refstepcounter{rownumber}
\label{r1}  C\ref{r1}  
&
$1 \oplus e^{i\widetilde{\theta}} $
&
$1 \oplus e^{i\theta} $
&
$u^{2},y^{2}$, $|x|^{2}-1$, $|v|^{2}-1$ 
  &   $\scriptstyle{0<\theta<\pi}$, $\scriptstyle{0<\widetilde{\theta}<\pi}$
\\
		\hline	
\refstepcounter{rownumber}
\label{r5}  C\ref{r5} &	
$
\begin{bsmallmatrix}
\alpha & \beta \\
\beta & \omega 
\end{bsmallmatrix}
$ &
$
\begin{bsmallmatrix}
0 & 1 \\
1 & i 
\end{bsmallmatrix}
$ 
& $2\Rea(\overline{x}u)-(-1)^{k}\alpha$, $2\Rea(\overline{y}v)-(-1)^k\Rea (\omega)$ &  $\scriptstyle{k\in \mathbb{Z}}$;$\scriptstyle{\beta=0,-\omega=\alpha\in \{0,1\}}$ \\
&   &  	 & 
$\overline{x}v+\overline{u}y-(-1)^k\beta$, $u^{2}$ , $|v|^2-(-1)^k\Ima (\omega)$ &  $\scriptstyle{\textrm{or } \beta=1,\alpha=0,\omega\in\{0,i\}}$
   \\
\hline
\refstepcounter{rownumber}
\label{r6}  C\ref{r6}  &
$
\begin{bsmallmatrix}
0 & 1 \\
\widetilde{\tau} & 0
\end{bsmallmatrix}
$
&
$
\begin{bsmallmatrix}
0 & 1 \\
\tau & 0
\end{bsmallmatrix}
$  & $\overline{x}u$, \,$\overline{y}v$,\, $\overline{y}u$,\, $\overline{x}v-c^{-1}$ &   $\scriptstyle{0\leq\widetilde{\tau}< 1,0<\tau< 1}$\\
&&& & $\scriptstyle{\textrm{or }\tau=\widetilde{\tau}=0}$ \\
\cline{4-5}
  &   &  &   $c^{-1}-(-1)^{k}$ & $\scriptstyle{0<\widetilde{\tau},\tau< 1}$,\,\,$\scriptstyle{k\in \mathbb{Z}}$ \\   
 \hline
\refstepcounter{rownumber}
\label{r9}  C\ref{r9}  &	
$\alpha\oplus 0$
&
$1\oplus \sigma$
 &
 $|x|^2+\sigma|u|^2-c^{-1}\alpha$, $\overline{x}y+\sigma \overline{u} v$,  $|y|^2+\sigma |v|^2$   & $\scriptstyle{\sigma\in\{1,-1\}}$, $\scriptstyle{\alpha\in \{1,0\}}$
 \\
\hline
\refstepcounter{rownumber}
\label{r99}  C\ref{r99}  &	
$1\oplus \sigma$
&
$1\oplus e^{\theta}$
 &
 $|x|^2+\sigma|u|^2-(-1)^{k}$, $\overline{x}y+\sigma \overline{u} v$  & $\scriptstyle{\sigma\in\{1,-1\}}$, $\scriptstyle{k\in \mathbb{Z}}$;  $\scriptstyle{\sigma=e^{i\theta}}$ \\
&  &   	& $|y|^2+\sigma |v|^2-(-1)^{k}$     & $\scriptstyle{\textrm{or } 0< \theta < \pi , \|E\|\leq \frac{1}{392}}$   \\
\hline
\refstepcounter{rownumber}
\label{r11}  C\ref{r11}  
& 	
$\alpha\oplus 0$ 
&
$1\oplus 0$  
&   
$y^2$, $|x|^2-\alpha$  &  $\scriptstyle{\alpha\in\{0,1\}}$ \\
\hline
	\end{tabular}
	%
\label{table3}
\end{table}	
\normalsize
%
%
%


\begin{remark} 
For calculations of $\mu$, $\nu$ in Lemma \ref{lemapsi1} 
see the proof of the lemma.
\end{remark}

We are ready to state the main results of the paper. 
The proof is 
given in Sec. \ref{proofT}. 
%

%
%
%
\small
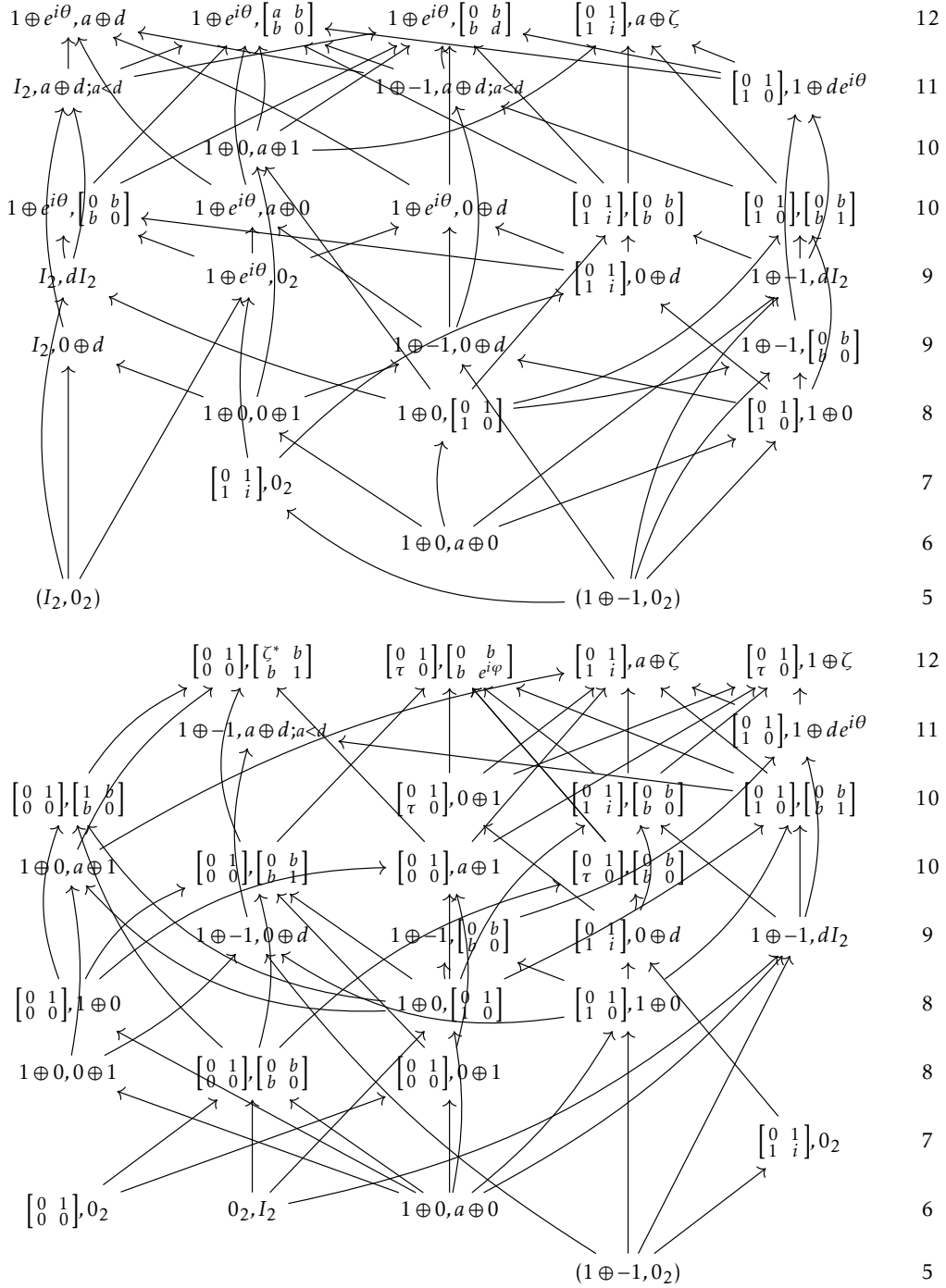
\begin{figure}
\centering
\begin{equation*}
\hspace*{-12mm}
\begin{tikzcd}[row sep=0.8em,column sep=small]
1\oplus e^{i\theta},a\oplus d
&
1\oplus e^{i\theta}, \begin{bsmallmatrix}
a & b\\
b & 0
\end{bsmallmatrix}
&
1\oplus e^{i\theta}, \begin{bsmallmatrix}
0 & b\\
b & d
\end{bsmallmatrix}
&
\begin{bsmallmatrix}
0 & 1\\
1 & i
\end{bsmallmatrix},a\oplus \zeta 
&& 12\\
I_2,a\oplus d; \scriptstyle{a<d} \arrow[ur, ""] \arrow[urr, ""] \arrow[u, ""]
&&
1\oplus -1,a\oplus d; \scriptstyle{a<d} \arrow[u,bend left= 15, ""] \arrow[ul, ""]\arrow[ull, ""]
&&
\begin{bsmallmatrix}
0 & 1\\
1 & 0
\end{bsmallmatrix},1\oplus de^{i\theta} \arrow[ul, ""] \arrow[ull, ""] \arrow[ulll, ""]
& 11
\\
& 
1\oplus 0,a\oplus 1 \arrow[uur, ""] \arrow[uu,bend right= 15, ""] \arrow[uurr,bend right= 20, ""] 
&
&
&
&
10\\
1\oplus e^{i\theta}, 
\begin{bsmallmatrix}
0 & b\\
b & 0
\end{bsmallmatrix} \arrow[uuur, ""] \arrow[uuurr, ""] 
& 
1\oplus e^{i\theta}, 
a\oplus 0 \arrow[uuu,bend left= 15, ""] \arrow[uuul,bend left= 15, ""]
& 
1\oplus e^{i\theta}, 
0\oplus d \arrow[uuu, ""] \arrow[uuull,bend right= 5, ""]
& 
\begin{bsmallmatrix}
0 & 1\\
1 & i
\end{bsmallmatrix}, 
\begin{bsmallmatrix}
0 & b\\
b & 0
\end{bsmallmatrix}  \arrow[uuu, ""]\arrow[uuul, ""]\arrow[uuull, ""]
& 
\begin{bsmallmatrix}
0 & 1\\
1 & 0
\end{bsmallmatrix} ,\begin{bsmallmatrix}
0 & b\\
b & 1
\end{bsmallmatrix} \arrow[uuul, ""]\arrow[uull, ""]
& 10\\
I_2, dI_2 \arrow[u,bend left= 10, ""]\arrow[uuu,bend right= 15, ""]
& 
1\oplus e^{i\theta}, 0_2 \arrow[u, ""]\arrow[ul, ""]\arrow[ur, ""]
& 
&
\begin{bsmallmatrix}
0 & 1\\
1 & i
\end{bsmallmatrix} ,0\oplus d \arrow[u, ""]\arrow[ul, ""] \arrow[ulll, ""]
&
1\oplus -1, dI_2 \arrow[ul, ""]\arrow[uuu,bend right= 25, ""]\arrow[u, ""]
&   9 \\
I_2,0\oplus d \arrow[uuuu,bend left= 15, ""]   
& 
&
1\oplus -1,0\oplus d \arrow[uul, ""] \arrow[uu, ""] \arrow[uuuu,bend right= 20, ""]
&
& 
1\oplus -1,
\begin{bsmallmatrix}
0 & b\\
b & 0
\end{bsmallmatrix} \arrow[uuuu,bend left= 10, ""]
& 9 \\
& 
1\oplus 0,0\oplus 1 \arrow[ul, ""] \arrow[ur, ""] \arrow[uuuu,bend right= 15, ""] 
& 
1\oplus 0,
\begin{bsmallmatrix}
0 & 1\\
1 & 0
\end{bsmallmatrix} 
\arrow[uull,bend left= 5, ""]\arrow[uuuul, ""]\arrow[uuur, ""]\arrow[uuurr,bend right= 20, ""]\arrow[urr,bend right= 5, ""]
& 
&  
\begin{bsmallmatrix}
0 & 1\\
1 & 0
\end{bsmallmatrix},1\oplus 0 \arrow[uuu,bend right= 25,""] \arrow[uul,""]\arrow[u,""]\arrow[ull,""]
& 8 \\
 & 
 \begin{bsmallmatrix}
 0 & 1\\
 1 & i
\end{bsmallmatrix}, 0_2 \arrow[uuurr,bend left= 15,""] \arrow[uuu,bend left= 10,""]
  & 
& & & 7\\
& &  1\oplus 0,a\oplus 0  \arrow[uul,""] \arrow[uurr,""] \arrow[uu,bend left= 15,""] \arrow[uuuurr,""]
& & & 6\\
(I_2,0_2) \arrow[uuuu,""] \arrow[uuuuur,""] \arrow[uuuuu,bend left= 15,""]& & & (1\oplus -1,0_2) \arrow[uuuur,bend left= 15,""]\arrow[uull,bend left= 20,""]\arrow[uuur,""] \arrow[uuuul,""] \arrow[uuuuur,bend left= 20,""] & & 5\\
\hspace{2mm}
%
%
%
&
\begin{bsmallmatrix}
0 & 1\\
0 & 0
\end{bsmallmatrix}, \begin{bsmallmatrix}
\zeta^{*} & b\\
b & 1
\end{bsmallmatrix}
&
\begin{bsmallmatrix}
0 & 1\\
\tau & 0
\end{bsmallmatrix}, \begin{bsmallmatrix}
0 & b\\
b & e^{i\varphi}
\end{bsmallmatrix}
&
\begin{bsmallmatrix}
0 & 1\\
1 & i
\end{bsmallmatrix},a\oplus \zeta 
&
\begin{bsmallmatrix}
0 & 1\\
\tau & 0
\end{bsmallmatrix},1\oplus \zeta 
& 12\\
& 1\oplus -1,a\oplus d; \scriptstyle{a<d} &
&&
\begin{bsmallmatrix}
0 & 1\\
1 & 0
\end{bsmallmatrix},1\oplus de^{i\theta} \arrow[ul, ""]  \arrow[u, ""]
& 11
\\
\begin{bsmallmatrix}
0 & 1\\
0 & 0
\end{bsmallmatrix}, \begin{bsmallmatrix}
1 & b\\
b & 0
\end{bsmallmatrix} \arrow[uur,bend left= 15, ""] 
& 
&
\begin{bsmallmatrix}
0 & 1\\
\tau & 0
\end{bsmallmatrix}, 0\oplus 1 \arrow[uu, ""] \arrow[uur, ""]  \arrow[uurr, ""] 
&
\begin{bsmallmatrix}
0 & 1\\
1 & i
\end{bsmallmatrix}, \begin{bsmallmatrix}
0 & b\\
b & 0
\end{bsmallmatrix} \arrow[uu, ""]  \arrow[uul, ""] \arrow[uur, ""] 
&
\begin{bsmallmatrix}
0 & 1\\
1 & 0
\end{bsmallmatrix}, \begin{bsmallmatrix}
0 & b\\
b & 1
\end{bsmallmatrix} \arrow[uull, ""]  \arrow[uul, ""] \arrow[ulll, ""] 
&
10\\
1\oplus 0,a\oplus 1
\arrow[uuur,bend left= 15, ""] \arrow[uuurrr,bend left= 10, ""]
& 
\begin{bsmallmatrix}
0 & 1\\
0 & 0
\end{bsmallmatrix}, 
\begin{bsmallmatrix}
0 & b\\
b & 1
\end{bsmallmatrix}  \arrow[uuur, ""]\arrow[uuu,bend left= 25, ""]
& 
\begin{bsmallmatrix}
0 & 1\\
0 & 0
\end{bsmallmatrix},  
a\oplus 1 \arrow[uuul, ""] \arrow[uuur, ""]  \arrow[uuurr, ""]
& 
\begin{bsmallmatrix}
0 & 1\\
\tau & 0
\end{bsmallmatrix}, 
\begin{bsmallmatrix}
0 & b\\
b & 0
\end{bsmallmatrix}  \arrow[uuul, ""]\arrow[uuul, ""]  
& 
& 10\\
& 
1\oplus -1,0\oplus d  \arrow[uuu,bend left= 15, ""]
&
1\oplus -1,
\begin{bsmallmatrix}
0 & b\\
b & 0
\end{bsmallmatrix}
\arrow[uuurr,bend right= 15, ""]
&
\begin{bsmallmatrix}
0 & 1\\
1 & i
\end{bsmallmatrix} ,0\oplus d \arrow[u,bend right= 15, ""]\arrow[uu,bend right= 25, ""]\arrow[uul, ""]
& 
1\oplus -1,dI_2 \arrow[uuu,bend right= 15, ""]\arrow[uu, ""]\arrow[uul, ""]
& 9 \\
\begin{bsmallmatrix}
0 & 1\\
0 & 0
\end{bsmallmatrix} ,1\oplus 0 \arrow[uuu,bend left= 20, ""]\arrow[uur,bend left= 20, ""]\arrow[uurr,bend left= 20, ""]
& 
&
1\oplus 0,
\begin{bsmallmatrix}
0 & 1\\
1 & 0
\end{bsmallmatrix}
\arrow[uull,bend left= 25, ""]\arrow[uuull,bend left= 25, ""]\arrow[uul, ""]\arrow[u,bend left= 10, ""]\arrow[uu, ""]\arrow[uuur,bend left= 18, ""]\arrow[uuurr,bend right= 5, ""]
&
\begin{bsmallmatrix}
0 & 1\\
1 & 0
\end{bsmallmatrix}, 
1\oplus 0 \arrow[uuur,bend right= 15, ""]\arrow[u, ""]\arrow[ul, ""]\arrow[ull,bend left= 21, ""]
& 
& 8 \\
1\oplus 0,0\oplus 1 
\arrow[uuu,bend right= 10, ""] \arrow[uur,bend right= 10, ""]
&
\begin{bsmallmatrix}
0 & 1\\
0 & 0
\end{bsmallmatrix}, 
\begin{bsmallmatrix}
0 & b\\
b & 0
\end{bsmallmatrix}  \arrow[uuuul,bend left= 15, ""]\arrow[uuu,bend right= 15, ""]\arrow[uuurr,bend left= 15, ""]
&
\begin{bsmallmatrix}
0 & 1\\
0 & 0
\end{bsmallmatrix}, 
0\oplus 1 \arrow[uuul, ""]\arrow[uuu,bend right= 15, ""]
&&& 8 \\
& & & &
\begin{bsmallmatrix}
0 & 1\\
1 & i
\end{bsmallmatrix}, 
0_2 \arrow[uuul, ""] 
& 7\\
\begin{bsmallmatrix}
0 & 1\\
0 & 0
\end{bsmallmatrix}, 
0_2 
\arrow[uurr, ""]\arrow[uur, ""]
&
0_2,I_2 \arrow[uu,""] 
\arrow[uuur,""] \arrow[uuuurrr,bend right= 15,""]
& 
1\oplus 0,a\oplus 0  \arrow[uu,""] \arrow[uuu,bend right= 10,""] \arrow[uuur,bend right= 10,""] \arrow[uuull,""]\arrow[uul,""] \arrow[uull,""]\arrow[uuuurr,bend right= 10,""]
& 
& & 6\\
& & &  (1\oplus -1,0_2) \arrow[uur,""] \arrow[uuuu,""] \arrow[uuuuur,""] \arrow[uuuuull,bend left= 10,""] & & 5
\end{tikzcd}
\end{equation*}
\caption{
The paths 
not mentioned in Theorem \ref{izrek} (\ref{izrek3}), (\ref{izrek2}), (\ref{izrek1}), (\ref{1izrek1}); $a,b,d>0$, $\zeta\in \mathbb{C}$, $\zeta^{*}\in \mathbb{C}^{*}$, $\tau\in (0,1)$, $\theta\in (0,\pi)$, $\varphi\in [0,\pi)$. 
} 
\label{cgraph2}
\end{figure}
%
%
\normalsize

\begin{theorem}\label{izrek}
Let bundles with normal forms of types from Lemma \ref{lemalist} be vertices in the closure graph for the action $\Psi$ in (\ref{aAB}). The graph contains precisely the paths described by the following statements:
%
\begin{enumerate}
\item \label{izrek3} There is a path from $(0_2,0_2)$ to any bundle. There exist paths from $\Bun_{\Psi}(1\oplus 0,0_2)$ to all bundles, except to 
$\Bun_{\Psi}(0_2,B)$ for $B\in \mathbb{C}_S^{2\times 2}$.
%
%
\item \label{izrek2} 
There exist paths from $\Bun_{\Psi}(0_2,1\oplus 0)$ 
to all bundles, except to 
$\Bun_{\Psi}(A,0_2)$ for $A\in \mathbb{C}^{2\times 2}$. 
%
%
%
\item \label{izrek1} From every bundle, except $\Bun_{\Psi}(1\oplus e^{i\theta},B)$ for $0\leq \theta<\pi$, $B\in \mathbb{C}_S^{2\times 2}$, there exists a path to the bundle $\Bun_{\Psi}\big(\begin{bsmallmatrix}
0 & 1\\
\tau & 0
\end{bsmallmatrix},\begin{bsmallmatrix}
e^{i\varphi} & b\\
b & \zeta
\end{bsmallmatrix}\big)$ with $0\leq \varphi<\pi$, $0<b$, $\zeta\in \mathbb{C}$.
\item \label{1izrek1}
From every bundle, except 
$\Bun_{\Psi}\big(\begin{bsmallmatrix}
0 & 1\\
\tau & 0
\end{bsmallmatrix},B\big)$ for $0\leq \tau < 1$, $B\in \mathbb{C}_S^{2\times 2}$, there exists a path to the bundle 
$\Bun_{\psi}\big(1\oplus e^{i\theta},\begin{bsmallmatrix}
a & \zeta^{*}\\
\zeta^{*} & d
\end{bsmallmatrix}\big)$ with $0\leq \theta<\pi$, $\zeta^{*}\in \mathbb{C}^{*}$ and
$a,d>0$.
\item \label{izrek5} All other paths that are
not mentioned in (\ref{izrek3}), (\ref{izrek2}), (\ref{izrek1}), (\ref{1izrek1}) are noted in Figure \ref{cgraph2}. (Dimensions of bundles are indicated on the right.)
\end{enumerate}  
\end{theorem}


\begin{remark}
We prove $(\widetilde{A},\widetilde{B})\to (A,B)$ by finding $(A(s),B(s))\in \Bun (A,B)$, $c(s)\in S^{1}$, $P(s)\in GL_2(\mathbb{C})$ such that
$c(s) (P(s))^*A(s)P(s)\to\widetilde{A}$ and $(P(s))^TB(s)P(s)\to\widetilde{B}$ as $s\to 0$.
It often includes tedious calculations and intriguing estimates; but since these do not seem to be of any special interest we omit them and thus shorten the proof significantly. 
When $(\widetilde{A},\widetilde{B})\not \to (A,B)$, then a lower bound for the distance from $(\widetilde{A},\widetilde{B})$ to $\Bun_{\Psi}(A,B)$ will be provided as part of the proof of Theorem \ref{izrek}. Note that if $\dim \Bun_{\Psi}(A,B)\leq \dim \Orb_{\Psi}(\widetilde{A},\widetilde{B})$, then it implies $(\widetilde{A},\widetilde{B})\not\to (A,B)$ (\cite[Propositions 2.8.13,2.8.14]{RAG}), but it gives no estimate on the distance from 
$(\widetilde{A},\widetilde{B})$ to $\Bun_{\Psi}(A,B)$.
\end{remark}

The following result is an immediate consequence of Theorem \ref{izrek} (see \cite[Corollary 3.8]{TS2} for an analogous result in the case of the closure graph for orbits).

\begin{corollary}\label{pospert}
Let $M$ be a compact real $4$-manifold embedded $\mathcal{C}^2$-smoothly in a complex $3$-manifold $X$ and let $p_1,\ldots,p_k\in M$ be precisely (all) its complex points with the corresponding normal forms up to quadratic terms $(A_1,B_1),\ldots,(A_k,B_k)\in \mathbb{C}^{n\times n}\times \mathbb{C}_S^{n\times n}$. Assume that $M'$ is a deformation of $M$ obtained by a smooth isotopy of $M$, and let $p\in M'$ be a complex point with the corresponding quadratic normal form $(A,B)$. If the isotopy is sufficiently $\mathcal{C}^2$-small then $p$ is arbitrarily close to some $p_{j_0}$, $j_{0}\in \{1,\ldots,k\}$, and $(A_{j_{0}},B_{j_{0}})\to(A,B)$ is a path in the closure graph for bundles for the action (\ref{aAB}).
\end{corollary}

\begin{remark}
The lower bounds for the distances from normal forms to other bundles 
give the estimate how small the isotopy 
in the corollary needs to be.
\end{remark}


%
%
%
%
%
\section{Proof of Lemma \ref{lemapsi2} and Lemma \ref{lemapsi1}}\label{proofL}


In this section we prove Lemma \ref{lemapsi2} and Lemma \ref{lemapsi1}. We start with a technical lemma which is an adaptation of \cite[Lemma 4.1]{TS2} to the case of bundles. 


\begin{lemma}\label{lemadet}
\begin{enumerate}
Suppose $P\in GL_2(\mathbb{C})$, $\widetilde{A},A,E,\widetilde{B},B,F\in \mathbb{C}^{2\times 2}$, $c\in S^{1}$. 
\item \label{lemadeta} If $cP^*AP=\widetilde{A}+E$,
$\|E\|\leq \min \{\tfrac{|\det \widetilde{A}|}{8\|\widetilde{A}\|+4},1\}$ it then follows that
\begin{align}\label{PAE}
&\big|\sqrt{\det A}\big|\,|\det P|=\bigl|\sqrt{\det \widetilde{A}}\bigr|+r, \qquad 
\small{
|r|\leq \left\{
\begin{array}{ll}
\tfrac{\|E\|(4\|\widetilde{A}\|+2)}{|\det \widetilde{A}|}, & \det \widetilde{A}\neq 0\\
\sqrt{\|E\|(4\|\widetilde{A}\|+2)}, & \det \widetilde{A}= 0
\end{array}.\right.}
\end{align}
\normalsize
Moreover, if $A,\widetilde{A}\in GL_2(\mathbb{C})$ and $\Delta:=\arg \bigl(\tfrac{\det \widetilde{A}}{\det A}\bigr)$ we have 
\begin{align}
\label{cE}
&c= (-1)^{k}e^{ \frac{i\Delta}{2}}+g, \quad c^{-1}= (-1)^{k}e^{-\frac{i\Delta}{2}}+\overline{g}, \qquad k\in \mathbb{Z},\quad
|g|\leq \tfrac{\|E\|(8\|\widetilde{A}\|+4)}{|\det \widetilde{A}|}.
\end{align}
\item \label{lemadetb} If $P^TBP=\widetilde{B}+F$,
$\|F\|\leq \min \{\tfrac{|\det \widetilde{B}|}{4\|\widetilde{B}\|+2},1\}$, then
\begin{align*}
\sqrt{\det B} \,\det P=\sqrt{\det \widetilde{B}}+r, \qquad 
\small{
|r|\leq \left\{\begin{array}{ll}
\tfrac{\|F\|(4\|\widetilde{B}\|+2)}{|\det \widetilde{B}|}, & \det \widetilde{B}\neq 0\\
\sqrt{\|F\|(4\|\widetilde{B}\|+2)}, & \det \widetilde{B}= 0
\end{array}\right..}
\end{align*}
\normalsize
\item \label{pathlema2ii} 
Let further $A,\widetilde{A}\in GL_2(\mathbb{C})$, $\|E\|\leq \min\{1,||\widetilde{A}^{-1}||^{-1},\frac{|\det\widetilde{A}|}{8\|\widetilde{A}\|+4}\}$ and $cP^*AP=\widetilde{A}+E$, $P^TBP=\widetilde{B}+F$.
It then implies that
\begin{align*}
&|\det \widetilde{A}\det B|=|\det \widetilde{B}\det A|+r,\\
&|r|\leq \max\{\|E\|,\|F\|\}\tfrac{|\det A|}{|\det \widetilde{A}|}
\big(4\max\{\|\widetilde{A}\|,\|\widetilde{B}\|,|\det\widetilde{A}|,|\det \widetilde{B}|\}+2\big)^{2}.
\end{align*}
%
Moreover, if in addition $B,\widetilde{B}$ are nonsingular and $|\det A|=|\det \widetilde{A}|=\|\widetilde{A}\|=1$, 
$\|E\|,\|F\|\leq \frac{|\det \widetilde{B}|}{4(4\max\{1,\|\widetilde{B}\|,\det \widetilde{B}\}+2)^{2}}$, $\Gamma:=\arg \big(\frac{\det \widetilde{B}}{\det B}\big)$, then we have
\[
\det P= (-1)^{l}e^{i\frac{\Gamma}{2}}+p,\qquad l\in \mathbb{Z}, \quad |p|\leq \|F\|\tfrac{8\|\widetilde{B}\|+4}{\sqrt{3}|\det \widetilde{B}|}.
\]
\end{enumerate}
\end{lemma}

\vspace{-4mm}

\begin{proof}
For $\xi,h\in \mathbb{C}$, $\zeta\in \mathbb{C}^{*}$
we have $\xi\zeta^{-1}=1+\frac{h}{\zeta}=|1+\frac{h}{\zeta}|e^{i \psi}$ with $|\frac{h}{\zeta}|\leq \frac{1}{2}$, hence $\psi \in (-\tfrac{\pi}{2},\tfrac{\pi}{2})$ and 
$|\sin \psi|=\bigl|\Ima \big(\tfrac{1+\frac{h}{\zeta}}{|1+\frac{h}{\zeta}|}\big)\bigr|
\leq 
\tfrac{|\Ima \frac{h}{\zeta}|}{|1+\frac{h}{\zeta}|} 
\leq \tfrac{|\frac{h}{\zeta}|}{1-|\frac{h}{\zeta}|}\leq \tfrac{2|h|}{|\zeta|}$.  
Thus
\begin{equation}\label{ocenah}
\xi=\zeta+h, \,|h|\leq \tfrac{|\zeta|}{2}\neq 0 \quad\textrm{implies}\quad \arg (\xi)-\arg(\zeta)=\psi \in (-\tfrac{\pi}{2},\tfrac{\pi}{2}),\,\,|\sin \psi|\leq 2|\tfrac{h}{\zeta}|.
\end{equation}
%

Estimating the absolute values of the entries of the matrices by the max norm of the matrices, and by slightly simplifying, we obtain that for any $X,D\in \mathbb{C}^{2\times 2}$:
\begin{align}\label{detxe}
\big||\det (X+D)|-|\det X|\big|& \leq \big|\det (X+D)-\det X \big|\leq 
\|D\|\big (4\|X\| +2\|D\|\big).
\end{align}

Furthermore, we apply the determinant to $cP^{*}AP=\widetilde{A}+E$, $Q^{T}BQ=\widetilde{B}+F$ to get 
\begin{equation}\label{detcP}
c^2|\det P|^2\det A= \det (\widetilde{A}+E), \qquad (\det Q)^2\det B= \det (\widetilde{B}+F).
\end{equation}
Assuming $\|E\|,\|F\|\leq 1$ and using (\ref{detxe}) for $X=\widetilde{A}$, $D=E$ and $X=\widetilde{B}$, $D=F$ gives
%
\begin{align}\label{detPQ}
&|\det A|\,|\det P|^{2}=|\det \widetilde{A}|+p, \qquad |p|\leq \|E\|(4\|\widetilde{A}\|+2),\\
&\det B(\det Q)^{2}=\det \widetilde{B}+q, \qquad |q|\leq \|E\|(4\|\widetilde{B}\|+2),\nonumber
\end{align}
%
respectively. We observe another simple fact. If $|s|\leq 1$ then there exists $s'$ so that
\begin{equation}\label{ocenakoren}
\sqrt{1+s}=(-1)^{l}(1+s'), \qquad l\in \mathbb{Z}, \,\,\Rea (s')\geq -1  , \,\, |s'|\leq |s|.
\end{equation}
%
We apply (\ref{ocenakoren}) to (\ref{detPQ}) for $\|E\|\leq \tfrac{|\det \widetilde{A}|}{4\|\widetilde{A}\|+2}$ and $\|F\|\leq \tfrac{|\det \widetilde{B}|}{4\|\widetilde{B}\|+2}$ to obtain (\ref{PAE}) and (\ref{lemadetb}).

The right-hand side of (\ref{detxe}) for $X$ nonsingular and $D$ with $\|D\|\leq 1$ leads to
\begin{equation}\label{XEf}
|\tfrac{\det (X+D)}{\det (X)}-1|                                                \leq \tfrac{\|D\|(4\|X\|+2)}{|\det X|}.
\end{equation} 
By assuming $\|E\|\leq \tfrac{|\det \widetilde{A}|}{8\|\widetilde{A}\|+4}$ and applying (\ref{ocenah}) to (\ref{XEf}) for $X=\widetilde{A}$, $D=E$ we obtain 
\begin{equation}\label{detex2}
\psi=\arg \big(\tfrac{\det (\widetilde{A}+E)}{\det \widetilde{A}}\big)\in (-\tfrac{\pi}{2},\tfrac{\pi}{2}),\qquad|\sin \psi |
\leq \tfrac{\|E\|(8\|\widetilde{A}\|+4)}{|\det \widetilde{A}|}. 
\end{equation}
From (\ref{detcP}) we get
\begin{equation}\label{cP}
c^2|\det P|^2=\tfrac{\det (\widetilde{A}+E)}{\det A}=\tfrac{\det (\widetilde{A}+E)}{\det \widetilde{A}}\tfrac{\det \widetilde{A}}{\det A}
\end{equation}
%
and it follows that
$c=(-1)^{k}e^{i(\frac{\Delta}{2}+\frac{\psi}{2})}$, $k\in \mathbb{Z}$, $\Delta=\arg \bigl(\frac{\det \widetilde{A}}{\det A}\bigr)$.
Using $e^{i\frac{\psi}{2}}=1+2i(\sin \frac{\psi}{4})e^{i\frac{\psi}{4}}$ and 
$2|\sin \frac{\psi}{4}|\leq |\frac{\psi}{2}|\leq  |\sin \psi | 
$ for $\psi\in (-\tfrac{\pi}{2},\tfrac{\pi}{2})$, we deduce (\ref{cE}).

We multiply (\ref{detcP}) for $P=Q$ by $\det \widetilde{B}$ and $\det \widetilde{A}$. By comparing the moduli of the expressions, and assuming $\|E\|\leq||\widetilde{A}^{-1}||^{-1}$ (hence $\det(\widetilde{A}+E)\neq 0$), 
we get
\begin{equation}\label{eqdet}
|\det B||\det \widetilde{A}|=|\det A|\tfrac{|\det \widetilde{A}||\det (\widetilde{B}+F)|}{|\det(\widetilde{A}+E)|}.
\end{equation}
Set $d_{X,D}:=|\det (X+D)|-|\det (X)|$ and apply (\ref{detxe}) for $X=\widetilde{A}$, $D=E$ and $X=\widetilde{B}$, $D=F$: 
\begin{align*}
\big|\tfrac{|\det \widetilde{A}||\det (\widetilde{B}+F)|}{|\det(\widetilde{A}+E)|}-|\det \widetilde{B}|\big|=\big|\tfrac{d_{\widetilde{B},F}|\det \widetilde{A}|-d_{\widetilde{A},E}|\det \widetilde{B}|}{d_{\widetilde{A},E}+|\det(\widetilde{A})|} \big|\leq 
\tfrac{|\det \widetilde{B}|\|E\|\big (4\|\widetilde{A}\| +2\big)+|\det \widetilde{A}|\|F\|\big (4\|\widetilde{B}\| +2\big)}{|\det \widetilde{A}|-\|E\|\big (4\|\widetilde{A}\| +2\big)},
\end{align*}
provided that $\|E\|\leq
\min\big\{||\widetilde{A}^{-1}||^{-1},\tfrac{|\det\widetilde{A}|}{8\|\widetilde{A}\|+4}\big\}$.
We combine it with (\ref{eqdet}):
\vspace{-1mm}
\begin{align}\label{ocenaABdet}
&\big||\det \widetilde{A}\det B|-|\det \widetilde{B}\det A|\big|
=|\det A| \big| \tfrac{|\det \widetilde{A}||\det (\widetilde{B}+F)|}{|\det(\widetilde{A}+E)|}-|\det \widetilde{B}|       \big|  \\
\quad &\leq \tfrac{|\det A|}{|\det \widetilde{A}|}\max\{\|E\|,\|F\|\}4\max\big\{|\det \widetilde{A}|,|\det \widetilde{B}|\big\}\big(4\max\{\|\widetilde{A}\|,\|\widetilde{B}\|\} +2\big).\nonumber
%
\end{align}
Further, let $B,\widetilde{B}$ be nonsingular and $|\det A|=|\det \widetilde{A}|=\|\widetilde{A}\|=1$, $\|F\|\leq\{\tfrac{|\det \widetilde{B}|}{4\|\widetilde{B}\|+2},1\}$, $r:=|\det B|-|\det \widetilde{B}|$. Applying (\ref{XEf}) for $X=\widetilde{B}$, $D=F$ and 
(\ref{detcP}) for $Q=P$ yields 
\[
(\det P)^{2}=\tfrac{\det (\widetilde{B}+F)}{\det \widetilde{B}}\tfrac{\det \widetilde{B}}{\det B}= e^{i\Gamma}\big(1-\tfrac{r}{|\det\widetilde{B}|+r}\big)(1+\epsilon'),
\quad \Gamma=\arg \big(\tfrac{\det{\widetilde{B}}}{\det B}\big), \, |\epsilon'|\leq\|F\|\tfrac{4\|\widetilde{B}\|+2}{|\det \widetilde{B}|}.
\]
Provided that $\|E\|,\|F\|\leq \frac{|\det \widetilde{B}|}{4(4\max\{1,\|\widetilde{B}\|,\det \widetilde{B}\}+2)^{2}}$ we use (\ref{ocenaABdet}) to assure $|r|\leq \frac{|\det \widetilde{B}|}{4}$ (hence $|1-\frac{r}{|\det \widetilde{B}|+r}|\leq \frac{4}{3}$).
By applying (\ref{ocenakoren}) we complete the proof of (\ref{pathlema2ii}). 
\end{proof}

\vspace{-1mm}

We proceed with a simple proof of Lemma \ref{lemapsi2}.

\vspace{-1mm}

\begin{proof}[Proof of Lemma \ref{lemapsi2}.]
The closure graph for $2\times 2$ symmetric matrices is obtained by an easy and straightforward calculation. 

We write the matrix equation $P^{T} AP=F+\widetilde{B}$ for $B=\begin{bsmallmatrix}
0 & b\\
b & d
\end{bsmallmatrix}$ componentwise:
\vspace{-1mm}
\begin{align}\label{eqBFbasic1}
&2bux+du^2=\widetilde{a}+\epsilon_1\nonumber \\
&bvx+buy+duv=\widetilde{b}+\epsilon_2\\ 
&2byv+dv^2=\widetilde{d}+\epsilon_4. \nonumber
\end{align}
By adding and subtracting $b\det P=b(vx-uy)$ from the second equation 
yields
\vspace{-1mm}
\begin{equation}\label{eq2bv}
2bvx+duv=b\det P+\widetilde{b}+\epsilon_2, \qquad 
2buy+duv=\widetilde{b}+\epsilon_2-b\det P.
\end{equation}
We multiply the first (the second) equation of (\ref{eq2bv}) by $u$ (by $v$) and compare it with the first (the last) equation of (\ref{eqBFbasic1}), multiplied by $v$ (by $u$): 
\begin{equation}\label{ubv}
u(b\det P+\widetilde{b}+\epsilon_2)=v(\widetilde{a}+\epsilon_4), \qquad v(-b\det P+\widetilde{b}+\epsilon_2)=u(\widetilde{d}+\epsilon_4).
\end{equation}
For $b=0$ we obtain (D\ref{p5d}). Since $\det B=-b^{2}$ we deduce from Lemma \ref{lemadet} (\ref{lemadetb}) that 
\vspace{-1mm}
\begin{align}\label{bPBF}
b\det P=i(-1)^{l}\sqrt{\det \widetilde{B}}+r, \qquad l\in \mathbb{Z},\,\,
\small{
|r|\leq \left\{\begin{array}{ll}
\tfrac{\|F\|(4\|\widetilde{B}\|+2)}{|\det \widetilde{B}|}, & \det \widetilde{B}\neq 0\\
\sqrt{\|F\|(4\|\widetilde{B}\|+2)}, & \det \widetilde{B}= 0
\end{array}\right..}
\end{align}
%
\normalsize
Together with (\ref{eq2bv}) for $d=0$ and (\ref{ubv}) this gives (D\ref{p4}) and (D\ref{p3bd}).

Next, the equation $P^{T} AP=F+\widetilde{B}$ for $B=\begin{bsmallmatrix}
a & b\\
b & 0
\end{bsmallmatrix}$ yields
\vspace{-1mm}
\begin{align}\label{eqBF5obrat}
&ax^2+2bux=\widetilde{a}+\epsilon_1\nonumber \\
&axy+bvx+buy=\widetilde{b}+\epsilon_2\\ 
&ay^2+2byv=\widetilde{d}+\epsilon_4. \nonumber
\end{align}
We add and subtract $b\det P=b(vx-uy)$ from the second equation of (\ref{eqBF5obrat}):
\vspace{-1mm}
\[
2bvx+axy=b\det P+\widetilde{b}+\epsilon_2, \qquad 
2buy+axy=\widetilde{b}+\epsilon_2-b\det P.
\]
By multiplying the first (the second) equation by $y$ (by $x$) and comparing it with the last (the first) equation of (\ref{eqBF5obrat}), multiplied by $x$ (by $y$), gives 
\vspace{-1mm}
\begin{equation}\label{ybx}
y(b\det P+\widetilde{b}+\epsilon_2)=x(\widetilde{d}+\epsilon_4), \qquad x(-b\det P+\widetilde{b}+\epsilon_2)=y(\widetilde{a}+\epsilon_1).
\end{equation}
For $b=0$ we get (D\ref{p5a}), while using (\ref{bPBF}) and (\ref{ybx}) we obtain (D\ref{p3ba}). 
\end{proof}


\begin{proof}[Proof of Lemma \ref{lemapsi1}.]
For actions $\Psi$, $\Psi_1$ (see (\ref{aAB}) and (\ref{actionpsi1})), it follows that $(A',B')\in \Orb_{\Psi}(A,0)$ if and only if $B'=0$ and $A'\in \Orb_{\Psi_1}(A)$. Hence $\dim \bigl(\Orb_{\Psi_1}(A)\bigr)=\dim \bigl(\Orb_{\Psi}(A,0)\bigr)$, where dimensions of orbits of $\Psi$ are obtained from Lemma \ref{lemalist}.


To prove $\widetilde{A}\to A$ it suffices to find $c(s)\in S^{1}$,
$P(s)\in GL_2(\mathbb{C})$, $A(s)\in \Bun (A)$ so that 
\begin{equation}\label{cPepsipsi1}
c(s)(P(s))^*A(s)P(s)-\widetilde{A}\to 0 \textrm{   as   } s \to 0.
\end{equation}
Trivially $0_2\to 1\oplus 0$, $
\begin{bsmallmatrix}
0 & 1 \\
0 & 0
\end{bsmallmatrix}
\to
\begin{bsmallmatrix}
0 & 1 \\
\tau & 0
\end{bsmallmatrix}
$, 
$
\begin{bsmallmatrix}
0 & 1 \\
1 & 0
\end{bsmallmatrix}
\to
\begin{bsmallmatrix}
0 & 1 \\
\tau & 0
\end{bsmallmatrix}
$ for $0<\tau <1$ 
and  
$
1\oplus e^{i\widetilde{\theta}}
\to
1\oplus e^{i\theta}
$ for $\widetilde{\theta}\in \{0,\pi\}$, $0< \theta< \pi$.
It is 
not too difficult to show 
$
1\oplus 0
\to
1\oplus \lambda
$, 
$
1\oplus 0
\to
\begin{bsmallmatrix}
0 & 1 \\
\tau & 0
\end{bsmallmatrix}
$ for
$0\leq \tau\leq 1$,
$
1\oplus -1
\to
\begin{bsmallmatrix}
0 & 1 \\
1 & i
\end{bsmallmatrix}
$
and 
$
\begin{bsmallmatrix}
0 & 1 \\
1 & i
\end{bsmallmatrix}
\to
\begin{bsmallmatrix}
0 & 1 \\
\tau & 0
\end{bsmallmatrix}
$ for $0<\tau <1$,
we take 
$
P(s)=1\oplus s
$, 
$
P(s)=\frac{1}{\sqrt{1+\tau}}\begin{bsmallmatrix}
1 & 0 \\
1 & s
\end{bsmallmatrix}
$,
$
P(s)=\frac{1}{\sqrt{2}}\begin{bsmallmatrix}
s^{-1} & s^{-1} \\
s & -s
\end{bsmallmatrix}
$ and 
$
P(s)=
\frac{1}{2\sqrt{s}}\begin{bsmallmatrix}
s & -2i \\
-is & 2
\end{bsmallmatrix}
$ with $\tau(s)=1-s$ in (\ref{cPepsipsi1}), respectively; in all cases $c(s)=1$. 
Finally,
$A(s)=1\oplus e^{i\theta(s)}$ with $\cos (\frac{\theta(s)}{2})=\frac{s}{2}$, $c(s)=1$,
$
P(s)=\sqrt{s}
\begin{bsmallmatrix}
i & is^{-1}\\
0 & -is^{-1}
\end{bsmallmatrix}
$
proves $
\begin{bsmallmatrix}
0 & 1 \\
1 & i
\end{bsmallmatrix}
\to
1\oplus e^{i\theta}
$ for $0<\theta <1$.

It is left to find necessary conditions for the existence of these paths, i.e. given $\widetilde{A}$, $E$, we must find out how $c$, $P$, $A$ depend on $E$,  $\widetilde{A}$, if the following is satisfied:
\begin{equation}\label{eqcPAE}
cP^*AP=\widetilde{A}+E, \qquad c\in S^{1}, P\in GL_2(\mathbb{C}).
\end{equation}
On the other hand, if (\ref{eqcPAE}) fails for every sufficiently small $E$, it gives $\widetilde{A}\not \to A$. 
In such cases the
lower estimates for $\|E\|$ will be provided. 
These easily follow for $\widetilde{A}\neq 0$, $A=0$ and $\det \widetilde{A}\neq 0$, $\det A=0$ (Lemma \ref{lemadet} (\ref{lemadeta})). 

Throughout the rest of the proof we denote  
\begin{equation}\label{wAE}
\widetilde{A}=
\begin{bmatrix}
\alpha & \beta\\
\gamma & \omega
\end{bmatrix}
,\qquad
E=
\begin{bmatrix}
\epsilon_1 & \epsilon_2\\
\epsilon_3 & \epsilon_4
\end{bmatrix}
,\qquad
P=
\begin{bmatrix}
x & y\\
u & v
\end{bmatrix}.
\end{equation}
%

\begin{enumerate}[label={\bf Case \Roman*.},ref={Case \Roman*.},   ,wide=0pt,itemsep=10pt]
\item \label{p1i2}
$
A=\begin{bsmallmatrix}
0 & 1\\
1 & i
\end{bsmallmatrix}
$ \quad ($\Bun_{\Psi_1}(A)=\Orb_{\Psi_1}(A)$)\\ 
This case coincides with \cite[Lemma 3.4. Case I]{TS2}; see (C\ref{r10}), (C\ref{r5}).

\item \label{p100l}
$A=1\oplus \lambda$, \quad $|\lambda|\in \{1,0\}$\\
%
%
The equation (\ref{eqcPAE}) multiplied by $c^{-1}$, written componentwise and rearranged is: 
\begin{align}\label{eq1l}
& |x|^2+\lambda|u|^2 -c^{-1}\alpha=c^{-1}\epsilon_1, \quad && \overline{x}y+\lambda \overline{u}v-c^{-1}\beta=c^{-1}\epsilon_2,\\
&\overline{y}x+\lambda\overline{v}u-c^{-1}\gamma=c^{-1}\epsilon_3,\quad && |y|^2+\lambda|v|^2-c^{-1}\omega=c^{-1}\epsilon_4.\nonumber
\end{align}
Subtracting the second complex-conjugated equation (and multiplied by $\lambda$) from the third equation (and multiplied by $\overline{\lambda}$) for $\beta, \gamma\in \mathbb{R}$ gives
\begin{align}\label{eq1l3}
&2\Ima (\lambda)\overline{v}u-c^{-1}\gamma+\overline{c}^{-1}\beta= c^{-1}\epsilon_3- \overline{c}^{-1}\overline{\epsilon}_2,\\
&-2\Ima (\lambda)\overline{y}x-c^{-1}\overline{\lambda}\gamma+\overline{c}^{-1}\lambda\beta=c^{-1}\overline{\lambda}\epsilon_3-\overline{c}^{-1}\lambda\overline{\epsilon}_2\nonumber .
\end{align}

\begin{enumerate}[label=(\alph*),wide=0pt,itemindent=2em,itemsep=6pt]

\item 
$\lambda=e^{i\theta}$, \quad $0\leq \theta \leq \pi$\\
From (\ref{eq1l3}) for $\beta=\gamma=0$, $\Ima (\lambda)=\sin \theta$ we get 
\begin{equation}\label{sintuv}
\big|(\sin \theta)\overline{v}u\big|\leq  \|E\|, \qquad \big|(\sin \theta)\overline{x}y\big|\leq  \|E\|.
\end{equation}
We take the (real) imaginary parts of the (last) first equation of (\ref{eq1l}) for $\lambda=e^{i\theta}$: 
\begin{align}\label{eqa0}
&(\sin \theta) |u|^2=\Ima (c^{-1}\alpha+c^{-1}\epsilon_1), \qquad |x|^2+(\cos \theta)|u|^2=\Rea (c^{-1}\alpha+c^{-1}\epsilon_1),\\
&(\sin \theta) |v|^2=\Ima (c^{-1}\omega+c^{-1}\epsilon_4), \qquad |y|^2+(\cos \theta)|v|^2=\Rea (c^{-1}\omega+c^{-1}\epsilon_4).\nonumber
\end{align}
If $\alpha=0$ 
we further have:
\begin{align}\label{e01}
(\sin \theta)|u|^2\leq \|E\|,& \quad 
(\sin \theta)|x|^2\leq \|E\|(\sin \theta+|\cos \theta|),\\
%
%
\big|(\sin \theta)|v|^2-\Ima (c^{-1}\omega)\big|\leq \|E\|,& \quad 
\big|(\sin \theta)|y|^2-\Rea (c^{-1}\omega)\big|\leq \|E\|(\sin \theta+|\cos \theta|).\nonumber
\end{align}
\begin{enumerate}[label=(\roman*),wide=0pt,itemindent=4em,itemsep=3pt]

\item
$
\widetilde{A}=
\begin{bsmallmatrix}
0 & 1\\
\widetilde{\tau} & 0
\end{bsmallmatrix}
$, \quad
$0\leq \widetilde{\tau}\leq 1$
%
%

If $1\leq \widetilde{\tau}< 1$, then by applying the triangle inequality to the first equation of (\ref{eq1l3}) for $\beta=1$, $\gamma=\widetilde{\tau}$, $\Ima (\lambda)=\sin \theta$ and using the first estimates of (\ref{e01}) for $\omega=0$ we obtain 
$2\|E\|\geq 2(\sin\theta)| uv|\geq 1-\widetilde{\tau}-2\|E\|$, which fails for $\|E\|<\frac{1-\widetilde{\tau}}{4}$. 

\item
$
\widetilde{A}=
\begin{bsmallmatrix}
0 & 1\\
1 & \omega
\end{bsmallmatrix}
$, \quad $\omega \in \{0,i\}$

By applying the triangle inequality to the second equation of (\ref{eq1l}), and using (\ref{e01}) with $|c^{-1}\omega|\leq 1$ leads to the inequality:
\[
(\sin \theta)(1-\|E\|)\leq \sin \theta|\overline{x}y+\lambda \overline{u}v|\leq 
\sqrt{\|E\|(1+\|E\|)}+\sqrt{2\|E\|(1+2\|E\|)}.
\]
If $\|E\|\leq\frac{1}{12}$ then we deduce $\sin \theta\leq 3\sqrt{\|E\|}$ and $\cos^{2} \theta\geq 1-9\|E\|$. 
If $\theta$ is close to $0$ then the second and the last equation of (\ref{eqa0}) for $\alpha=0$, $|c^{-1}\omega|\leq 1$ imply that $|x|^{2},|u|^{2}\leq \frac{\|E\|}{\sqrt{1-9\|E\|}}$ and $|y|^{2},|v|^{2}\leq \frac{1+\|E\|}{\sqrt{1-9\|E\|}}$, respectively. 
For $\|E\|$ so small that $1>2\frac{\sqrt{\|E\|(1+\|E\|)}}{\sqrt{1-9\|E\|}} +\|E\|$, the second equation of (\ref{eq1l}) for $\beta=1$ fails.
Next, when $\theta$ is close to $\pi$, we deduce that $\frac{1+\cos \theta}{\sin \theta}=\cot \frac{\theta}{2}$ is close to $0$ and $\pi-\theta\in (0,\frac{\pi}{2})$, hence 
\begin{align}\label{ocenasin2cos}
&|\cos\tfrac{\theta}{2}|=|\sin (\tfrac{\pi- \theta}{2})|\leq \sin (\pi-\theta)=\sin \theta,\quad |\cos(\tfrac{\theta+\pi}{4})|=|\sin (\tfrac{\pi- \theta}{4})|\leq 
\sin \theta,\\
&1+\cos \theta=
\tfrac{\cos \frac{\theta}{2}\sin \theta}{\sin \frac{\theta}{2}}\leq \tfrac{\sin \theta}{\sqrt{1-\sin^{2}\theta}}\sin \theta  
,\quad
1-\sin \tfrac{\theta}{2}=\tfrac{\cos^{2} \frac{\theta}{2}}{1+\sin \frac{\theta}{2}}\leq \sin \theta.\nonumber
\end{align}
We have $c^{-1}=-i(-1)^{k}e^{i\frac{\theta}{2}}+\overline{g}$, $|g|\leq 12\|E\|$ with $\|E\|\leq \frac{1}{12}$ (Lemma \ref{lemadet} (\ref{cE})), thus 
$\big|\Rea (c^{-1}i)\big|=|\cos \frac{\theta}{2}+i\overline{g}|\leq 3\sqrt{\|E\|}+12\|E\|$ (since $\sin \theta\leq 3\sqrt{\|E\|}$). 
Using the second (fourth) equation of (\ref{eqa0}) and (\ref{e01}) for $\alpha=0$, $\omega\in \{0,i\}$ with (\ref{ocenasin2cos}) further implies
\begin{align}\label{eqxuyvth1}
\big||x|^{2}-|u|^{2}\big|-\|E\|  &  \leq \big||x|^2-|u|^{2}+(1+\cos \theta)|u|^2\big|=\big||x|^{2}+(\cos \theta)|u|^{2}\big|\leq \|E\|, \\ 
\big||y|^{2}-|v|^{2}\big|-\tfrac{3\sqrt{\|E\|}(1+\|E\|)}{\sqrt{1-9\|E\|}}   
                                &    \leq \big||y|^{2}-|v|^{2}+(1+\cos \theta)|v|^{2}\big|=\big||y|^{2}+(\cos \theta)|v|^{2}\big|\nonumber\\
                                &    \leq 3\sqrt{\|E\|}+13\|E\|.\nonumber
\end{align}
%
Using the second equation of (\ref{eq1l}) and (\ref{e01}), (\ref{ocenasin2cos}) (for $\alpha=0$, $\omega\in \{0,i\}$) we get:
\begin{align}\label{eqxuyvth2}
14\|E\|&\geq |\overline{x}y+e^{i\theta}\overline{u}v+i(-1)^{k}e^{i\frac{\theta}{2}}|=\nonumber\\
       &= \big|(\overline{x}y-\overline{u}v-(-1)^{k})+2(\cos \tfrac{\theta}{2})e^{i\frac{\theta}{2}}\overline{u}v+2(-1)^{k}(\cos\tfrac{\theta+\pi}{4})e^{i\frac{\theta+\pi}{4}}
       \big|\\
       &\geq \big|\overline{x}y-\overline{u}v-(-1)^{k}\big|-
       2\sqrt{\|E\|(1+\|E\|)}+6\sqrt{\|E\|}.\nonumber
\end{align}
For $\omega=i$ we have $\Ima (c^{-1}i)=\sin \frac{\theta}{2}+\Ima (i\overline{g})$, $|g|\leq 12\|E\|$, therefore (\ref{e01}) yields
\[
13\|E\|\geq 
\big|(\sin\theta) |v|^{2}-(-1)^k-(-1)^{k}(\sin \tfrac{\theta}{2}-1)  \big|\geq
\big|(\sin\theta) |v|^{2}-(-1)^k\big|-3\sqrt{\|E\|}.
\]
Together with (\ref{eqxuyvth1}) and (\ref{eqxuyvth2}) it proves (C\ref{r33}).
Note that the third equation of (\ref{eqa0}) for $\theta=\pi$, $\omega=i$ 
fails for $\|E\|<\frac{1}{13}$.

\item $\widetilde{A}=\alpha\oplus 0$, $\alpha\in \{0,1\}$

If $\theta \in \{0,\pi\}$, then (\ref{eq1l}) for $e^{i\theta}=\sigma$ yields (C\ref{r9}).

\quad
By (\ref{sintuv}) and the second equation of (\ref{eq1l}) we have 
\begin{equation}\label{eqxycuv}
\big|\overline{x}y+(\cos \theta )\overline{u}v\big|\leq 2\|E\|.
\end{equation}
If $0<\theta\leq\pi$, then (\ref{eq1l}), (\ref{sintuv}), (\ref{eqxycuv}) for $\omega=\beta=\gamma=0$, $\lambda=e^{i\theta}$ give (C\ref{r32}). 

\item
$
\widetilde{A}=
1\oplus e^{i\widetilde{\theta}}
$,\quad
$0\leq \widetilde{\theta} \leq \pi$.\\
By Lemma \ref{lemadet} (\ref{cE}) we have 
$c^{-1}=(-1)^k e^{i\frac{\theta-\widetilde{\theta}}{2}}+\overline{g}$, $|\overline{g}|\leq 12\|E\|$, assuming that $\|E\|\leq \frac{1}{12}$. 
Thus the first and the last equation of (\ref{eq1l}) for $\alpha=1$, $\lambda=e^{i\widetilde{\theta}}$ are of the form:
\begin{align}\label{eq1l5}
&|x|^2+e^{i\theta}|u|^2 =(-1)^k e^{i\frac{\theta-\widetilde{\theta}}{2}}+(\overline{g}+c^{-1}\epsilon_1),\\
&|y|^2+e^{i\theta}|v|^2=(-1)^k e^{i\frac{\widetilde{\theta}+\theta}{2}}+(\overline{g}e^{i\widetilde{\theta}}+c^{-1}\epsilon_4).\nonumber
\end{align}
We take the 
imaginary parts of (\ref{eq1l5}) and apply the triangle inequality:
\begin{align}\label{eq1lim}
& \bigl| |u|^2\sin \theta-(-1)^k\sin (\tfrac{\theta-\widetilde{\theta}}{2})\bigr| \leq  \bigl| \Ima (\overline{g})+\Ima (\epsilon_1)\bigr|\leq 13\|E\|,\\
& \bigl| |v|^2\sin \theta-(-1)^k\sin (\tfrac{\widetilde{\theta}+\theta}{2})\bigr|\leq \bigl| \Ima (\overline{g}e^{i\widetilde{\theta}}+\overline{c}^{-1}\epsilon_4)\bigr|\leq 13\|E\|.\nonumber
\end{align}
In particular we have
\begin{equation*}
|u|^2\sin \theta \geq |\sin (\tfrac{\widetilde{\theta}-\theta}{2})|-13\|E\|, \quad 
|v|^2\sin \theta \geq |\sin (\tfrac{\widetilde{\theta}+\theta}{2})|-13\|E\|.
\end{equation*}
By multiplying these inequalities and using the triangle inequality we deduce 
\begin{align*}
(\sin ^2\theta)|uv|^2\geq 	
 \big|\sin (\tfrac{\widetilde{\theta}-\theta}{2})\sin (\tfrac{\widetilde{\theta}+\theta}{2})\big|-13\|E\|\bigl(\big|\sin (\tfrac{\widetilde{\theta}-\theta}{2})\big|+\big|\sin (\tfrac{\widetilde{\theta}+\theta}{2})\big|  \bigr)-169\|E\|^2 .				   
\end{align*}
By combining it with (\ref{sintuv}) and rearranging the terms we obtain
\begin{align}\label{eqtnwt}
\tfrac{1}{2}|\cos\widetilde{\theta}-\cos \theta|=\big|\sin (\tfrac{\widetilde{\theta}-\theta}{2})\sin (\tfrac{\widetilde{\theta}+\theta}{2})\big|  \leq  170 \|E\|^{2} + 26\|E\|
\leq 196 \|E\|. 
\end{align}
If $\theta\in \{0,\pi\}$ with $\widetilde{\theta}\neq \theta$ then (\ref{eqtnwt}) fails for $\|E\|<\frac{1-|\cos \widetilde{\theta}|}{392}$.

\quad 
We take the real parts in the first equation of (\ref{eq1l5}), multiply them by $\sin \theta$, then rearrange the terms and apply (\ref{eq1lim}):
\begin{align}\label{octwt}
(\sin\theta) \big||x|^{2}-(-1)^{k} \cos (\tfrac{\widetilde{\theta}-\theta}{2})\big|=&
\big|-\sin \theta \cos \theta |u|^{2}+(\sin \theta) \Rea(\overline{g}+c^{-1}\epsilon_1)\big|,\\
(\sin\theta) \big||x|^{2}-(-1)^{k}\big|-(\sin\theta) \big|\cos (\tfrac{\widetilde{\theta}-\theta}{2})-1\big|   \leq   &\big|\sin (\tfrac{\theta-\widetilde{\theta}}{2})\big|+13\|E\|+13\|E\|.\nonumber
\end{align}

\quad
Next, let $0<\widetilde{\theta},\theta<\pi$. Thus
$\frac{\theta-\widetilde{\theta}}{2}\in (-\frac{\pi}{2},\frac{\pi}{2})$ and $\frac{\theta+\widetilde{\theta}}{2}\in (\frac{\widetilde{\theta}}{2},\frac{\widetilde{\theta}+\pi}{2})\subset (0,\pi) $ with $\sin (\frac{\widetilde{\theta}+\theta}{2})\geq \min\{\sin \frac{\widetilde{\theta}}{2},\cos \frac{\widetilde{\theta}}{2}\}$. We apply (\ref{eqtnwt}) and make a trivial estimate: 
\begin{equation}\label{sintwtcos}
\tfrac{196\|E\|}{\min\{\sin \widetilde{\theta},\cos \widetilde{\theta}\}}\geq\big|\sin (\tfrac{\widetilde{\theta}-\theta}{2})\big|\geq\big|\sin (\tfrac{\widetilde{\theta}-\theta}{4})\big|= 
\tfrac{1}{\sqrt{2}}\big|\cos (\tfrac{\widetilde{\theta}-\theta}{2})-1\big|^{\frac{1}{2}}.
\end{equation}
%
By combining (\ref{octwt}) and (\ref{sintwtcos}) it is straightforward to get a constant $C>0$ so that 
\begin{equation}\label{sintx2}
(\sin\theta )\big||x|^{2}-(-1)^{k}\big|\leq \tfrac{196\|E\|}{\min\{\sin \widetilde{\theta},\cos \widetilde{\theta}\}}+2\big(\tfrac{196\|E\|}{\min\{\sin \widetilde{\theta},\cos \widetilde{\theta}\}}\big)^{2}+26\|E\|\leq C \|E\|.
\end{equation}
%

\quad
We multiply the second equation of (\ref{eq1l5}) by $e^{-i\theta}$. Then we take the imaginary parts or only rearrange the terms; in both cases we also use (\ref{sintwtcos}): 
\begin{align}\label{yv21}
(\sin \theta)|y|^{2}     & \leq \big| \sin (\tfrac{\widetilde{\theta}-\theta}{2}) \big|+14\|E\|\leq C'\|E\|, \quad
C':=\tfrac{196}{\min\{\sin \widetilde{\theta},\cos \widetilde{\theta}\}}+14,\\
\big||v|^{2}-(-1)^{k}\big|  \leq  &|e^{i\frac{\widetilde{\theta}-\theta}{2}}-1|+  |y|^{2}+|\overline{g}e^{i\widetilde{\theta}}+c^{-1}\epsilon_4|\leq
2\big(\tfrac{196\|E\|}{\min\{\sin \widetilde{\theta},\cos \widetilde{\theta}\}}\big)^{2}+13\|E\|+|y|^{2}.\nonumber
\end{align}
From the first estimate in (\ref{eq1lim}) we similarly obtain $(\sin \theta)|u|^{2}\leq C'\|E\|$.
If $\sin\theta\leq \max\{\sqrt{C},\sqrt{C'}\}\sqrt{\|E\|}$, then (\ref{eqtnwt}) yields a contradiction for sufficiently small $\|E\|$.
Otherwise $|u|^{2}\leq \sqrt{C'\|E\|}$ and (\ref{sintx2}), (\ref{yv21}) imply $\big||x|^{2}-(-1)^{k}\big|\leq \sqrt{C\|E\|}$, $|y|^{2}\leq \sqrt{C'\|E\|}$, respectively.
The last estimate in (\ref{yv21}) concludes the proof of (C\ref{r1}).

\quad
Finally, suppose $0<\theta<\pi$ and $\widetilde{\theta}\in \{0,\pi\}$; hence $\frac{\widetilde{\theta}-\theta}{2}\in (-\frac{\pi}{2},\frac{\pi}{2})$. We apply (\ref{eqtnwt}) and use (\ref{eq1lim}) for $\widetilde{\theta}=0$ or $\widetilde{\theta}=\pi$ to deduce
\begin{equation}\label{0pi}
14 \sqrt{\|E\|}\geq 
\big|\sin (\tfrac{\widetilde{\theta}-\theta}{2})\big|\geq \big|\sin (\tfrac{\widetilde{\theta}-\theta}{4})\big|,\qquad
|u|^2\sin \theta,|v|^2\sin \theta\leq 13 \|E\|+14\sqrt{\|E\|}.
\end{equation}
Assume now that $\sqrt{\|E\|}\leq \frac{\sqrt{2}}{28}$. If $\widetilde{\theta}=0$, then $|\cos \frac{\theta}{2}|\geq \frac{\sqrt{2}}{2}$, 
therefore $1-\cos \theta=
(\sin \theta )|\tan \frac{\theta}{2}|\leq \sqrt{2}\sin \theta$. Similarly, for $\widetilde{\theta}=\pi$ we have 
$|\sin \frac{\theta}{2}|\geq \frac{\sqrt{2}}{2}$ and so $1+\cos \theta=
 (\sin \theta )|\cot \frac{\theta}{2}|\leq \sqrt{2}\sin \theta $.
We take the real parts of the first equation (\ref{eq1l5}) for $\sigma=e^{i\widetilde{\theta}}$ with $\widetilde{\theta}\in \{0,\pi\}$, rearrange the terms, and apply the triangle inequality:
\begin{align}
13\|E\| &\geq\big| |x|^{2}+ \sigma|u|^{2}-(-1)^{k}+(-1)^{k}(1-\cos (\tfrac{\theta-\widetilde{\theta}}{2}))-|u|^{2}(\sigma-\cos \theta)    \big| 
\\
        &\geq \big| |x|^{2}+\sigma |u|^{2}-(-1)^{k}\big|-392\|E\|-\sqrt{2}(13 \|E\|+14\sqrt{\|E\|})
        .\nonumber
\end{align}
%
The same proof applies if we replace $x,u,(-1)^{k}$ by $y,v,\sigma (-1)^{k}$, respectively. The second equation (\ref{eq1l}) for $\beta=0$, $\lambda=e^{i\theta}$ and (\ref{sintuv}) finally yield  
\begin{align*}
\|E\| &\geq\big|\overline{x}y+e^{i\theta}\overline{u}v\big|= \big|\overline{x}y+\sigma\overline{u}v-(\sigma-\cos \theta)\overline{u}v+i(\sin\theta)\overline{u}v  \big| \nonumber\\
        &\geq \big|\overline{x}y+\sigma\overline{u}v\big|-(1+\sqrt{2})\big(13 \|E\|+14\sqrt{\|E\|}\big).
\end{align*}
Thus (C\ref{r99}) follows.

\end{enumerate}

\item $\lambda=0$\ \quad (hence $\det \widetilde{A}=0$.)\\ 
If $\widetilde{A}=\alpha\oplus 0$ for $\alpha\in\{0,1\}$, 
then (C\ref{r11}) follows from (\ref{eq1l}) for $\omega=\lambda=0$. Applying (\ref{ocenah}) for $\|E\|\leq \frac{1}{2}$ to the first equation of (\ref{eq1l}) for $\alpha=1$, $\lambda=0$ (multiplied by $c$), yields 
$c=e^{i\psi}=1+2i(\sin \frac{\psi}{2})e^{i\frac{\psi}{2}}$ with $|\sin \frac{\psi}{2}|\leq 2\|E\|$. 
If $\widetilde{A}=\begin{bsmallmatrix}
0 & 1\\
0 & 0
\end{bsmallmatrix}$, then (\ref{eq1l}) for $\lambda=\alpha=\omega=0$ yields $|x|^{2},|y|^{2}\leq \|E\|$, thus (\ref{eq1l}) fails for $\lambda=\gamma=0$,  $\|E\|<\frac{1}{2}$.

\end{enumerate}

\item \label{p01t0}
$
A=\begin{bsmallmatrix}
0 & 1\\
\tau & 0
\end{bsmallmatrix}
$, \quad $0\leq \tau \leq 1$\\
%
%
From (\ref{eqcPAE}) multiplied by $c^{-1}$ we obtain 
\begin{align}\label{equ01t}
\overline{x}u+\tau\overline{u}x -c^{-1}\alpha=c^{-1}\epsilon_1, \qquad \overline{x}v+\tau \overline{u}y-c^{-1}\beta=c^{-1}\epsilon_2,\\
\tau \overline{v}x+\overline{y}u-c^{-1}\gamma=c^{-1}\epsilon_3,\qquad \overline{y}v+\tau\overline{v}y-c^{-1}\omega=c^{-1}\epsilon_4.\nonumber
\end{align}
Rearranging the terms of the first and the last equation immediately yields 
\begin{align}\label{eq01t2}
(1+\tau)\Rea(\overline{x}u)+i(1-\tau)\Ima(\overline{x}u)=c^{-1}\alpha+c^{-1}\epsilon_1,\\
(1+\tau)\Rea(\overline{y}v)+i(1-\tau)\Ima(\overline{y}v)=c^{-1}\omega+c^{-1}\epsilon_4,\nonumber
\end{align}
while multiplying the third (second) complex-conjugated equation with $\tau$, subtracting it from the second (third) equation, and rearranging the terms, give
\begin{align}\label{eq011t3}
(1-\tau^2)\overline{x}v=&c^{-1}(\beta+\epsilon_2)-\tau\overline{c}^{-1}(\overline{\gamma}+\overline{\epsilon_3})= 
(c^{-1}\beta-\tau\overline{c}^{-1}\overline{\gamma})+(c^{-1}\epsilon_2-\tau\overline{c}^{-1}\overline{\epsilon_3})\\
(1-\tau^{2})\overline{y}u= &c^{-1}(\gamma+\epsilon_3)-\tau \overline{c}^{-1}(\overline{\beta}+\overline{\epsilon_2})=
(c^{-1}\gamma-\tau \overline{c}^{-1}\overline{\beta})+(c^{-1}\epsilon_3-\tau \overline{c}^{-1}\overline{\epsilon_2}).\nonumber
\end{align}

For the existence of paths to $\begin{bsmallmatrix}
0 & 1\\
1 & 0
\end{bsmallmatrix}$ ($*$-congruent to $1\oplus -1$) see  \ref{p100l} 



\quad
Using (\ref{equ01t}) we obtain that
\begin{equation}\label{xyuvt2}
(1+\tau)|\overline{x}u|\geq |\alpha+\epsilon_1|\geq (1-\tau)|\overline{x}u|,\qquad (1+\tau)|\overline{y}v|\geq |\omega+\epsilon_4|\geq (1-\tau)|\overline{y}v|.
\end{equation}
By multiplying the left-hand and the right-hand sides of these inequalities we get
\begin{align}\label{xyuvt22}
(1+\tau)^{2}|\overline{x}u\overline{y}v|\geq |\alpha\omega|-\bigl(|\alpha|+|\omega|\bigr)\|E\|-\|E\|^2,\\
\label{xyuvt}
|\alpha\omega|+\bigl(|\alpha|+|\omega|\bigr)\|E\|+\|E\|^2 \geq (1-\tau)^{2}|\overline{x}u\overline{y}v|.
\end{align}

\begin{enumerate}[label=(\alph*),wide=0pt,itemindent=2em,itemsep=3pt]


\item $\widetilde{A}=\begin{bsmallmatrix}
0 & 1\\
\gamma & \omega
\end{bsmallmatrix}$,\quad either  $0 \leq \gamma\leq 1$, $\omega=0$ or $\gamma=1$, $\omega=i$

Equations (\ref{eq011t3}) for $\beta=1$, $0\leq \gamma \leq 1$ imply
\begin{align*}
(1-\tau^2)|\overline{x}v|\geq |\tau \gamma-1|-(\tau+1)\|E\|,\quad
(1-\tau^{2})|\overline{u}y|\geq |\gamma-\tau|-(1+\tau)\|E\|.\nonumber
\end{align*}
By combining these inequalities and making some trivial estimates we deduce
\[
(1-\tau^{2})^2|\overline{y}u\overline{x}v|\geq |\tau\gamma-1|\,|\gamma-\tau|-(1+\tau)\bigl(\tau\gamma+1+\gamma+\tau\bigr)\|E\|-(1+\tau)^2\|E\|^2.
\]
Together with (\ref{xyuvt}) for $\alpha=0$ and using $\|E\|\geq \|E\|^{2}$ we get 
\begin{align}\label{twt}
&(1+\tau)^2(1+|\omega|)\|E\|  \geq |\tau\gamma-1|\,|\gamma-\tau|-(1+\tau)^{2}(\gamma+1)\|E\|-(1+\tau)^2\|E\|,\nonumber\\
&(1+\tau)^2\big(3+|\omega|+\gamma\big)\|E\| \geq |\tau\gamma-1|\,|\gamma-\tau|\geq |1-\gamma|\,|\gamma-\tau|.
\end{align}
%
If $0\leq \gamma<1$ (if $\gamma=1$) then the right-hand (the left-hand) side of 
(\ref{twt}) implies 
\begin{equation}\label{ogt}
|\gamma-\tau|\leq \left\{\begin{array}{ll}
\frac{(1+\tau)^{2}}{1-\gamma}\big(4+|\omega|\big)\|E\|, & 0\leq \gamma<1\\
(1+\tau)\sqrt{(4+|\omega|)\|E\|}, & \gamma=1
\end{array}.\right.
\end{equation}
When either $\tau=0$, $\gamma> 0$ or $\tau=1$, $\gamma< 1$ (and $\|E\|$ is small enough), then (\ref{ogt}) fails. 
If $0\leq \gamma<1$ and $\|E\|\leq \frac{(1-\gamma)^{2}}{2(1+\tau)^{2}(4+|\omega|)}$ (hence $1-\tau\geq |1-\gamma|-|\gamma-\tau|\geq \frac{1-\gamma}{2}$), then (\ref{xyuvt2}) for $\alpha=0$ (for $\omega=0$) yields $|\overline{x}u|\leq \frac{2}{1-\gamma}\|E\|$ (and $|\overline{y}v|\leq \frac{2}{1-\gamma}\|E\|$). 
Next, (\ref{eq011t3}), (\ref{ogt}) for $\beta=1$, $\gamma=0$,  imply $|\overline{y}u|\leq C \|E\|$ and $|\overline{x}v-c^{-1}|\leq C \|E\|$ for some constant $C>0$ (see (C\ref{r6}) for $\widetilde{\tau}=0$, $0\leq \tau<1$).

\quad
By Lemma \ref{lemadet} (\ref{cE}) for $1\geq\tau> 0$, $\widetilde{A}=\begin{bsmallmatrix}
0 & 1\\
\gamma & \omega
\end{bsmallmatrix}$ with $1\geq\gamma> 0$ and $\|E\|\leq \frac{\gamma}{12}\leq \frac{1}{12}$, we have $c^{-1}=(-1)^{k}+\overline{g}$, $k\in\mathbb{Z}$, $|g|\leq \frac{12}{\gamma}\|E\|$, thus (\ref{eq011t3}) for $\beta=1$ (and $\gamma\in \mathbb{R}$) gives
\begin{align*}
&(1-\tau^2)\overline{x}v=\big((-1)^{k}(1-\tau\gamma)-g\tau\gamma+\overline{g}\big)+(c^{-1}\epsilon_2-\tau\overline{c}^{-1}\overline{\epsilon_3})\\
&(1-\tau^{2})\overline{y}u=(-1)^{k}(\gamma-\tau)+\gamma\overline{g}-\tau g+(c^{-1}\epsilon_3-\tau \overline{c}^{-1}\overline{\epsilon_2}).\nonumber
\end{align*}
We further obtain
\begin{align}\label{ocetyuxv}
(1-\tau^{2})|\overline{y}u| & \leq (\gamma-\tau)+(\tau\gamma+1)\tfrac{12}{\gamma}\|E\|+(1+\tau)\|E\|, \\ 
(1-\tau^2)\big|\overline{x}v-(-1)^{k}\big|  &  \leq \tau(\gamma-\tau)+\tfrac{12(\tau\gamma+1)}{\gamma}\|E\|+(\tau+1)\|E\|.\nonumber
\end{align}
Using (\ref{ogt}) for $0<\gamma<1$ we deduce that the left-hand sides of (\ref{ocetyuxv}) are bounded by $D \|E\|$, where $D:=\frac{4(4+|\omega|)}{1-\gamma}+\frac{12(\gamma+1)}{\gamma}+2$. Thus either $1-\tau^{2}\leq \sqrt{D}\sqrt{\|E\|}$ and  
\[
|1-\gamma|\leq |\tau-\gamma|+|1-\tau|\leq \tfrac{(1+\tau)^{2}}{1-\gamma}(2+|\omega|)\|E\|+\tfrac{\sqrt{D}\sqrt{\|E\|}}{2}
\]
fails for small $\|E\|$, or we have $|\overline{y}u|,\big|\overline{x}v-(-1)^{k}\big|\leq \sqrt{D}\sqrt{\|E\|}$ (see (C\ref{r6}) for $0<\tau_0,\tau<1$).
%
The second equation of (\ref{equ01t}) with $\beta=1$, $c^{-1}=(-1)^{k}+\overline{g}$, $k\in\mathbb{Z}$, $|g|\leq 12\|E\|$ gives
\begin{equation}\label{ocexvtuy2}
\big|\overline{x}v+\overline{u}y-(-1)^{k}\big|-(1-\tau)|\overline{u}y|\leq \big|\overline{x}v+\tau\overline{u}y-(-1)^{k}\big|\leq 12\|E\|+\|E\|.
\end{equation}
%
From (\ref{equ01t}), (\ref{eq01t2}), (\ref{ogt}), (\ref{ocetyuxv}), (\ref{ocexvtuy2})
for $\alpha=0$, $\omega\in \{0,i\}$, $\gamma=1$ we deduce (C\ref{r44}).
If $\omega=i$, $\tau=1$ and $\|E\|<\frac{1}{13}$, then the second equality of (\ref{eq01t2}) fails.

\item $\widetilde{A}=\alpha\oplus \omega$ 

From (\ref{eq011t3}) for $\beta=\gamma=0$ it follows that
\begin{equation}\label{xbarv}
(1-\tau^{2})|\overline{x}v|\leq (1+\tau)\|E\|, \quad (1-\tau^{2})|\overline{u}y|\leq (1+\tau)\|E\|,\quad (1-\tau)^{2}|\overline{x}v\overline{u}y|\leq  \|E\|^2.
\end{equation}
%

%
%
Next, (\ref{xbarv}) yields either $(1-\tau)|\overline{x}u|\leq \|E\|$ or $(1-\tau)|\overline{y}v|\leq \|E\|$.

\quad
By Lemma \ref{lemadet} (\ref{cE}) for $0<\tau\leq 1$, $\widetilde{A}=1\oplus e^{i\widetilde{\theta}}$, we have $c^{-1}=(-1)^{k}e^{-i\frac{\widetilde{\theta}+\pi}{2}}+\overline{g}$, $k\in\mathbb{Z}$, $|g|\leq 12\|E\|$.
We take the imaginary parts of (\ref{eq01t2}) with $\alpha=1$, $\omega=e^{i\widetilde{\theta}}$, $0<\tau<1$ to deduce $|\cos \frac{\widetilde{\theta}}{{2}}|\leq 14\|E\|$, which fails for $0\leq \widetilde{\theta} <\pi$ and small $\|E\|$.  

\quad
By combining (\ref{xbarv}) with (\ref{xyuvt22}) for $|\alpha|=|\omega|=1$ and using $\|E\|\leq \frac{1}{4}$,
we get 
\[
\tfrac{1}{4}(1-\tau)^{2}\leq (1-\tau)^{2}(1-2\|E\|-\|E\|^2)\leq (1-\tau^{2})^{2}|\overline{x}v\overline{u}y| \leq (1+\tau)^{2}\|E\|^{2}.
\]
Thus $1-\tau\leq 4\|E\|$. (In particular, we obtain a contradiction for $\tau=0$, $|\alpha|=|\omega|=1$.)
%
%
When $\widetilde{\theta}=\pi$ (i.e. $\widetilde{A}=1\oplus-1$, $c^{-1}=(-1)^{k+1}+\overline{g}$, $k\in\mathbb{Z}$, $|g|\leq 12\|E\|$), we use (\ref{equ01t}), (\ref{eq01t2}) 
for $\beta=0$, $\alpha=-\omega=1$ to get $(1-\tau) \Ima (\overline{x}u),(1-\tau) \Ima (\overline{x}u)\leq 13\|E\|$ and
\begin{align}\label{equal10}
&|\overline{x}v+\overline{u}y|-2\|E\|\leq  |\overline{x}v+\overline{u}y|-(1-\tau)|\overline{u}y|\leq|\overline{x}v+\tau\overline{u}y|\leq \|E\|,\nonumber\\
&\bigl|2\Rea(\overline{x}u)-(-1)^{k+1}\bigr|= 
\tfrac{2}{1+\tau}
\bigl|(1+\tau)\Rea(\overline{x}u)-(-1)^{k+1}+(-1)^{k+1}\tfrac{1-\tau}{2}\bigr|\leq 30 \|E\|,\\
&\bigl|2\Rea(\overline{y}v)-(-1)^{k}\bigr|= 
\tfrac{2}{1+\tau}
\bigl|(1+\tau)\Rea(\overline{y}v)-(-1)^{k}+(-1)^{k}\tfrac{1-\tau}{2}\bigr|\leq 30 \|E\|.\nonumber
\end{align}
It gives (C\ref{r7}).
The first line of (\ref{equal10}) is valid also 
for $\alpha\in \{0,1\}$, $\beta=\omega=0$ 
(see (\ref{xbarv})).
If $\alpha=1$, then (\ref{eq01t2}) for $\tau=1$ yields
$2c\Rea(\overline{x}u)=1+\epsilon_1$. By applying (\ref{ocenah}) for $\|E\|\leq \frac{1}{2}$ we get $c=(-1)^{k}e^{i\psi}$, $k\in \mathbb{Z}$, $\psi\in (-\frac{\pi}{2},\frac{\pi}{2})$, $|\sin \psi|\leq 2\|E\|$. Moreover, $\big|c-(-1)^{k}\big|=2|\sin \frac{\psi}{2}|\leq 4\|E\|$. To conclude, (\ref{eq01t2}), (\ref{xyuvt2}), (\ref{xbarv}) provides (C\ref{r4}).
\end{enumerate}

\end{enumerate}
 
This completes the proof of the lemma.
\end{proof}

\vspace{-4mm}

\section{Proof of Theorem \ref{izrek}}\label{proofT}

\vspace{-1mm}

To prove the nonexistence of some paths in the closure graph for bundles under (\ref{aAB}), the proof of \cite[Theorem 3.6]{TS2} (the closure graph for orbits) 
applies mutatis mutandis; we shall not rewrite the proof in these cases, instead we refer to \cite{TS2} for the proof.
However, we reprove the existence of paths for bundles consisting of one orbit, since short and plausible arguments can be given (see e.g. (\ref{cPepsi})).

\vspace{-1mm}

\begin{proof}[Proof of Theorem \ref{izrek}.]
Given normal forms $(\widetilde{A},\widetilde{B})$, $(A,B)$ from Lemma \ref{lemalist} the existence of a path $(\widetilde{A},\widetilde{B})\to (A,B)$ in the closure graph for bundles for the action (\ref{aAB}) immediately implies
%
$\widetilde{A}\to A$, $ \widetilde{B}\to B$. 
%
When this is not fulfilled, then $(\widetilde{A},\widetilde{B})\not \to (A,B)$ and we already have a lower estimate on the distance from $(\widetilde{A},\widetilde{B})$ to the bundle of $(A,B)$ (see Lemma \ref{lemapsi2}, Lemma \ref{lemapsi1}). Further, $(\widetilde{A},0_2)\to (A,0_2)$ (or $(0_2,\widetilde{B})\to (0_2,B)$) if and only if $\widetilde{A}\to A$ (or $\widetilde{B}\to B$), and trivially $(A,B)\to (A,B)$ for any $A,B$.

From now on suppose 
$\widetilde{A}\to A$, $ \widetilde{B}\to B\neq 0$ with 
$(\widetilde{A},\widetilde{B})\not \in \Bun_{\Psi}(A,B)$ 
and let
\begin{equation}\label{eqABEF}
cP^{*}AP=\widetilde{A}+E, \quad P^{T}BP=\widetilde{B}+F, \qquad c\in S^{1},\,P\in GL_2(\mathbb{C}),\quad E,F\in \mathbb{C}^{2\times 2}.
\end{equation}
%
Due to Lemma \ref{lemapsi1} and Lemma \ref{lemadet} (\ref{lemadeta}) the first equation of (\ref{eqABEF}) yields restrictions on $P$, $c$, $A$ imposed by $\|E\|$, $\widetilde{A}$. The trick of the proof is to use these to analyse the second equation of (\ref{eqABEF}). 
We now work with equations with larger set of parameters than in \cite[Theorem 3.6]{TS2}, and it usually makes the analysis more involved.
If it eventually leads to 
an inequality that fails for any sufficiently small $E$ and $F$, it will prove $(\widetilde{A},\widetilde{B})\not \to(A,B)$; it is then straightforward to estimate how small $E$, $F$ should be, thus we omit this calculation.
%
Otherwise, to prove $(\widetilde{A},\widetilde{B})\to (A,B)$, we find $c(s)\in S^{1}$, $P(s)\in GL_2(\mathbb{C})$, $(A(s),B(s))\in \Bun (A,B)$ such that
\vspace{-1mm}
\begin{equation}\label{cPepsi}
c(s) \bigl(P(s)\bigr)^*A(s)P(s)-\widetilde{A}=:E(s)\stackrel{s\to 0}{\longrightarrow} 0, \quad \bigl(P(s)\bigr)^TB(s)P(s)-\widetilde{B}=:F(s)\stackrel{s\to 0}{\longrightarrow} 0. 
\end{equation}
%
When we can arrange the parameter $s$ so that $A(s)\to \widetilde{A}$ and $B(s)\to \widetilde{B}$, this is trivial.

Throughout the proof we denote $\delta=\nu\sqrt{\|E\|}$ for $\nu>0$ (Lemma \ref{lemapsi1} (\ref{lemapsi11})), $\epsilon=\|F\|$,
\begin{equation*}
B=\begin{bmatrix}
a & b\\
b & d
\end{bmatrix},\quad 
\widetilde{B}=\begin{bmatrix}
\widetilde{a} & \widetilde{b}\\
\widetilde{b} & \widetilde{d}
\end{bmatrix},\qquad 
F=\begin{bmatrix}
\epsilon_1 & \epsilon_2\\
\epsilon_2 & \epsilon_4
\end{bmatrix},
\qquad
P=\begin{bmatrix}
x & y\\
u & v
\end{bmatrix},
\end{equation*}
where sometimes polar coordinates for $x,y,u,v$ in $P$ might be preferred:
\begin{equation}\label{polarL}
x=|x|^{i\phi},\quad y=|y|e^{i\varphi},\quad u=|u|e^{i\eta},\quad v=|v|e^{i\kappa}, \qquad \phi,\varphi,\eta,\kappa\in \mathbb{R}.
\end{equation}
The second matrix equation of (\ref{eqABEF}) can thus be written componentwise as:
\begin{align}\label{eqBF1}
&ax^2+2bux+du^2=\widetilde{a}+\epsilon_1,\nonumber \\
&axy+buy+bvx+duv=\widetilde{b}+\epsilon_2,\\ 
&ay^2+2bvy+dv^2=\widetilde{d}+\epsilon_4. \nonumber
\end{align}
For the sake of simplicity some estimates in the proof are crude, and it is always assumed $\epsilon,\delta\leq \frac{1}{2}$. Since we shall often apply Lemma \ref{lemadet}, we take for granted that $(\frac{\delta}{\nu})^{2}=\|E\|\leq \min\{1,\frac{|\det\widetilde{A}|}{8\|\widetilde{A}\|+4}\}$, $\epsilon=\|F\|\leq \tfrac{|\det \widetilde{B}|}{4\|\widetilde{B}\|+2}$. If $A,\widetilde{A}$ are nonsingular we also assume $\|E\|\leq \|\widetilde{A}^{-1}\|^{-1}$, while for $B,\widetilde{B}$ nonsingular with $1=|\det A|=|\det \widetilde{A}|=\|\widetilde{A}\|$ it is assumed $\|E\|,\|F\|\leq \frac{|\det \widetilde{B}|}{4(4\max\{1,\|\widetilde{B}\|,|\det \widetilde{B}|\}+2)^{2}}$.

We split our analysis
into several cases (see Lemma \ref{lemalist} for normal forms).
The notation $(\widetilde{A},\widetilde{B})\dashrightarrow (A,B)$ is used when the existence of a path is yet to be considered. 

\begin{enumerate}[label={\bf Case \Roman*.},ref={Case \Roman*},wide=0pt,itemsep=10pt]

\item \label{p1t1t}
$
(1\oplus e^{i\theta},\widetilde{B})\dashrightarrow
(1\oplus e^{i\theta},B)
$, \quad $0\leq \widetilde{\theta}\leq\pi$, $0\leq\theta\leq\pi$ 

\begin{enumerate}[label=(\alph*),wide=0pt,itemindent=2em,itemsep=6pt]

\item \label{p1t1ta}  $0<\widetilde{\theta},\theta<\pi$

From Lemma \ref{lemapsi1} (\ref{lemapsi11}) for (C\ref{r1}) we get 
\begin{equation}\label{thth}
|y|^{2},|u|^{2}\leq \delta, \qquad \bigl||v|^{2}-1\bigr|,\, \bigl||x|^{2}-1\bigr|\leq \delta.
\end{equation}
%
%

\begin{enumerate}[label=(\roman*),wide=0pt,itemindent=4em,itemsep=3pt]

\item \label{p1t1tai}
$B=\begin{bsmallmatrix}
a & b\\
b & 0
\end{bsmallmatrix}$, $b,a\geq 0$

Using (\ref{thth}) and Lemma \ref{lemapsi2} (D\ref{p3ba}) we get a contradiction for small $\epsilon,\delta$ and $\widetilde{d}\neq 0$.
Next, 
Lemma \ref{lemadet} (\ref{pathlema2ii}) for $\widetilde{d}=0$ gives
$b^{2}=|\widetilde{b}|^{2}+\delta_5$, $|\delta_5|\leq \max\{\epsilon,\frac{\delta^{2}}{\nu^{2}}\}\big(4\max\{1,|\widetilde{b}|^{2},|\widetilde{b}|\}+2\big)^{2}$.
It fails for $b=0$, $\widetilde{b}\neq 0$ and  
$\epsilon,\frac{\delta^{2}}{\nu^{2}}<\widetilde{b}^{2}\big(4\max\{1,|\widetilde{b}|^{2},|\widetilde{b}|\}+2\big)^{-2}$, while the case $b=\widetilde{b}=0$ is trivial.
For $a=0$, $\widetilde{a}\neq 0$ then the first equation of (\ref{eqBF1}) for $a=d=0$ and (\ref{thth}) yields an inequality that fails for $\epsilon,\delta$ small enough:
\begin{align*}
&|\widetilde{a}|=|\epsilon_1-2bux|\leq \epsilon+2(\widetilde{b}+\delta_5)\sqrt{\delta(1+\delta)}.
\end{align*}

\item 
$B=\begin{bsmallmatrix}
0 & b\\
b & d
\end{bsmallmatrix}$, $b\geq 0$, $d\neq 0$\\
Due to a symmetry we deal with this case similarly as with \ref{p1t1t} \ref{p1t1ta} \ref{p1t1tai}.

\item \label{p1t1taii} $B=a\oplus d$, \quad $a,d>0$

From (\ref{eqBF1}) for $b=0$ we obtain
\begin{align}\label{eqBFadad}
&ax^2+du^2=\widetilde{a}+\epsilon_1, \nonumber\\
&axy+duv=\widetilde{b}+\epsilon_2, \\
&ay^2+dv^2=\widetilde{d}+\epsilon_4.\nonumber
\end{align}

By multiplying the first and the last equation of (\ref{eqBFadad}) by $\delta_6=\frac{y}{x}$ and $\delta_5=\frac{u}{v}$, respectively, and by slightly simplifying them, we get
\[
axy+duv\delta_5\delta_6=\delta_6(\widetilde{a}+\epsilon_1), \qquad  axy\delta_5\delta_6+duv=\delta_5(\widetilde{d}+\epsilon_4).
\]
Adding these two equations and using the second equation of (\ref{eqBFadad}) we deduce 
\[
(\widetilde{b}+\epsilon_2)(1+\delta_5\delta_6)=\delta_5(\widetilde{d}+\epsilon_4)+\delta_6(\widetilde{a}+\epsilon_1),
\]
which fails for $\widetilde{b}\neq 0$ and sufficiently small $\epsilon,\delta$  (by (\ref{thth}) we have $|\delta_5|,|\delta_6|\leq \frac{\delta}{1-\delta}$ ).


\end{enumerate}

\item \label{p1t1tbb} $\widetilde{\theta}\in \{0,\pi\}$

Set $\sigma=e^{i\widetilde{\theta}}\in\{1,-1\}$. Lemma \ref{lemapsi1}  (C\ref{r99}) yields
\begin{equation}\label{case1s1t}
|x|^2+\sigma|u|^2=(-1)^{k}+\delta_1,\quad \overline{x}y+\sigma \overline{u} v=\delta_2, \quad |y|^2+\sigma |v|^2=\sigma (-1)^{k}+\delta_4, 
\end{equation}
where $|\delta_1|,|\delta_2|,|\delta_4|\leq \delta$.  
Next, for $v\neq 0$,
$\big(|x|-|u|\big)^{2}\leq \big||x|^2-|u|^2\big|=: 1+\delta_1'$
we deduce
\begin{align}\label{ocenaxysuv}
|\overline{x}y+\sigma\overline{u} v|
&\geq 
\big||\overline{x}y|-|\overline{x}v|+|\overline{x}v|-|\overline{u} v|\big|
 \geq |v|\big|\overline{x}|-|\overline{u} |\big|-|\overline{x}|\big||y|-|v|\big| \\
&\geq \tfrac{\big|\overline{x}|^{2}-|\overline{u} |^{2}\big|}{\frac{1}{|v|}(|\overline{x}|+|\overline{u} |)}-\tfrac{|x|\big|\overline{y}|^{2}-|\overline{v} |^{2}\big|}{|\overline{y}|+|\overline{v} |}\nonumber
\geq \tfrac{1-|\delta_1'|}{2\frac{|u|}{|v|}+\frac{\sqrt{1+|\delta_1'|}}{|v|}}-\big(\tfrac{|u|}{|v |}+\tfrac{\sqrt{1+|\delta_1'|}}{|v|}\big)\big||\overline{y}|^{2}-|\overline{v} |^{2}\big|.\nonumber
\end{align}

\vspace{-2mm}


\begin{enumerate}[label=(\roman*),wide=0pt,itemindent=4em,itemsep=3pt]

\item \label{p1t1tbbi} 
$B=\begin{bsmallmatrix}
a & b\\
b & d
\end{bsmallmatrix}$, $a,d,b\geq 0$, $a+d\neq 0$

Let first $B=a\oplus d$.
Using the notation (\ref{polarL}) the following calculation is trivial:
\begin{align}\label{exu2v2}
&ax^{2}+du^{2}= ae^{2i\phi}\big(|x|^{2}+\sigma |u|^{2}\big)-u^{2} (\sigma a e^{2i(\phi-\eta)}-d), \qquad \sigma\in \{-1,1\},\nonumber\\
&ay^{2}+dv^{2}=a e^{2i\varphi}\big(|y|^{2}+\sigma |v|^{2}\big)-v^{2} (\sigma ae^{2i(\varphi-\kappa)}-d),\\
&ay^{2}+dv^{2}=d\sigma e^{2i\kappa}\big(|y|^{2}+\sigma |v|^{2}\big)-y^{2} (\sigma de^{2i(\kappa-\varphi)}-a)\nonumber.
\end{align}
Furthermore, one easily writes:
\begin{align}\label{exuv}
a xy+d uv &=ae^{2i\phi}(\overline{x}y+\sigma\overline{u}v)-uv(\sigma a e^{2i(\phi-\eta)}-d), \qquad \sigma\in \{-1,1\},\nonumber\\
a xy+d uv&=ae^{2i\varphi}(x\overline{y}+\sigma u\overline{v})-uv(\sigma a e^{2i(\varphi-\kappa)}-d),\\
a xy+d uv&=d\sigma e^{2i\kappa}(x\overline{y}+\sigma u\overline{v})-xy(d\sigma e^{2i(\kappa-\varphi)}-a).\nonumber
\end{align}
Rearranging the terms in (\ref{exu2v2}), (\ref{exuv}) and using (\ref{eqBFadad}), (\ref{case1s1t}) yields for $\sigma\in \{-1,1\}$:
\small
\begin{align*}
& u^{2}(\sigma a e^{2i(\phi-\eta)}-d)= ae^{2i\phi}((-1)^{k}+ \delta_1)-\widetilde{a}-\epsilon_1, 
&  uv(\sigma a e^{2i(\phi-\eta)}-d)=ae^{2i\phi} \delta_2-\widetilde{b}- \epsilon_2, \nonumber \\
&v^{2}(\sigma a e^{2i(\varphi-\kappa)}-d)= ae^{2i\varphi}(\sigma (-1)^{k}+ \delta_4)-\widetilde{d}-\epsilon_4,
& uv(\sigma a e^{2i(\varphi-\kappa)}-d)= ae^{2i\varphi}\overline{\delta}_2-\widetilde{b}-\epsilon_2,\nonumber\\
&y^{2}(\sigma d e^{2i(\kappa-\varphi)}-a)= d\sigma e^{2i\kappa}(\sigma (-1)^{k}+ \delta_4)-\widetilde{d}-\epsilon_4,\nonumber
& xy(d\sigma  e^{2i(\kappa-\varphi)}-a)= d\sigma\overline{\delta}_2-\widetilde{b}-\epsilon_2\nonumber.
\end{align*}
\normalsize
By dividing the equations in each line we get
\begin{align}\label{fuvvu}
\tfrac{u}{v}=\tfrac{ae^{2i\phi}((-1)^{k}+ \delta_1)-\widetilde{a}-\epsilon_1}{ae^{2i\phi} \delta_2-\widetilde{b}- \epsilon_2}
=\tfrac{ ae^{2i\varphi}\overline{\delta}_2-\widetilde{b}-\epsilon_2}{ae^{2i\varphi}(\sigma (-1)^{k}+ \delta_4)-\widetilde{d}-\epsilon_4},\qquad
\tfrac{x}{y}
=\tfrac{d\sigma\overline{\delta}_2-\widetilde{b}-\epsilon_2}{d\sigma e^{2i\kappa}(\sigma (-1)^{k}+ \delta_4)-\widetilde{d}-\epsilon_4}.
\end{align}
If $\widetilde{B}=\begin{bsmallmatrix}
0 & \widetilde{b}\\
\widetilde{b} & 0
\end{bsmallmatrix}$ for $\widetilde{b}>0$ (hence $\sigma=-1$, $d\geq a> 0$), then Lemma \ref{lemadet} (\ref{pathlema2ii}) implies $a^{2}\leq ad=\widetilde{b}^{2}+\epsilon'$, $|\epsilon'|\leq \max \{\epsilon,\frac{\delta^{2}}{\nu^{2}}\}\big(4\max\{1,|\widetilde{b}|,|\widetilde{b}|^{2}\}+2\big)^{2}$. From the first equation of (\ref{fuvvu}) for $\widetilde{d}=\widetilde{a}=0$, $\sigma=1$ we now obtain a contradiction for small $\epsilon,\delta$. 
Similarly, it follows from Lemma \ref{lemadet} (\ref{pathlema2ii}) for $\widetilde{B}=\widetilde{a}\oplus \widetilde{d}$ and $B=aI_2$ that $a^{2}=\widetilde{a}\widetilde{d}+\epsilon'$, $|\epsilon'|\leq \max \{\epsilon,\frac{\delta^{2}}{\nu^{2}}\}(4\max\{1,|\widetilde{d}|,|\widetilde{d}\widetilde{a}|\}+2)^{2}$. If $\widetilde{d}>\widetilde{a}> 0$, then the first equation of (\ref{fuvvu}) (with $\sigma\in \{-1,1\}$, $\widetilde{b}=0$) fails as well.   
Next, when $\widetilde{a}=0$ we have $a^{2}=\epsilon'$. Hence (\ref{fuvvu}) for $\sigma\in \{-1,1\}$, $\widetilde{a}=\widetilde{b}=0$ yields
$|\frac{u}{v}|,|\frac{x}{y}|\leq  \frac{\epsilon+\sqrt{\epsilon'}\delta}{\widetilde{d}-\epsilon-\sqrt{\epsilon'}(1+\delta)}$. Further, 
the third equation of (\ref{exu2v2}) with (\ref{eqBFadad}), (\ref{case1s1t}) for $a=d=\sqrt{\epsilon'}$, $\widetilde{b}=0$ gives $\frac{1}{|v|^{2}}\leq \frac{2\epsilon'}{\widetilde{d}-\epsilon-\epsilon'(1+\delta)}$.
We apply this and (\ref{case1s1t}) to (\ref{ocenaxysuv}) to deduce a contradiction for small $\epsilon,\delta$ and $\widetilde{d}\neq 0$.

\quad
Take 
$
P(s)=\frac{1}{\sqrt{\widetilde{d}+\sigma \widetilde{a}}}
\begin{bsmallmatrix}
-i\sqrt{\widetilde{d}} & \sqrt{\widetilde{a}}\\
i\sqrt{\widetilde{a}} & \sigma \sqrt{\widetilde{d}}
\end{bsmallmatrix}
$, 
$B(s)=
\begin{bsmallmatrix}
0 & \sqrt{\widetilde{a}\widetilde{d}}+s\\
\sqrt{\widetilde{a}\widetilde{d}}+s & \widetilde{d}-\sigma \widetilde{a}+s
\end{bsmallmatrix}
$, $c(s)=1$, $e^{i\theta}\to \sigma$
in (\ref{cPepsi}) to see
$
\big(1\oplus \sigma,
\widetilde{a}\oplus \widetilde{d}\big)
\to 
\big(
1\oplus e^{i\theta},
\begin{bsmallmatrix}
0 & b\\
b & d
\end{bsmallmatrix}\big)$, and
$P(s)=\frac{1}{\sqrt{\widetilde{d}+\sigma \widetilde{a}}}
\begin{bsmallmatrix}
i\sqrt{\widetilde{a}} & \sigma \sqrt{\widetilde{d}}\\
-i\sqrt{\widetilde{d}} & \sqrt{\widetilde{a}}
\end{bsmallmatrix}
$, 
$B(s)=
\begin{bsmallmatrix}
\widetilde{d}-\sigma \widetilde{a}+s & \sqrt{\widetilde{a}\widetilde{d}}+s\\
\sqrt{\widetilde{a}\widetilde{d}}+s & 0
\end{bsmallmatrix}
$, 
$c(s)=\sigma$, 
$e^{i\theta}\to \sigma$ to show
$
\big(1\oplus \sigma,
\widetilde{a}\oplus \widetilde{d}\big)
\to 
\big(
1\oplus e^{i\theta},
\begin{bsmallmatrix}
a & b\\
b & 0
\end{bsmallmatrix}\big)$, $0<\theta<\pi$. 

\item \label{p1t1tbbbb} 
$B=\begin{bsmallmatrix}
0 & b\\
b & 0
\end{bsmallmatrix}$, $b>0$

From (\ref{eqBF1}) for $a=d=0$ we obtain
\begin{align}\label{eqadb}
&2bux=\widetilde{a}+\epsilon_1, \nonumber\\
&buy+bvx=\widetilde{b}+\epsilon_2,\\
&2bvy=\widetilde{d}+\epsilon_4. \nonumber
\end{align}
It suffices to take $0\leq\widetilde{a}\leq \widetilde{d}$, $\widetilde{d}> 0$, $\widetilde{b}=0$. 
By Lemma \ref{lemadet} (\ref{pathlema2ii}) and (\ref{ocenakoren}) we have 
$b=\sqrt{\widetilde{a}\widetilde{d}}+\delta_5>0$, 
$|\delta_5|\leq \max\{\epsilon,\frac{\delta^{2}}{\nu^{2}}\}(4\max\{1,\widetilde{d},\widetilde{a}\widetilde{d}\}+2)^{2}$.
Thus (\ref{eqadb}) 
and (\ref{case1s1t})
give:
\begin{equation}\label{ocenavya0}
|v|^{2},|y|^{2}\leq \tfrac{\widetilde{d}+\epsilon}{2(\sqrt{\widetilde{a}\widetilde{d}}+\delta_5)}+1+\delta,\quad
|u|^{2},|x|^{2}\leq \tfrac{\widetilde{a}+\epsilon}{2(\sqrt{\widetilde{a}\widetilde{d}}+\delta_5)}+1+\delta .
\end{equation}
Using Lemma \ref{lemapsi2} (D\ref{p3bd}), (D\ref{p3ba}) for $\det \widetilde{B}=\widetilde{a}\widetilde{d}$ we get
\begin{equation}\label{eqabd}
\begin{array}{ll}
u(\widetilde{d}+\epsilon_4)=v\big(-i(-1)^{l}\sqrt{\widetilde{a}\widetilde{d}}+\epsilon_2''\big)\\
x(\widetilde{d}+\epsilon_4)=y\big(i(-1)^{l}\sqrt{\widetilde{a}\widetilde{d}}+\epsilon_2'\big)
\end{array},
|\epsilon_2'|,|\epsilon_2''|\leq \left\{\begin{array}{ll}
\tfrac{\epsilon (4\max\{\widetilde{d},\widetilde{a}\}+2+\widetilde{d}\widetilde{a})}{\widetilde{d}\widetilde{a}}, & \widetilde{a}\widetilde{d}\neq 0\\
\scriptstyle{\sqrt{\epsilon} (4\max\{\widetilde{d},\widetilde{a}\}+3)^{\frac{1}{2}}}, & \widetilde{a}\widetilde{d}= 0
\end{array}\right..\\
\end{equation}
%
%
%
By further applying the first and the third equality of (\ref{case1s1t}) we deduce
\begin{align}\label{x2u2ad}
(-1)^{k}+\delta_1=&|x|^2+\sigma|u|^2=\tfrac{|(-1)^{l}\sqrt{\widetilde{a}\widetilde{d}}+\epsilon_2'|^{2}}{|\widetilde{d}+\epsilon_4|^{2}}|y|^{2}+\sigma\tfrac{|-(-1)^{l}\sqrt{\widetilde{a}\widetilde{d}}+\epsilon_2''|^{2}}{|\widetilde{d}+\epsilon_4|^{2}}|v|^{2}=\\
                 =&\tfrac{|-(-1)^{l}\sqrt{\widetilde{a}\widetilde{d}}+\epsilon_2'|^{2}}{|\widetilde{d}+\epsilon_4|^{2}}\big(\sigma(-1)^{k}+\delta_4\big)+\delta'|v|^{2}\nonumber
\end{align}
with $|\delta'|\leq C\max\{\epsilon,\delta\}$, where $C>0$ is a constant that can be computed easily. 

\quad
By combining (\ref{x2u2ad}) and (\ref{ocenavya0}) we obtain a contradiction for $0<\widetilde{a}< \widetilde{d}$ and sufficiently small $\epsilon, \delta$.
Next, let $\widetilde{a}=0$, $\widetilde{d}>0$. 
From (\ref{eqabd}) 
it follows $|\frac{u}{v}|\leq \frac{|\epsilon_2'|}{|\widetilde{d}|-\epsilon}$, $|\frac{x}{y}|\leq \frac{|\epsilon_2''|}{|\widetilde{d}|-\epsilon}$ 
($y=0$ or $v=0$ would contradict (\ref{eqadb}) for $|\widetilde{d}|> \epsilon$). By applying this with (\ref{case1s1t}) and (\ref{x2u2ad}) (hence $|v|$ is large) to  (\ref{ocenaxysuv}), we obtain a contradiction for small $\epsilon, \delta$.

\quad
We take 
$
P(s)=\frac{1}{\sqrt{2}}
\begin{bsmallmatrix}
1 & i\\
1 & -i
\end{bsmallmatrix}
$, 
$B(s)=(\widetilde{d}+s)
\begin{bsmallmatrix}
0 & 1\\
1 & 0
\end{bsmallmatrix}
$ and $c(s)=1$, $e^{i\theta}\to 1$
in (\ref{cPepsi}) to prove
$
\big(I_2,
\widetilde{d}I_2\big)
\to 
\big(
1\oplus e^{i\theta},
\begin{bsmallmatrix}
0 & b\\
b & 0
\end{bsmallmatrix}\big)$, $b>0$, $0\leq \widetilde{d}$, $0<\theta<\pi$.
%
%
%
Using (\ref{case1s1t}) for $\sigma=-1$ leads to
\begin{equation}\label{eq88}
\delta\geq |\overline{x}y-\overline{u}v|\geq \big||x|^{2}|\tfrac{y}{x}|-|u|^{2}|\tfrac{v}{u}|\big|\geq \big||x|^{2}-|u|^{2}\big|-|x|^{2}\big(1-|\tfrac{y}{x}|\big)-|u|^{2}\big(1-|\tfrac{v}{u}|\big).
\end{equation}
From (\ref{eqabd}) we get that $|\frac{y}{x}|,|\frac{v}{u}|$ are close to $1$, and (\ref{ocenavya0}) implies that $|u|^{2},|x|^{2}$ are bounded.
Thus the last two terms on the right-hand side of (\ref{eq88}) are small, while the first one is close to $1$ (see (\ref{case1s1t}) for $\sigma=-1$). For small $\epsilon,\delta$ we get a contradiction. 

\end{enumerate}

\end{enumerate}


\item \label{p01t001t0}
$
\bigl(\begin{bsmallmatrix}
0 & 1\\
\widetilde{\tau} & 0 
\end{bsmallmatrix},
\begin{bsmallmatrix}
\widetilde{a} & \widetilde{b}\\
\widetilde{b} & \widetilde{d}
\end{bsmallmatrix}\bigr)\dashrightarrow
\bigl(\begin{bsmallmatrix}
0 & 1\\
\tau & 0
\end{bsmallmatrix},
\begin{bsmallmatrix}
a & b\\
b & d
\end{bsmallmatrix}\bigr)
$, \quad $\widetilde{b},b \geq 0$, $(\widetilde{\tau},\tau)\in \big([0,1)\times (0,1)\big)\cup \{(0,0)\}$ 

\quad
By Lemma \ref{lemapsi1} (\ref{lemapsi11}) for (C\ref{r6}) we have
\begin{equation}\label{1txv}
|xu|,|yu|,|vy|\leq \delta, \qquad \big||vx|-1\big|\leq \delta.
\end{equation}
%
It yields $\delta_6=\frac{y}{x}=\frac{yv}{xv}$ with $|\delta_6|\leq\frac{\delta}{1-\delta}\leq 2\delta$, $\delta_5=\frac{u}{v}=\frac{ux}{xv}$ with $|\delta_5|\leq\frac{\delta}{1-\delta}\leq 2\delta$ and 
$\delta_7=\frac{uy}{vx}$ with $|\delta_7|\leq 2\delta$ (note $\delta\leq \frac{1}{2}$).

\begin{enumerate}[label=(\alph*),wide=0pt,itemindent=2em,itemsep=6pt]

\item \label{p01t0prva}
$B=\begin{bsmallmatrix}
0 & b\\
b & d
\end{bsmallmatrix}$, $b\geq 0$, $|d|\in \{0,1\}$, $|b|+|d|\neq 0$

%
%
By multiplying the last two equations of (\ref{eqBFbasic1}) by $\delta_5=\frac{u}{v}$ and using $\delta_7=\frac{uy}{vx}$ we get 
\begin{align}\label{eq01t0druga}
du^{2}+(1+\delta_7)bux=(\widetilde{b}+\epsilon_2)\delta_5, \qquad
2\delta_7 bvx+dvu=(\widetilde{d}+\epsilon_4)\delta_5. 
\end{align}
Subtracting the first and the second equation of (\ref{eq01t0druga}) from the first and the second equation of (\ref{eqBFbasic1}) (in the form $duv+b(1+\delta_7)vx=\widetilde{b}+\epsilon_2$), we deduce
\begin{align}\label{eq01t0tretja}
(1-\delta_7)bux=\widetilde{a}+\epsilon_1-(\widetilde{b}+\epsilon_2)\delta_5,\quad 
(1-\delta_7)bvx=\widetilde{b}+\epsilon_2-(\widetilde{d}+\epsilon_4)\delta_5.
\end{align}
%
It is clear that the first (the second) equality in (\ref{eq01t0tretja}) fails for $\widetilde{a}\neq 0$ (for $\widetilde{b}\neq 0$) and $b=0$, provided that $\epsilon,\delta$ are sufficiently small.
Next, from the second equation of (\ref{eq01t0tretja}) and using $vx=e^{i\vartheta}-\delta_0$ with $|\delta_0|\leq \delta$, $\vartheta\in \mathbb{R}$ (see (\ref{1txv})) we obtain
\begin{align}\label{eqabwb}
b=\tfrac{\widetilde{b}+\epsilon_2-(\widetilde{d}+\epsilon_4)\delta_5}{(1-\delta_7)(e^{i\vartheta}-\delta_0)}=e^{-i\vartheta}\widetilde{b}+\tfrac{e^{-i\vartheta}\widetilde{b}(\delta_7+e^{i\vartheta}\delta_0-\delta_0\delta_7)+\epsilon_2-(\widetilde{d}+\epsilon_4)\delta_5}{(1-\delta_7)(e^{i\vartheta}-\delta_0)}
\end{align}
From (\ref{eqabwb}) and $|ux|\leq \delta$ (and $|yv|\leq \delta$) we get that the first equation of (\ref{eq01t0tretja}) fails for $\widetilde{a}\neq 0$ (the last equation of (\ref{eqBFbasic1}) fails for $\widetilde{d}\neq 0$, $d=0$), and $\epsilon,\delta$ small enough. 

\quad
Finally, it is easy to check that 
$
P(s)=
\begin{bsmallmatrix}
s^{-1} & 0\\
s^2 & s
\end{bsmallmatrix}
$, 
$B(s)=
\begin{bsmallmatrix}
0 & b(s)\\
b(s) & d
\end{bsmallmatrix}
$ with $b(s)\to \widetilde{b}$, 
$A(s)=
\begin{bsmallmatrix}
0 & 1\\
\widetilde{\tau}+s & 0
\end{bsmallmatrix}
$, $c(s)=1$
in (\ref{cPepsi}) proves
$
\big(\begin{bsmallmatrix}
0 & 1\\
\widetilde{\tau} & 0
\end{bsmallmatrix},
\begin{bsmallmatrix}
0 & \widetilde{b}\\
\widetilde{b} & 0
\end{bsmallmatrix}\big)
\to 
\big(
\begin{bsmallmatrix}
0 & 1\\
\tau & 0
\end{bsmallmatrix},
\begin{bsmallmatrix}
0 & b\\
b & d
\end{bsmallmatrix}\big)$, $b\geq \widetilde{b}\geq 0$.

\vspace{-2mm}

\item \label{p01t0druga}
$B=\begin{bsmallmatrix}
1 & b\\
b & 0
\end{bsmallmatrix}$, $b\geq 0$, $\tau=0$ 

We argue similarly as in \ref{p01t001t0} \ref{p01t0prva}. We have equations (\ref{eqBF5obrat});
%
%
by multiplying the first two equations by $\delta_6=\frac{y}{x}$ and using $\delta_7=\frac{uy}{vx}$ we obtain 
\begin{align}\label{eq01t0druga22}
ay^{2}+(1+\delta_7)bvy=(\widetilde{b}+\epsilon_2)\delta_6, \qquad
2\delta_7 bvx+axy=(\widetilde{a}+\epsilon_1)\delta_6. 
\end{align}
Subtracting the first and the second equation of (\ref{eq01t0druga22}) from the last and the second equation of (\ref{eqBFbasic1}) (written as $axy+b(1+\delta_7)vx=\widetilde{b}+\epsilon_2$), respectively, we get
\begin{align}\label{eq01t0tretja33}
(1-\delta_7)bvy=\widetilde{d}+\epsilon_4-(\widetilde{b}+\epsilon_2)\delta_5,\quad 
(1-\delta_7)bvx=\widetilde{b}+\epsilon_2-(\widetilde{a}+\epsilon_1)\delta_6.
\end{align}
%
The first (the second) equality in (\ref{eq01t0tretja33}) fails for $\widetilde{d}\neq 0$ (for $\widetilde{b}\neq 0$) and $b=0$, provided that $\epsilon,\delta$ are sufficiently small.
We obtain a similar expression for $b$ as in (\ref{eqabwb}). It yields a contradiction for $b=0$, $\widetilde{b}\neq 0$ and $\delta,\epsilon$ small enough, while by combining it with $|yv|\leq \delta$ (and $|ux|\leq \delta$) we contradict the first equation of (\ref{eq01t0tretja33}) for $\widetilde{d}\neq 0$ (or (\ref{eqBF5obrat}) for $\widetilde{a}\neq 0$, $a=0$), provided that $\epsilon,\delta$ are small. 
Take 
$
P(s)=
\begin{bsmallmatrix}
s & s^2\\
0 & s^{-1}
\end{bsmallmatrix}
$,
$B(s)=
\begin{bsmallmatrix}
1 & b(s)\\
b(s) & 0
\end{bsmallmatrix}
$, $b(s)\to \widetilde{b}$, $c(s)=1$
in (\ref{cPepsi}) to prove
$
\big(\begin{bsmallmatrix}
0 & 1\\
0 & 0
\end{bsmallmatrix},
\begin{bsmallmatrix}
0 & \widetilde{b}\\
\widetilde{b} & 0
\end{bsmallmatrix}\big)
\to 
\big(
\begin{bsmallmatrix}
0 & 1\\
0 & 0
\end{bsmallmatrix},
\begin{bsmallmatrix}
1 & b\\
b & 0
\end{bsmallmatrix}\big)$, $b\geq \widetilde{b}$.

\item \label{p01t0tretja}
$B=1\oplus d$, $d\in\mathbb{C}$ ($0<\tau<1$) or $B=a\oplus 1$, $a>0$ ($\tau=0$)

%
%
Since $|\frac{y}{x}|\leq \frac{\delta}{1-\delta}$ and $|\frac{u}{v}|\leq \frac{\delta}{1-\delta}$ the same proof as in \ref {p1t1t} \ref{p1t1ta} \ref{p1t1taii} applies.

\quad
From (\ref{cPepsi}) for 
$
P(s)=
\begin{bsmallmatrix}
s & s^2\\
0 & s^{-1}
\end{bsmallmatrix}
$,
$B(s)=
1\oplus s^{2}\widetilde{d}
$ 
and $
P(s)=
\sqrt{\widetilde{a}}\oplus \frac{1}{\sqrt{\widetilde{a}}}
$,
$B(s)=
1\oplus \widetilde{a}\widetilde{d}
$ 
with $\tau\to\widetilde{\tau}$, $c(s)=1$,
in (\ref{cPepsi}) we obtain 
$
\big(\begin{bsmallmatrix}
0 & 1\\
\widetilde{\tau} & 0
\end{bsmallmatrix},
\widetilde{a}\oplus \widetilde{d}\big)
\to 
\big(
\begin{bsmallmatrix}
0 & 1\\
\tau & 0
\end{bsmallmatrix},
1\oplus d\big)$ (with $0<\tau<1$) for $\widetilde{a}=0$ and $\widetilde{a}>0$, respectively.
Finally,
$
P(s)=
\begin{bsmallmatrix}
s^{-1} & 1\\
s^2 & s
\end{bsmallmatrix}
$,
$B(s)=
(\widetilde{d}s^{2}+s^{3})\oplus 1
$ 
with $c(s)=1$ gives 
$
\big(\begin{bsmallmatrix}
0 & 1\\
0 & 0
\end{bsmallmatrix},
\widetilde{d}\oplus 0\big)
\to 
\big(
\begin{bsmallmatrix}
0 & 1\\
0 & 0
\end{bsmallmatrix},
a\oplus 1\big)$, $a>0$, $\widetilde{d}\in \{0,1\}$.

\vspace{-2mm}

\item
$B=\begin{bsmallmatrix}
e^{i\varphi} & b\\
b & \zeta
\end{bsmallmatrix}$, $\zeta\in \mathbb{C}$, $\varphi\in [0,\pi)$, $\tau\in (0,1)$ or 
$B=\begin{bsmallmatrix}
\zeta^{*} & b\\
b & 1
\end{bsmallmatrix}$,
$\zeta^{*}\in \mathbb{C}^{*}$, $\tau=0$; $b>0$

Let $B=\begin{bsmallmatrix}
e^{i\varphi} & b\\
b & \zeta
\end{bsmallmatrix}$, $\zeta\in \mathbb{C}$, $0\leq \varphi<\pi$.
If $\widetilde{B}$ is either 
$\begin{bsmallmatrix}
0 & \widetilde{b}\\
\widetilde{b} & \widetilde{d}
\end{bsmallmatrix}$ or 
$\begin{bsmallmatrix}
\widetilde{\zeta} & \widetilde{b}\\
\widetilde{b} & 1
\end{bsmallmatrix}$ with $\widetilde{\zeta}\neq 0$ we 
take 
$
P(s)=
\begin{bsmallmatrix}
s & s^{2}\\
1 & s^{-1}
\end{bsmallmatrix}
$,
$B(s)=
\begin{bsmallmatrix}
e^{i\varphi} & \widetilde{b}+s\\
\widetilde{b}+s & \widetilde{d}s^{2}
\end{bsmallmatrix}
$ or
$
P(s)=
|\widetilde{\zeta}|e^{i\frac{k\pi}{2}}\oplus \frac{1}{|\widetilde{\zeta}|}e^{i\frac{k\pi}{2}}
$, $c(s)=(-1)^{k}$,
$B(s)=
\begin{bsmallmatrix}
e^{i\varphi} & \widetilde{b}+s\\
\widetilde{b}+s & (-1)^{k}|\widetilde{\zeta}|^{2}
\end{bsmallmatrix}
$ with $\arg \widetilde{\zeta}=\arg (\varphi+k\pi)$ in (\ref{cPepsi}) to get a path.
%
%
Next,
$B(s)=
\begin{bsmallmatrix}
\widetilde{a}s^{2}+s^{3} & \widetilde{b}+s\\
\widetilde{b}+s & 1
\end{bsmallmatrix}
$, 
$c(s)=1$,
$
P(s)=
\begin{bsmallmatrix}
s^{-1} & 1\\
s^2 & s
\end{bsmallmatrix}
$
shows
$
\big(\begin{bsmallmatrix}
0 & 1\\
0 & 0
\end{bsmallmatrix},
\begin{bsmallmatrix}
\widetilde{a} & \widetilde{b}\\
\widetilde{b} & 0
\end{bsmallmatrix}\big)
\to 
\big(
\begin{bsmallmatrix}
0 & 1\\
0 & 0
\end{bsmallmatrix},
\begin{bsmallmatrix}
a & b\\
b & 1
\end{bsmallmatrix}\big)$, $\widetilde{b}\geq 0$, $\widetilde{a}\in \{0,1\}$.

\end{enumerate}

\item \label{p1-1v0110} 
$
(1\oplus -1
,\widetilde{B})
\dashrightarrow
\big(\begin{bsmallmatrix}
0 & 1\\
\tau & 0
\end{bsmallmatrix},
B\big)
$, \quad $0<\tau \leq 1$

Lemma \ref{lemapsi1} (\ref{lemapsi11}) with (C\ref{r7}) for $\alpha=-\omega=1$, $\beta=0$ gives ($|\delta_1|,|\delta_2|,|\delta_4|\leq \delta$):

\vspace{-4mm}

\begin{equation}\label{eqBF8}
2\Rea (\overline{x}u)=(-1)^{k}+\delta_1, \quad 2\Rea (\overline{y}v)=-(-1)^{k}+\delta_2, \quad \overline{x}v+\overline{u}y=\delta_4, \quad 1-\tau, \quad k\in \mathbb{Z},
\end{equation}

\vspace{-1mm}
Observe that $u,v\neq 0$, otherwise (\ref{eqBF8}) fails. We compute
\begin{align}\label{xvuyre}
\overline{x}v+\overline{u}y  
&=e^{-2i\phi}(xv-yu) +2\cos (\phi-\eta)e^{-i(\phi+\eta)}uy
=e^{-2i\phi}\det P +\Rea (\overline{x}u)\tfrac{y}{x}, \\
%
%
\label{xvuyre2}
\overline{x}v+\overline{u}y 
&=-e^{-2i\eta}\det P +2\cos (\phi-\eta)e^{-i(\phi+\eta)}vx
=-e^{-2i\eta}\det P +\Rea (\overline{x}u)\tfrac{v}{u}.
\end{align}
Therefore, by combining (\ref{xvuyre}) and (\ref{xvuyre2}) with (\ref{eqBF8}) we obtain
\begin{equation}\label{yoxvou}
\tfrac{y}{x}=\tfrac{\delta_4-e^{-2i\phi}\det P}{(-1)^{k}+\delta_1}, \qquad 
\tfrac{v}{u}=\tfrac{\delta_4+e^{-2i\eta}\det P}{(-1)^{k}+\delta_1}.
\end{equation}

\begin{enumerate}[label=(\alph*),wide=0pt,itemindent=2em,itemsep=6pt]


\item \label{p1-1v0110c}
$B=a\oplus d$, \quad $a\geq 0$ 

%


Equations (\ref{eqBFadad}) and (\ref{yoxvou}) yield
\begin{align}\label{baxdydet}
\widetilde{b}+\epsilon_2= axy+duv=& \tfrac{1}{(-1)^{k}+\delta_1}\big(ax^{2}(\delta_4-e^{-2i\phi}\det P)+du^{2}(\delta_4+e^{-2i\eta}\det P)  \big)\nonumber\\
          =& \tfrac{1}{(-1)^{k}+\delta_1}\big(\delta_4(ax^{2}+du^{2})+
             \det P(-ax^{2}e^{-2i\phi}+du^{2}e^{-2i\eta})  \big)\\
          =&  \tfrac{1}{(-1)^{k}+\delta_1}\big(\delta_4(\widetilde{a}+\epsilon_1)+
             \det P(-a|x|^{2}+d|u|^{2})  \big),  \nonumber
\end{align}
and further for $a,\widetilde{a}\in \mathbb{R}$:

\vspace{-3mm}

\begin{align}\label{y2v2x2u2}
\widetilde{d}+\epsilon_4=& ay^{2}+dv^{2}=\tfrac{1}{((-1)^{k}+\delta_1)^{2}}\big(ax^{2}(\delta_4-e^{-2i\phi}\det P)^{2}+du^{2}(\delta_4+e^{-2i\eta}\det P)^{2}  \big)\nonumber\\
          =& \tfrac{1}{((-1)^{k}+\delta_1)^{2}}\big(  \delta_4^{2}(ax^{2}+du^{2})+
            2\delta_4 \det P(-a|x|^{2}+d|u|^{2}) +(\det P)^{2}(a\overline{x}^{2}+d\overline{u}^{2})  \big)\nonumber\\          
           =&  \tfrac{1}{((-1)^{k}+\delta_1)^{2}}\Big(   \delta_4^{2}(\widetilde{a}+\epsilon_1)+
            2\delta_4 \big((\widetilde{b}+\epsilon_2)((-1)^{k}+\delta_1)-\delta_4(\widetilde{a}+\epsilon_1)\big)\\
            &\qquad \qquad +(\det P)^{2}(\widetilde{a}+\overline{\epsilon}_1+2i\Ima (d)\overline{u}^{2})  \Big).   \nonumber
\end{align}
The equation (\ref{baxdydet}) gives ($a\in\mathbb{R}$):
\begin{equation}\label{detPImd}
\Ima (d)|u|^{2}=\Ima \left(\tfrac{1}{\det P}\big((\widetilde{b}+\epsilon_2)((-1)^{k}+\delta_1)-\delta_4(\widetilde{a}+\epsilon_1)\big)\right).
\end{equation}
Lemma \ref{lemadet} (\ref{lemadeta}) yields $|\det P|\geq \frac{1-\frac{6\delta^{2}}{\nu^{2}}}{\sqrt{1-\delta}}$ (note $1-\tau\leq \delta$ by (\ref{eqBF8}). It follows for $\widetilde{b}=0$ that $|\Ima (d)u^{2}|\leq \nu^{2}\frac{\epsilon (1+\delta)^{\frac{3}{2}}+\delta(\widetilde{a}+\epsilon)\sqrt{1+\delta}}{\nu^{2}-6\delta^{2}}$. It contradicts (\ref{y2v2x2u2}) for $\widetilde{a}<\widetilde{d}$, $\widetilde{b}=0$ and $\epsilon,\delta$ small enough.
Next,
$c(s)=1$, $P(s)=\sqrt{\frac{\widetilde{b}}{2}}\begin{bsmallmatrix}
1 & 1\\
\widetilde{b}^{-1} & -\widetilde{b}^{-1}
\end{bsmallmatrix}$, $B(s)=1\oplus \widetilde{b}^{2}e^{is}$ (or $B(s)=1\oplus -\widetilde{b}^{2}e^{-is}$) yields a path from 
$(1\oplus -1, \widetilde{b}I_2)$, $\widetilde{b}> 0$ (from $\big(1\oplus -1, 
\begin{bsmallmatrix}
0 & \widetilde{b}\\
\widetilde{b} & 0
\end{bsmallmatrix}\big)$, $\widetilde{b}> 0$) to $\big(\begin{bsmallmatrix}
0 & 1\\
\tau & 0
\end{bsmallmatrix},1 \oplus d\big)$, $\Ima (d)>0$.
%
%
%
%
For $P(s)=\frac{1}{2}\begin{bsmallmatrix}
2s & s^{-1}\\
2s & -s^{-1}
\end{bsmallmatrix}$
we get
$(1\oplus -1,0_2)\to 
\big(\begin{bsmallmatrix}
0 & 1\\
1 & 0
\end{bsmallmatrix},1 \oplus 0\big)
$.
%
%
%
%

\item 
$
B=\begin{bsmallmatrix}
a & b\\
b & d
\end{bsmallmatrix}
$, $b>0$ 

Let $\widetilde{B}=\begin{bsmallmatrix}
0 & \widetilde{b}\\
\widetilde{b} & 0
\end{bsmallmatrix}$, $\widetilde{b}>0$. For $a=0$ we have 
$b^{2}=
\widetilde{b}^{2}-(1-\tau)\widetilde{b}^{2}+\epsilon'$ with $1-\tau\leq \delta$, $|\epsilon'|\leq \max\{\epsilon,\frac{\delta^{2}}{\nu^{2}}\}(4\max\{\widetilde{b},\widetilde{b}^{2},1\}+2)^{2}$ (Lemma \ref{lemadet} (\ref{pathlema2ii})).
If
$d=e^{i\varphi}$ with $\varphi<\pi$, 
the proof in \cite[Theorem 3.6, Case VII. (b) (i)]{ST} applies, while for $d=0$ the first equation of (\ref{eqadb}) for $\widetilde{a}=0$ and (\ref{eqBF8}) yield $b(1-\delta)\leq 2b|ux|\leq \epsilon$, which fails for small $\epsilon,\delta$.
%
%

\quad
Suppose
$\widetilde{B}=\widetilde{a}\oplus \widetilde{d}$ for 
$0\leq \widetilde{a}\leq\widetilde{d}$.
If $d=e^{i\varphi}$ the proof in \cite[Theorem 3.6, Case VII. (b) (ii)]{ST} for $\widetilde{a}\neq \widetilde{d}$ applies almost mutatis mutandis, we only replace $\frac{|\det \widetilde{A}|}{|\det A|}=\frac{|\det \widetilde{B}|}{|\det B|}=|\frac{\widetilde{a}\widetilde{d}}{b^{2}}|=1$ with $b^{2}=\widetilde{a}\widetilde{d}-(1-\tau)\widetilde{a}\widetilde{d}+\epsilon'$, $1-\tau\leq \delta$,
$|\epsilon'|\leq \max\{\epsilon,\frac{\delta^{2}}{\nu^{2}}\}(4\max\{\widetilde{d},\widetilde{a}\widetilde{d},1\}+2)^{2}$ (Lemma \ref{lemadet} (\ref{pathlema2ii})). If $d=0$, the first equation of (\ref{eqadb}) for $\widetilde{a}=0$ and (\ref{eqBF8}) give $(\widetilde{a}\widetilde{d}-\delta\widetilde{a}\widetilde{d}-|\epsilon'|)(1-\delta)^{2}\leq 4b^{2}|ux|^{2}\leq |\widetilde{a}+\epsilon|^{2}$, which fails for small $\epsilon,\delta$.
Note, $c(s)=1$, $P(s)=\begin{bsmallmatrix}
\frac{1}{2s} & -\frac{i}{2s}\\
s & is
\end{bsmallmatrix}$, $B(s)=\begin{bsmallmatrix}
0 & \widetilde{d}+s\\
\widetilde{d}+s & 1
\end{bsmallmatrix}$ in (\ref{cPepsi})    
implies
$(1\oplus -1,\widetilde{d}I_2)\to  
\big(\begin{bsmallmatrix}
0 & 1\\
1 & 0
\end{bsmallmatrix},
\begin{bsmallmatrix}
0 & b\\
b & 1
\end{bsmallmatrix}\big)$ for $\widetilde{d}\geq 0$.
By conjugating with 
$\frac{1}{2}\begin{bsmallmatrix}
1 & -2\\
-1 & -2
\end{bsmallmatrix}$ and $r\oplus \frac{1}{r}$ for $r>0$, we get a path
\vspace{-1mm}
\[
(1\oplus -1,\widetilde{a}\oplus\widetilde{d})\approx  
\big(\begin{bsmallmatrix}
0 & 1\\
1 & 0
\end{bsmallmatrix},
\tfrac{1}{4}\begin{bsmallmatrix}
\widetilde{a}+\widetilde{d} & 2(\widetilde{d}-\widetilde{a})\\
2(\widetilde{d}-\widetilde{a}) & 4(\widetilde{a}+\widetilde{d})
\end{bsmallmatrix}\big)\to  
\big(\begin{bsmallmatrix}
0 & 1\\
\tau & 0
\end{bsmallmatrix},
\begin{bsmallmatrix}
r^{2}e^{i\varphi} & b\\
b & {r^{-2}\zeta}
\end{bsmallmatrix}\big)
\approx 
\big(\begin{bsmallmatrix}
0 & 1\\
\tau & 0
\end{bsmallmatrix},\begin{bsmallmatrix}
e^{i\varphi} & b\\
b & \zeta
\end{bsmallmatrix}\big).
\]

\end{enumerate}

\vspace{-3mm}

\item \label{p1i}
$
\big(
\begin{bsmallmatrix}
0 & 1\\
1 & i
\end{bsmallmatrix},
\widetilde{B}\big)\dashrightarrow
\big(\begin{bsmallmatrix}
0 & 1\\
1 & i
\end{bsmallmatrix},
B\big)
$\\
Lemma \ref{lemapsi1} (\ref{lemapsi11}) with (C\ref{r5}) for, $\beta=1$, $\omega=i$, $\alpha=k=0$ (since $||v|^{2}-(-1)^{k}|<\delta$) gives
\begin{align}\label{lemapsi11i}
\big|\overline{x}v+\overline{u}y-1\big|\leq \delta ,\quad |u|^2\leq \delta, \quad \big||v|^{2}-1\big|\leq \delta, \quad |\Rea (\overline{x}u)|,|\Rea (\overline{y}v)|\leq \delta.
\end{align}

\begin{enumerate}[label=(\alph*),wide=0pt,itemindent=2em,itemsep=6pt]
\item \label{p1ia}
$B=a\oplus d$, $a\geq 0$, $d\in \mathbb{C}$

It is not difficult to check that
$B(s)=
s \oplus \frac{\widetilde{b}^{2}}{s}
$, $c(s)=1$,
$
P(s)=e^{-i\frac{\pi}{4}}
\begin{bsmallmatrix}
1 & i\widetilde{b}s^{-1}\\
s^2e^{i\frac{\pi}{4}} & 1
\end{bsmallmatrix}
$
in (\ref{cPepsi}) proves
$
\big(\begin{bsmallmatrix}
0 & 1\\
1 & i
\end{bsmallmatrix},
\begin{bsmallmatrix}
0 & \widetilde{b}\\
\widetilde{b} & 0
\end{bsmallmatrix}\big)
\to 
\big(
\begin{bsmallmatrix}
0 & 1\\
1 & i
\end{bsmallmatrix},
a\oplus d\big)$, $d\in \mathbb{C}$, $a>0$, $\widetilde{b}\geq 0$.

\quad
Next, let $B=0\oplus d$, $d>0$, $\widetilde{B}=\widetilde{a}\oplus \widetilde{d}$, $\widetilde{a}>0$.
Using (\ref{eqBFadad}) for $a=0$ and $|u|^2\leq \delta $ 
we get
%
%

\vspace{-3mm}
\[
 \widetilde{a}+\epsilon\leq |du^{2}|\leq d\delta , \qquad d(1-\delta)\leq |dv^{2}|\leq \epsilon+|\widetilde{d}|.
\]

\vspace{-2mm}
Hence $\frac{|\widetilde{d}|+\epsilon}{1-\delta}\delta \geq \widetilde{a}+\epsilon$, which fails for sufficiently small $\epsilon,\delta$.

\item \label{p1ib}
$B=\begin{bsmallmatrix}
0 & b\\
b & 0
\end{bsmallmatrix}$, $b>0$, \quad (hence $\widetilde{B}=\widetilde{a}\oplus \widetilde{d}$ by Lemma \ref{lemalist})

The proof in \cite[Theorem 3.6, Case V. (b)]{ST} applies mutatis mutandis. 
Note, 
$B(s)=
\frac{\widetilde{d}s}{2}\begin{bsmallmatrix}
0 & 1\\
1 & 0
\end{bsmallmatrix}
$, 
$
P(s)=e^{-i\frac{\pi}{4}}
\begin{bsmallmatrix}
e^{i\frac{\pi}{4}}s & s^{-1}\\
s & i
\end{bsmallmatrix}
$
in (\ref{cPepsi}) implies
$
\big(\begin{bsmallmatrix}
0 & 1\\
1 & i
\end{bsmallmatrix},
0\oplus \widetilde{d}\big)
\to 
\big(
\begin{bsmallmatrix}
0 & 1\\
1 & i
\end{bsmallmatrix},
\begin{bsmallmatrix}
0 & b\\
b & 0
\end{bsmallmatrix}\big)$, 
$\widetilde{d}>0$. 
%
%
%
\end{enumerate}

\vspace{-1mm}

\item \label{p011001i}
$
\big(\begin{bsmallmatrix}
0 & 1\\
1 & 0
\end{bsmallmatrix},
\widetilde{B}\big)
\dashrightarrow
\big(\begin{bsmallmatrix}
0 & 1\\
1 & i
\end{bsmallmatrix},
B\big)
$

Lemma \ref{lemapsi1} (\ref{lemapsi11}) with (C\ref{r5}) for $\alpha=\omega=0$, $\beta=0$ yields
\begin{equation}\label{ocena2psi01}
|u|^{2},|v|^{2}\leq \delta,\quad \big|2\Rea(\overline{y}v)\big|\leq \delta, \quad \big|2\Rea(\overline{x}u)\big|\leq \delta, \quad \big|\overline{x}v+\overline{u}y-(-1)^{k}\big|\leq \delta,\,k\in \mathbb{Z}. 
\end{equation}

\begin{enumerate}[label=(\alph*),wide=0pt,itemindent=2em,itemsep=6pt]
\item \label{p011001ia}
$
B=a\oplus d$, $a\geq 0$\\ 
%
%
Taking $c(s)=1$, $P(s)=\begin{bsmallmatrix}
1 & s^{-1}\\
s & 0
\end{bsmallmatrix}$, $B(s)=0\oplus \frac{1}{s}$ in (\ref{cPepsi}) proves 
$\big(\begin{bsmallmatrix}
0 & 1\\
1 & 0
\end{bsmallmatrix},\widetilde{a}\oplus 0\big)\to\big(\begin{bsmallmatrix}
0 & 1\\
1 & i
\end{bsmallmatrix},0\oplus d \big)$, $\widetilde{a}\in \{0,1\}$, $d>0$.
Next, $c(s)=1$, $P(s)=e^{i\frac{1}{2}\widetilde{\vartheta}}\begin{bsmallmatrix}
1 & s^{-1}\\
s & 0
\end{bsmallmatrix}$, $B(s)=(|\widetilde{d}|+s)s^{2}\oplus \frac{1}{s^{2}}e^{-i\widetilde{\theta}}$  yields 
$\big(\begin{bsmallmatrix}
0 & 1\\
1 & 0
\end{bsmallmatrix},1\oplus \widetilde{d}\big)\to\big(\begin{bsmallmatrix}
0 & 1\\
1 & i
\end{bsmallmatrix},a\oplus d \big)$, $\widetilde{d}=|\widetilde{d}|e^{i\widetilde{\vartheta}}$, $a>0$, $d\in \mathbb{C}$.

\quad
Proceed with 
$\widetilde{b}=\begin{bsmallmatrix}
0 & \widetilde{b}\\
\widetilde{b} & 1
\end{bsmallmatrix}$, $\widetilde{b}>0$; we conjugate the first pair with 
$\frac{1}{2}
\begin{bsmallmatrix}
2 & -2\\
1 & 1
\end{bsmallmatrix}
$:
\begin{equation}\label{pathb1b}
\big(\begin{bsmallmatrix}
0 & 1\\
1 & 0
\end{bsmallmatrix},
\begin{bsmallmatrix}
0 & \widetilde{b}\\
\widetilde{b} & 1
\end{bsmallmatrix}\big)
\approx
\big(1\oplus -1,
\tfrac{1}{4}\begin{bsmallmatrix}
4\widetilde{b}+1 & 1\\
1 & -4\widetilde{b}+1
\end{bsmallmatrix}\big)
\dashrightarrow
\big(\begin{bsmallmatrix}
0 & 1\\
1 & i
\end{bsmallmatrix},
a\oplus d\big), \quad d\in \mathbb{C}, a>0.
\end{equation}

Using ideas from \ref{p1-1v0110} \ref{p1-1v0110c} we 
find
$c(s)=-1$, $P(s)=\frac{e^{i\frac{\pi}{4}}}{\sqrt{2}}
\begin{bsmallmatrix}
\frac{i}{s^{2}}e^{i\alpha(s)} & \frac{i}{s^{2}}e^{-i\alpha(s)}\\
se^{-i\alpha(s)} & se^{i\alpha(s)}
\end{bsmallmatrix}$ with $\sin (2 \alpha(s))=s$, $B(s)=\widetilde{b}s^{3}\oplus (\frac{\widetilde{b}}{s^{3}}-\frac{i}{2s^{2}})$ (see (\ref{cPepsi})), which proves 
the existence of (\ref{pathb1b}).

\item \label{ha1}
$
B=\begin{bsmallmatrix}
0 & b\\
b & 0
\end{bsmallmatrix}$, $b>0$


If
$
\widetilde{B}=\begin{bsmallmatrix}
0 & \widetilde{b}\\
\widetilde{b} & 1
\end{bsmallmatrix}
$, $\widetilde{b}>0$, 
the proof in \cite[Theorem 3.6, Case VI. (b) (i)]{ST} applies mutatis mutandis, we only use 
$b^{2}=\widetilde{b}^{2}+\epsilon'$, 
$|\epsilon'|\leq \max\{\epsilon,\frac{\delta^{2}}{\nu^{2}}\}(4\max\{\widetilde{b},\widetilde{b}^{2},1\}+2)^{2}$ (see Lemma \ref{lemadet} (\ref{pathlema2ii}))
instead of $1=\frac{|\det \widetilde{A}|}{|\det A|}=\frac{|\det \widetilde{B}|}{|\det B|}=\frac{\widetilde{b}^{2}}{b^{2}}$. 
For $\widetilde{B}=1\oplus \widetilde{d}$, $\widetilde{d}\neq 0$ we apply \cite[Theorem 3.6, Case VI. (b) (ii)]{ST}, we only replace $\frac{|\det \widetilde{A}|}{|\det A|}=\frac{|\det \widetilde{B}|}{|\det B|}=|\frac{\widetilde{d}}{b^{2}}|=1$ with $b^{2}=|\widetilde{d}|+\epsilon'$, 
$|\epsilon'|\leq \max\{\epsilon,\frac{\delta^{2}}{\nu^{2}}\}(4\max\{|\widetilde{d}|,1\}+2)^{2}$ (Lemma \ref{lemadet} (\ref{pathlema2ii})).

\end{enumerate}

\item \label{p11-101i}
$
(1\oplus -1,\widetilde{B})
\dashrightarrow
\big(\begin{bsmallmatrix}
0 & 1\\
1 & i
\end{bsmallmatrix},
B\big)$ 

Lemma \ref{lemapsi1} (\ref{lemapsi11}) with (C\ref{r5}) for $-\omega=\alpha=1$, $\beta=0$ yields ($|\delta_1|,|\delta_2|,|\delta_4|<\delta$, $k\in \mathbb{Z}$):
\vspace{-1mm}
\begin{equation}\label{ocena4adad}
2\Rea(\overline{x}u)=(-1)^{k}+ \delta_1, \quad 2\Rea(\overline{y}v)=-(-1)^k+\delta_2, \,\, |u|^{2},|v|^2\leq \delta,\quad 
\overline{x}v+\overline{u}y= \delta_4.
\end{equation}

\vspace{-1mm}

\begin{enumerate}[label=(\alph*),wide=0pt,itemindent=2em,itemsep=6pt]


\item \label{p11-101ib}
$
B=\begin{bsmallmatrix}
0 & b\\
b & 0
\end{bsmallmatrix}
$, $b>0$

The proof in \cite[Theorem 3.6, Case V. (b) (i)]{ST} applies mutatis mutandis for $\widetilde{B}=\begin{bsmallmatrix}
0 & \widetilde{b}\\
\widetilde{b} & 0
\end{bsmallmatrix}$, $\widetilde{b}> 0$; recall $b^{2}=\widetilde{b}^{2}+\epsilon'$, $|\epsilon'|\leq \max\{\epsilon,\frac{\delta^{2}}{\nu^{2}}\}(2\max\{1,\widetilde{b},\widetilde{b}^{2}\}+1)^{2}$ (Lemma \ref{lemadet} (\ref{pathlema2ii})). 

%
%

\quad
Let $\widetilde{B}=\widetilde{a}\oplus \widetilde{d}$, $\widetilde{d}\geq\widetilde{a}\geq 0$.
If $\widetilde{d}>\widetilde{a}>0$ 
the proof in \cite[Theorem 3.6, Case V. (b) (ii)]{ST} applies for 
$b^{2}=\widetilde{a}\widetilde{d}+\epsilon'$, $|\epsilon'|\leq  \max\{\epsilon,\frac{\delta^{2}}{\nu^{2}}\}(4\max\{1,\widetilde{d},\widetilde{a}\widetilde{d}\}+2)^{2}$ (Lemma \ref{lemadet} (\ref{pathlema2ii})).
For $c(s)=-1$, $P(s)=\frac{1}{\sqrt{2}}\begin{bsmallmatrix}
i s^{-1} & s^{-1}\\
-i s & s
\end{bsmallmatrix}$, $B(s)=(\widetilde{d}+s)\begin{bsmallmatrix}
0 & 1\\
1 & 0
\end{bsmallmatrix}$ in (\ref{cPepsi}) we get 
$(1\oplus -1,\widetilde{d}I_2)\to\big(\begin{bsmallmatrix}
0 & 1\\
1 & i
\end{bsmallmatrix},\begin{bsmallmatrix}
0 & b\\
b & 0
\end{bsmallmatrix}\big)$, $\widetilde{d}\geq 0$.
If $\widetilde{a}=0$, $\widetilde{d}> 0$ then Lemma \ref{lemapsi2} (D\ref{p3bd})
yields 
$\widetilde{d}-\epsilon\leq |\frac{v}{u}|\sqrt{\epsilon}(4\widetilde{d}+2)^{\frac{1}{2}}$, 
and Lemma \ref{lemadet} (\ref{lemadeta}) gives $|\det P|\leq 1 + \frac{6\delta^{2}}{\nu^{2}}$.
By applying this and (\ref{ocena4adad}) to (\ref{xvuyre2}) implies 
$(1-\delta)(\widetilde{d}-\epsilon)\leq \sqrt{\epsilon}(4\widetilde{d}+2)^{\frac{1}{2}}(\delta+1+\frac{6\delta^{2}}{\nu^{2}})$, which fails for small $\epsilon,\delta$. 
%
%
%

\item
$B=a\oplus d$, \quad $a\geq 0$, $d\in \mathbb{C}$


If $\widetilde{b}=0$, $0\leq \widetilde{a}<\widetilde{d}$ the same proof as in \ref{p1-1v0110} \ref{p1-1v0110c} applies (see (\ref{eqBF8}) and (\ref{ocena4adad})).
%
%


\end{enumerate}

\item \label{p0110v1-1}
$
\big(\begin{bsmallmatrix}
0 & 1\\
1 & 0
\end{bsmallmatrix},
\widetilde{B}\big)\dashrightarrow
(1\oplus -1,B)$

Lemma \ref{lemapsi1} (\ref{lemapsi11}) with (C\ref{r33}) for $\omega=0$, $\theta=\pi$ gives 
\begin{equation}\label{est12}
 |x|^2-|u|^2= \delta_1,\;\; \overline{x}y-\overline{u}v-(-1)^{k}= \delta_2,\;\; |y|^2-|v|^2= \delta_4,\;\; |\delta_1|,|\delta_2|,|\delta_4|\leq \delta,\,k\in\mathbb{Z}.
\end{equation}

\begin{enumerate}[label=(\alph*),wide=0pt,itemindent=2em,itemsep=6pt]

\item \label{p0110v1-1a} 
$B=a\oplus d$, \quad$0\leq a\leq d$, $d>0$


\begin{enumerate}[label=(\roman*),wide=0pt,itemindent=4em,itemsep=3pt]

\item
$
\widetilde{B}=\begin{bsmallmatrix}
0 & \widetilde{b}\\
\widetilde{b} & 1
\end{bsmallmatrix}
$, $\widetilde{b}>0$

%
%
First, $c(s)=-1$, $P(s)=\begin{bsmallmatrix}
\frac{-i}{2}s & is^{-1}\\
\frac{s}{2} & s^{-1}
\end{bsmallmatrix}$, $B(s)=\widetilde{b}\oplus (\widetilde{b}+s^2)$ in (\ref{cPepsi}) gives 
$\big(\begin{bsmallmatrix}
0 & 1\\
1 & 0
\end{bsmallmatrix},\begin{bsmallmatrix}
0 & \widetilde{b}\\
\widetilde{b} & 1
\end{bsmallmatrix}\big)\to (1\oplus -1,a\oplus d)$, $a<d$.
%
%
%
%
For $a=d$ we apply the proof of \cite[Theorem 3.6, Case VIII (a) (ii)]{ST}, but replace $\begin{bsmallmatrix}
0 & d\\
d & 1
\end{bsmallmatrix}$ with $\begin{bsmallmatrix}
0 & \widetilde{b}\\
\widetilde{b} & 1
\end{bsmallmatrix}$; 
and use 
$d^{2}=|\widetilde{b}|^{2}+\epsilon'$, 
$|\epsilon'|\leq \max\{\epsilon,\frac{\delta^{2}}{\nu^{2}}\}\big(4\max\{|\widetilde{b}|,|\widetilde{b}|^{2},1\}+2\big)^{2}$ (Lemma \ref{lemadet} (\ref{pathlema2ii}))
at the end of the proof.

\item \label{p0110v1-1ai} 
$\widetilde{B}=1\oplus \widetilde{d}$,\quad $ \widetilde{d}\in \mathbb{C}$

We prove
$\big(\begin{bsmallmatrix}
0 & 1\\
1 & 0
\end{bsmallmatrix},1 \oplus 0 \big)\to (1\oplus -1,0\oplus d)$ with 
$P(s)=\begin{bsmallmatrix}
s^{-1} & \frac{1}{2}s\\
s^{-1} & -\frac{1}{2}s
\end{bsmallmatrix}$, $c(s)=1$, $B(s)=0\oplus s^{2}$.

\quad
Proceed with $B=a\oplus d$, $0<a\leq d$. 
We have equations (\ref{eqBFadad}) for $\widetilde{a}=1$, $\Ima \widetilde{d}>0$, $\widetilde{b}=0$.
By combining them with (\ref{exu2v2}), (\ref{exuv}) for $\sigma=-1$ and with (\ref{est12}) we get 
\begin{align}\label{calu2}
\epsilon_1+1- ae^{2i\phi}\delta_1 = u^{2} (ae^{2i(\phi-\eta)}+d),\nonumber\\ 
a\big((-1)^{k}+ \delta_2\big)-e^{-2i\phi}\epsilon_2=-e^{-2i\phi}\big(uv(ae^{2i(\phi-\eta)}+d)\big),\nonumber\\
a\big((-1)^{k}+ \delta_2\big)-e^{-2i\varphi}\epsilon_2=-e^{-2i\varphi}\big(uv(ae^{2i(\varphi-\kappa)}+d)\big),\\
\epsilon_4+\widetilde{d}-ae^{2i\varphi}\delta_4= v^{2} (ae^{2i(\varphi-\kappa)}+d),\nonumber\\
d\big((-1)^{k}+ \delta_2\big)+e^{-2i\kappa}\epsilon_2=e^{-2i\kappa}\big(xy(de^{2i(\kappa-\varphi)}+a)\big),\nonumber
\end{align}
%

%
%

\vspace{-2mm}
 
We have $ad=|\widetilde{d}|+\delta'$, $|\delta'|\leq \max\{\epsilon,\frac{\delta^{2}}{\nu^{2}}\}\big(4\max\{1,|\widetilde{d}|\}+2\big)^{2}$ (Lemma \ref{lemadet}), hence $a\leq \sqrt{|\widetilde{d}|}+1$,
provided that $\max\{\epsilon,\frac{\delta^{2}}{\nu^{2}}\}\leq \frac{1}{(4\max\{1,|\widetilde{d}|\}+2)^{2}}$. 
Next, we divide the first and the second (the third and the fourth) two equations of (\ref{calu2}) to get
\[
\tfrac{u}{v}=\tfrac{\epsilon_1+1- ae^{2i\phi}\delta_1}{a\big((-1)^{k}+ \delta_2\big)-e^{-2i\phi}\epsilon_2}(-e^{-2i\phi})=\tfrac{a\big((-1)^{k}+ \delta_2\big)-e^{-2i\varphi}\epsilon_2}{\epsilon_4+\widetilde{d}-ae^{2i\varphi}\delta_4}(-e^{2i\varphi}).
\]
The second equality yields that there is a (computable) constant $D>0$ so that
\begin{equation}\label{rad}
a^{2}=\widetilde{d}e^{-2i(\phi+\varphi)}+\delta_5, \quad d^{2}=\tfrac{(|\widetilde{d}|+\delta')^{2}}{\widetilde{d}e^{2i(\phi-\varphi)}+\delta_5},\qquad |\delta_5|\leq D \max \{\epsilon,\delta\},
\end{equation}
%
Furthermore, we divide the third and the fifth equation of (\ref{calu2}) to conclude:
\begin{equation}\label{fxyuv}
\tfrac{xy}{uv}=\tfrac{(d((-1)^{k}+ \delta_2)+e^{-2i\kappa}\epsilon_2)}{(a((-1)^{k}+ \delta_2)e^{2i\varphi}\epsilon_2)}=1+\delta_6, \qquad |\delta_6|\leq C\max\{\epsilon,\delta\},
\end{equation}
while the first four equations of (\ref{calu2}) yield
\[
\tfrac{1}{\widetilde{d}}+\delta_0=\tfrac{\big(1+\epsilon_1- ae^{2i\phi}\delta_1\big)\big(a(-1)^{k}+ a\delta_2-e^{-2i\phi}\epsilon_2\big)}{\big(\widetilde{d}+\epsilon_1- ae^{2i\varphi}\delta_4\big)\big(a(-1)^{k}+ a\delta_2-e^{-2i\varphi}\epsilon_2\big)}=e^{i(2\eta-2\kappa-2\phi+2\varphi)}\tfrac{|u|^{2}}{|v|^{2}},\quad |\delta_0|\leq K \max\{\epsilon,\delta\},
\]
where constants $C,K>0$ can be computed.
By applying (\ref{ocenakoren}) for $\widetilde{d}=|\widetilde{d}|e^{i\widetilde{\vartheta}}$ we get $2\eta-2\kappa-2\phi+2\varphi+\widetilde{\vartheta}=\psi$ with 
$|e^{i\frac{\psi}{2}}-1|=|\sin \frac{\psi}{4}|\leq |\sin \psi|\leq\delta_0$. Using (\ref{fxyuv}) we get
\begin{align*}
\big|\tfrac{\overline{x}y}{\overline{u}v}
&-1\big|
=\big|\tfrac{|xy|}{|uv|}e^{i(\phi-\varphi-\kappa+\eta)}-1\big|
=\big||1+\delta_6|e^{i(-\frac{\widetilde{\vartheta}}{2}+\frac{\psi}{2})}-1\big|\\
&=\big|e^{i\frac{\psi}{2}}(e^{-i\frac{\widetilde{\vartheta}}{2}}+1)-(e^{i\frac{\psi}{2}}-1)+(|1-\delta_6|-1)e^{i(-\frac{\widetilde{\vartheta}}{2}+\frac{\psi}{2})}\big|
\geq |e^{-i\frac{\widetilde{\vartheta}}{2}}+1|-|\delta_0|-|\delta_6| \geq \cos \tfrac{\widetilde{\vartheta}}{4},
\end{align*}
provided that $\epsilon,\delta $ are such that $\tfrac{1}{4}|e^{-i\frac{\widetilde{\vartheta}}{2}}+1|=\frac{1}{2}\cos \frac{\widetilde{\vartheta}}{4}\geq|\delta_0|,|\delta_6|$ with $0<\widetilde{\vartheta}<\pi$. Thus: 
%
\begin{align*}
&2\geq 1+\delta\geq |\overline{x}y-\overline{u}v|=
|\overline{u}v|\big|\tfrac{\overline{x}y}{\overline{u}v}-1\big|
\geq \tfrac{1}{2}|uv|\cos \tfrac{\widetilde{\vartheta}}{4},\qquad
|u|^{2}=\tfrac{|u|}{|v|}|uv|\leq 4\tfrac{|\widetilde{d}|^{-1}+|\delta_0|}{\cos \frac{\widetilde{\vartheta}}{4}}.
\end{align*}
%

We simplify the first and the third equation of (\ref{calu2}) and rearrange the terms: 
%
%
\begin{align}\label{calu22}
&2au^{2}\cos (\phi-\eta)e^{i(\phi-\eta)}=1+\epsilon_1-ae^{2i\phi}\delta_1-(d-a)u^{2}, \\ 
&-2auv\cos (\varphi-\kappa)e^{-i(\varphi+\kappa)}=a(-1)^{k}+a\delta_2-e^{-2i\varphi}\epsilon_2+(d-a)uve^{-2i\varphi} \nonumber.
\end{align}
By applying (\ref{ocenah}) 
with $\widetilde{d}=|\widetilde{d}|e^{i\widetilde{\vartheta}}$ 
to (\ref{rad}) and (\ref{calu22}) we deduce ($L>0$ is a constant):
\begin{align*}
&\psi_0=\widetilde{\theta}-2(\phi+\varphi), \qquad |\sin \psi|\leq \tfrac{2|\delta_5|}{|\widetilde{d}|}, \\
&\psi_1=(\phi+\eta)-\pi l_1,\qquad  |\sin \psi_1|\leq L\max\{\epsilon,\delta\},\quad l_1\in \mathbb{Z}, \\
&\psi_2=(\eta-\varphi)-\pi(k+l_2),\qquad  |\sin \psi_2|\leq L\max\{\epsilon,\delta\},\quad l_2\in \mathbb{Z}.
\end{align*}
Thus 
$\big|\sin(\psi_0+2\psi_1-2\psi_2)\big|=|\sin \widetilde{\vartheta}|\leq \tfrac{2|\delta_5|}{|\widetilde{d}|} +4L \max\{\epsilon,\delta\}$ and it fails for small $\epsilon,\delta$.

\end{enumerate}

\item
$
B=\begin{bsmallmatrix}
0 & b\\
b & 0
\end{bsmallmatrix}$, $b>0$


If $
\widetilde{B}=\begin{bsmallmatrix}
0 & \widetilde{b}\\
\widetilde{b} & 1
\end{bsmallmatrix}
$ for $\widetilde{b}>0$ 
we can apply the proof of \cite[Theorem 3.6, Case VIII (b) (i)]{ST}, recall
$b^{2}=\widetilde{b}^{2}+\epsilon'$, 
$|\epsilon'|\leq \max\{\epsilon,\frac{\delta^{2}}{\nu^{2}}\}(4\max\{|\widetilde{b}|,|\widetilde{b}|^{2},1\}+2)^{2}$ (Lemma \ref{lemadet} (\ref{pathlema2ii})).


\quad
Let $\widetilde{B}=1\oplus \widetilde{d}$, $\widetilde{d}\in \mathbb{C}$. 
To get
$\big(\begin{bsmallmatrix}
0 & 1\\
1 & 0
\end{bsmallmatrix},1 \oplus 0 \big)\to \big(1\oplus -1,\begin{bsmallmatrix}
0 & b\\
b & 0
\end{bsmallmatrix}\big)$, we take 
$P(s)=\frac{1}{2}\begin{bsmallmatrix}
2s^{-1} & s\\
2s^{-1} & -s
\end{bsmallmatrix}$, $c(s)=1$, 
$B(s)=\frac{s^{2}}{2}\begin{bsmallmatrix}
0 & 1\\
1 & 0
\end{bsmallmatrix}$ in (\ref{cPepsi}). 
If $\widetilde{d}=|\widetilde{d}|e^{i\widetilde{\vartheta}}$, $0<\widetilde{\vartheta}<\pi$ Lemma \ref{lemapsi2} (D\ref{p4}) implies 
%
\begin{align*}
bvx=\tfrac{1}{2}\big(\epsilon_2'+(-1)^{l}i\sqrt{|\widetilde{d}|}e^{i\tfrac{\widetilde{\vartheta}}{2}}\big),\qquad 
buy=\tfrac{1}{2}\big(\epsilon_2''-(-1)^{l}i\sqrt{|\widetilde{d}|}e^{i\tfrac{\widetilde{\vartheta}}{2}}\big),
\end{align*}
where $
\small{|\epsilon_2'|,|\epsilon_2''|
\leq 
\tfrac{\epsilon(4\max\{1,|\widetilde{d}|\}+2+|\widetilde{d}|)}{|\widetilde{d}|}}.$
By applying (\ref{ocenah}) to these two equations and to the first equality of (\ref{eqadb}) we get $\psi_1,\psi_2,\psi_3\in (-\frac{\pi}{2},\frac{\pi}{2})$ such that:
\begin{align*}
&\psi_1=\phi+\kappa-\tfrac{\pi}{2}-\tfrac{\widetilde{\vartheta}}{2}-l\pi+2\pi l_3, \quad 
|\sin \psi_1|\leq \tfrac{2|\epsilon_2'|}{\sqrt{|\widetilde{d}|}}, \\
&\psi_2=\varphi+\eta-\tfrac{\pi}{2}-\tfrac{\widetilde{\vartheta}}{2}-(l+1)\pi+2\pi l_4, \quad |\sin \psi_2|\leq \tfrac{2|\epsilon_2''|}{\sqrt{|\widetilde{d}|}},\\
&\psi_3=\phi+\eta+2\pi l_1, \quad |\sin \psi_3|\leq \epsilon.  
\end{align*}
Therefore
\begin{align*}
(-1)^{k}+\delta_2&=\overline{x}y-\overline{u}v=|xy|e^{i(\varphi-\phi)}-|uv|e^{i(\kappa-\eta)}=e^{-i(\phi+\eta)}\big(|xy|e^{i(\varphi+\eta)}-|uv|e^{i(\kappa+\phi)}\big)\\
                 &=e^{i(-\psi_3+2\pi l_1)}\big(|xy|e^{i(\psi_2-2\pi l_4+\frac{\pi}{2}+\frac{\widetilde{\vartheta}}{2}+(l+1)\pi)}-|uv|e^{i(\frac{\pi}{2}+\frac{\widetilde{\vartheta}}{2}+l\pi-2\pi l_3+\psi_1)}\big)=\\
                 &=e^{i(\psi_2-\psi_3+\frac{\widetilde{\vartheta}}{2}+(l+1)\pi+\frac{\pi}{2})}\big(|xy|+|uv|e^{i(\psi_1-\psi_2)}\big).
\end{align*}
Since $\psi_1,\psi_2,\psi_3\in (-\frac{\pi}{2},\frac{\pi}{2})$ are close to $0$, the argument of the second factor is close to $0$, too.
Using (\ref{ocenah}) again we obtain a contradiction for $\epsilon,\delta$ small enough:
\begin{align*}
&\psi= k\pi-\big(\psi_2-\psi_3+\tfrac{\widetilde{\vartheta}}{2}+(l+1)\pi+\tfrac{\pi}{2}\big)-(\psi_1-\psi_2),\qquad |\sin \psi|\leq 2\delta,\\ 
&0\neq |\cos \tfrac{\widetilde{\vartheta}}{2}|=\big|\sin (\tfrac{\widetilde{\vartheta}}{2}+\tfrac{\pi}{2})\big|\leq \big|\sin (\psi_3+\psi_1)\big|\leq 2\epsilon_2'+2\delta.
\end{align*}
%


\end{enumerate}

\item \label{p01t001t01}
$
\bigl(\begin{bsmallmatrix}
0 & 1\\
1 & 0 
\end{bsmallmatrix}, \widetilde{B}
\bigr)\dashrightarrow
\bigl(\begin{bsmallmatrix}
0 & 1\\
1 & 0
\end{bsmallmatrix},B\bigr)$ 

%
%

\begin{enumerate}[label=(\alph*),wide=0pt,itemindent=2em,itemsep=6pt]


\item 
$
B=1\oplus d
$, $\Ima d> 0$, \quad
$\widetilde{B}=\begin{bsmallmatrix}
0 & \widetilde{b}\\
\widetilde{b} & 1
\end{bsmallmatrix}
$, $\widetilde{b}>0$

%
We can apply the proof of \cite[Theorem 3.6, Case IX (b)]{ST}, 
and use 
$|d|=|\widetilde{b}|^{2}+\epsilon'$, 
$|\epsilon'|\leq \max\{\epsilon,\frac{\delta^{2}}{\nu^{2}}\}(4\max\{|\widetilde{b}|,|\widetilde{b}|^{2},1\}+2)^{2}$ (Lemma \ref{lemadet} (\ref{pathlema2ii})).

\item \label{p01t001t01b}
$B=\begin{bsmallmatrix}
0 & b\\
b & 1
\end{bsmallmatrix}
$, $b>0$


For $P(s)=\begin{bsmallmatrix}
0 & 1\\
1 & 0
\end{bsmallmatrix}$, $B=\begin{bsmallmatrix}
0 & s\\
s & 1
\end{bsmallmatrix}$, $c(s)=1$ in (\ref{cPepsi}) we get  
$\big(\begin{bsmallmatrix}
0 & 1\\
1 & 0
\end{bsmallmatrix},1\oplus 0\big)\to \big(\begin{bsmallmatrix}
0 & 1\\
1 & 0
\end{bsmallmatrix},\begin{bsmallmatrix}
0 & b\\
b & 1
\end{bsmallmatrix}\big)$, $b>0$.
For $\widetilde{B}=1\oplus \widetilde{d}$, $\widetilde{d}\neq 0$ we use the proof of \cite[Theorem 3.6, Case IX (c)]{ST}, 
but replace $\frac{|\det \widetilde{B}|}{|\det B|}=\frac{\widetilde{a}\widetilde{d}}{b^{2}}$ with 
$b^{2}=|\widetilde{d}|+\epsilon'$, 
$|\epsilon'|\leq \max\{\epsilon,\frac{\delta^{2}}{\nu^{2}}\}(4\max\{|\widetilde{b}|,|\widetilde{b}|^{2},1\}+2)^{2}$ (Lemma \ref{lemadet} (\ref{pathlema2ii})).

\end{enumerate}

\item 
$ \big(\begin{bsmallmatrix}
0 & 1\\
1 & \omega 
\end{bsmallmatrix},
\widetilde{B}\big)\dashrightarrow 
(1\oplus e^{i\theta},
B)
$, \qquad $0< \theta <\pi$, \quad $\omega \in \{0,i\}$

From Lemma \ref{lemapsi1} (C\ref{r33}) we get
\begin{align}\label{inocena1psi}
&\big||u|^2-|x|^2\big| \leq \delta, \quad \big||v|^2-|y|^2\big| \leq \delta, \quad \big|\overline{x}y-\overline{u}v-(-1)^{k}\big|\leq \delta,\quad k\in \mathbb{Z},\,\,\sin\theta\leq \delta; \\
&\textrm{if}\quad \omega=i,\quad \textrm{then} \quad  (\sin \theta) |v|^2=1+\delta_2,\,\, (\sin \theta)|u|^{2}=\delta_3, \,\, 
|\delta_2|,|\delta_3|\leq \delta.\nonumber
\end{align}
%
For $\omega=i$ we further deduce 
\begin{equation}\label{xyuvth}
\quad \big|(\sin\theta)|y|^{2}- 1\big| \leq \delta+\delta^{2}, \quad  (\sin\theta)|x|^{2}\leq  \delta+\delta^{2}.
\end{equation}

\begin{enumerate}[label=(\alph*),wide=0pt,itemindent=2em,itemsep=6pt]
\item $B=\begin{bsmallmatrix}
0 & b\\
b & d 
\end{bsmallmatrix}$, $b\geq 0$, $d> 0$

Lemma \ref{lemapsi2} (D\ref{p3bd}) for $\widetilde{B}=\widetilde{a}\oplus \widetilde{d}$, $\widetilde{a}\neq 0$
and (\ref{inocena1psi}) for $\omega=i$ (hence $(1+\delta_2)|u|^{2}=\delta_3|v|^{2}$) yield a contradiction for small $\epsilon,\delta$.
Next, $c(s)=1$, $P(s)=i\sqrt{\widetilde{d}+s}\begin{bsmallmatrix}
\frac{s}{\widetilde{d}+s} & s^{-1}\\
0 & -s^{-1} 
\end{bsmallmatrix}$,
\small
$\cos (\frac{\theta}{2})=\left\{
\begin{array}{ll}
\frac{s^{2}}{2(\widetilde{d}+s)}, & \omega=i\\
s^{3}, & \omega=0
\end{array}\right.$
\normalsize, 
$B(s)=\begin{bsmallmatrix}
0 & \widetilde{b} \\
\widetilde{b}  & 2\widetilde{b}-s^2
\end{bsmallmatrix}$
in (\ref{cPepsi}) proves 
$
 \big(\begin{bsmallmatrix}
0 & 1\\
1 & \omega
\end{bsmallmatrix},
\begin{bsmallmatrix}
0 & \widetilde{b}\\
\widetilde{b} & \widetilde{d}
\end{bsmallmatrix}\big)\to 
\big(1\oplus e^{i\theta},
\begin{bsmallmatrix}
0 & b\\
b & d
\end{bsmallmatrix}\big)
$, $\widetilde{b}> 0$, either $\omega=0$, $\widetilde{d}=1$ or $\omega=i$, $\widetilde{d}=0 $. Taking
$c(s)=1$, $P(s)=\sqrt{\widetilde{d}+s}\begin{bsmallmatrix}
\frac{s}{\widetilde{d}+s} & s^{-1}\\
0 & -s^{-1} 
\end{bsmallmatrix}$,
$\cos (\frac{\theta}{2})=
\frac{s^{2}}{2(\widetilde{d}+s)}$, 
$B(s)=0\oplus s^2$ 
shows 
$
 \big(\begin{bsmallmatrix}
0 & 1\\
1 & i
\end{bsmallmatrix},
0\oplus \widetilde{d}
\big)\to 
(1\oplus e^{i\theta},
0\oplus d)$. 
%
%
%
%
Finally, $c(s)=-ie^{i\frac{\widetilde{\vartheta}}{2}}$, 
$\cos(\frac{\theta(s)}{2})=s^{3}$, $P(s)=\frac{1}{s}e^{-i\frac{\pi}{4}}\begin{bsmallmatrix}
e^{i\alpha(s)} & ie^{-i\alpha(s)}\sqrt{\widetilde{d}+s}\\
-e^{-i\alpha(s)} & -ie^{i\alpha(s)}\sqrt{\widetilde{d}+s} 
\end{bsmallmatrix}$, $\sin (2\alpha(s))=\frac{s^{2}}{2|\sqrt{\widetilde{d}+s}|}$, 
$B(s)=\big|\sqrt{\widetilde{d}+s}\big|\begin{bsmallmatrix}
0 & 1 \\
1  & 2
\end{bsmallmatrix}$ in (\ref{cPepsi}) proves 
$
\big(\begin{bsmallmatrix}
0 & 1\\
1 & 0
\end{bsmallmatrix},
1\oplus \widetilde{d}\big)\to 
\big(1\oplus e^{i\theta},
\begin{bsmallmatrix}
0 & b\\
b & d
\end{bsmallmatrix}\big)
$, $b > 0$, $\Ima (\widetilde{d})> 0$.


\item $B=\begin{bsmallmatrix}
0 & b\\
b & 0 
\end{bsmallmatrix}$, $b>0$

Let $\widetilde{B}=\begin{bsmallmatrix}
0 & \widetilde{b}\\
\widetilde{b} & 0 
\end{bsmallmatrix}$, $\widetilde{b}> 0$ and $\omega=i$. 
It follows from Lemma \ref{lemadet} (\ref{pathlema2ii}) that $b^{2}=\widetilde{b}^{2}+\epsilon'$, $|\epsilon'|\leq \max \{\epsilon,\frac{\delta^{2}}{\nu^{2}}\}\big(4\max\{1,|\widetilde{b}|,|\widetilde{b}|^{2}\}+2\big)^{2}$, so the third equation of (\ref{eqadb}) for $\widetilde{d}=0$ yields $(yv)^{2}=\frac{\epsilon_4^{2}}{4(\widetilde{b}^{2}+\epsilon')}$.
By combining it with (\ref{inocena1psi}) and (\ref{xyuvth}) we deduce 
\[
\big(1- \delta(1+\delta)\big)(1-\delta)\leq (\sin\theta)^{2}|yv|^{2}=\tfrac{\delta^{2}\epsilon^{2}}{4|\widetilde{b}^{2}+\epsilon'|},
\]
which fails for $\epsilon,\delta$ small enough.
Next, $c(s)=1$,
$\cos (\frac{\theta(s)}{2})=s^{2}$, 
$P(s)=\frac{i}{\sqrt{2}}
\begin{bsmallmatrix}
s & s^{-1}\\
s & -s^{-1}
\end{bsmallmatrix}$, 
$B(s)=(\widetilde{d}+s)s^{2}\begin{bsmallmatrix}
0 & 1 \\
1  & 0
\end{bsmallmatrix}$ in (\ref{cPepsi}) gives
$
\big(\begin{bsmallmatrix}
0 & 1\\
1 & i
\end{bsmallmatrix},
0\oplus \widetilde{d}\big)\to
\big(1\oplus e^{i\theta},
\begin{bsmallmatrix}
0 & b\\
b & 0
\end{bsmallmatrix}\big)$, $\widetilde{d}\geq 0$.
%

\quad
We apply the same proof as in \ref{p0110v1-1} (compare (\ref{est12}) and (\ref{inocena1psi})) to show
$
\big(\begin{bsmallmatrix}
0 & 1\\
1 & 0
\end{bsmallmatrix},
\begin{bsmallmatrix}
0 & \widetilde{b}\\
\widetilde{b} & 1
\end{bsmallmatrix}\big)\not\to 
\big(1\oplus e^{i\theta},
\begin{bsmallmatrix}
0 & b\\
b & 0
\end{bsmallmatrix}\big)
$, $\widetilde{b}>0$ and
$
\big(\begin{bsmallmatrix}
0 & 1\\
1 & 0
\end{bsmallmatrix},
1\oplus \widetilde{d}\big)\not\to 
\big(1\oplus e^{i\theta},
\begin{bsmallmatrix}
0 & b\\
b & 0
\end{bsmallmatrix}\big)
$, $\Ima \widetilde{d}> 0$.

\item $\begin{bsmallmatrix}
a & b\\
b & 0 
\end{bsmallmatrix}$, $a>0$, $b\geq 0$


We multiply the squared equation in Lemma \ref{lemapsi2} (D\ref{p3ba}) for $\widetilde{B}=\widetilde{a}\oplus \widetilde{d}$ with $(\sin \theta)^{2}$: 
\vspace{-1mm}
\[
(\widetilde{a}+\epsilon_1)^{2}y^{2}\sin^{2} \theta=\big(-i(-1)^{l}\sqrt{\widetilde{a}\widetilde{d}}+\epsilon_2'\big)x^{2} \sin^{2} \theta, \,\,
|\epsilon_2'|\leq \small{\left\{\begin{array}{ll}
\tfrac{\epsilon(4\max\{\widetilde{a},\widetilde{d}\}+2+|\widetilde{a}\widetilde{d}|)}{|\widetilde{a}\widetilde{d}|}, & \widetilde{a}\widetilde{d}\neq 0\\
\sqrt{\epsilon(4\max\{\widetilde{a},\widetilde{d}\}+3)}, & \widetilde{a}\widetilde{d}= 0
\end{array}\right.}.
\]
\vspace{-1mm}
By applying (\ref{inocena1psi}) and (\ref{xyuvth}) (for $\omega=i$) we get
$|\widetilde{a}+\epsilon_1|^{2}(1-\delta)\leq\big(\sqrt{|\widetilde{a}\widetilde{d}|}+|\epsilon_2'|\big)(\delta+\delta^{2})$, which fails for $\widetilde{a}\neq 0$ and small $\epsilon,\delta$.
For $c(s)=e^{i\frac{\widetilde{\vartheta}}{2}}$, $P(s)=\frac{1}{s}e^{-i\frac{\pi}{4}}\begin{bsmallmatrix}
-e^{-i\alpha} & -ie^{i\alpha}\sqrt{\widetilde{d}+s} \\
e^{i\alpha} & ie^{-i\alpha}\sqrt{\widetilde{d}+s}
\end{bsmallmatrix}$ with $\sin (2\alpha(s))=\frac{s^{2}}{2|\sqrt{\widetilde{d}+s}|}$, $\cos\frac{\theta}{2}=s^{3}$, 
$B(s)=|\sqrt{\widetilde{d}+s}|\begin{bsmallmatrix}
2 & 1 \\
1  & 0
\end{bsmallmatrix}$ in (\ref{cPepsi}), it follows  
$
\big(\begin{bsmallmatrix}
0 & 1\\
1 & 0
\end{bsmallmatrix},
1\oplus \widetilde{d}\big)\to 
\big(1\oplus e^{i\theta},
\begin{bsmallmatrix}
a & b\\
b & 0
\end{bsmallmatrix}\big)
$, $\pi> \widetilde{\vartheta}=\arg \widetilde{d}> 0$ or $\widetilde{d}=0$.
Taking $c(s)=1$, $P(s)=\frac{1}{\sqrt{\widetilde{d}+s}}\begin{bsmallmatrix}
0 & \frac{1}{s}(\widetilde{d}+s) \\
s & -\frac{1}{s}(\widetilde{d}+s)
\end{bsmallmatrix}$, 
$B(s)=\begin{bsmallmatrix}
2b(s)+s^2 & b(s) \\
b(s)  & 0
\end{bsmallmatrix}$, $b(s)\to \widetilde{b}$,
$\cos (\frac{\theta}{2})=\left\{
\begin{array}{ll}
\frac{s^{2}}{2(\widetilde{d}+s)}, & \omega=i\\
s^{3}, & \omega=0
\end{array}\right.$ proves 
$
\big(\begin{bsmallmatrix}
0 & 1\\
1 & \omega
\end{bsmallmatrix},
\begin{bsmallmatrix}
0 & \widetilde{b}\\
\widetilde{b} & \widetilde{d}
\end{bsmallmatrix}\big)\to 
\big(1\oplus e^{i\theta},
\begin{bsmallmatrix}
a & b\\
b & 0
\end{bsmallmatrix}\big)
$, $b\geq\widetilde{b}\geq 0$, either $\omega=0$, $\widetilde{b}>0$, $\widetilde{d}=1$ or $\omega=i$, $\widetilde{d},\widetilde{b}\geq 0 $.

\item $B=\begin{bsmallmatrix}
a & b\\
b & d 
\end{bsmallmatrix}$, $a,d>0$, $b\in \mathbb{C}$

First let $b=0$. We deal with the case $\omega=0$ in the same manner as in \ref{p0110v1-1} \ref{p0110v1-1a} \ref{p0110v1-1ai} (compare also (\ref{est12}) and (\ref{inocena1psi}); observe that the proof works in the case $a>d$, too). If $\omega =i$ we have $|v|^{2}\geq 1$, $\tfrac{|u|^{2}}{|v|^{2}}\leq \delta\leq \frac{1}{2}$ and using (\ref{inocena1psi}) we easily verify
\[
\tfrac{|x|^{2}}{|y|^{2}}=\tfrac{|u|^{2}+\delta}{|v|^{2}-\delta}=\tfrac{|u|^{2}}{|v|^{2}}+\tfrac{|v|^{2}\delta+|u|^{2}\delta}{(|v|^{2}-\delta)|v|^{2}}\leq 
\delta+\tfrac{|v|^{2}\delta+\frac{1}{2}|v|^{2}\delta}{\frac{1}{2}|v|^{2}}\leq 
4\delta.
\]
Multiplying the second equation of (\ref{eqBFadad}) with $\delta_5=\frac{x}{y}$ and $\delta_6=\frac{u}{v}$ yields
\[
ax^{2}+dv^{2}\delta_6\delta_5=(\widetilde{b}+\epsilon_2)\delta_5,\qquad ay^{2}\delta_6\delta_5+du^{2}=(\widetilde{b}+\epsilon_2)\delta_6.
\]
By adding them and using (\ref{eqBFadad}) yields a contradiction for $\widetilde{a}\neq 0$ and small $\epsilon,\delta$:
\[
(\widetilde{a}+\epsilon_2)+(\widetilde{d}+\epsilon_2)\delta_6\delta_5=(\widetilde{b}+\epsilon_2)(\delta_5+\delta_6).
\]

\quad
It is tedious to find $c(s)=1$,
$\cos (\frac{\theta(s)}{2})=s^{2}$,   
$B(s)=\frac{1}{2}\begin{bsmallmatrix}
\widetilde{a}s^{-2} & \widetilde{a}s^{-2}-2ds^2 \\
\widetilde{a}s^{-2}-2ds^2  & \widetilde{a}s^{-2}
\end{bsmallmatrix}$,
$P(s)=\frac{1}{\sqrt{2}}
\begin{bsmallmatrix}
s & s^{-1}\\
s & -s^{-1}
\end{bsmallmatrix}$
in (\ref{cPepsi}) to prove 
$
\big(\begin{bsmallmatrix}
0 & 1\\
1 & i
\end{bsmallmatrix},
\widetilde{a}\oplus \widetilde{d}\big)\to 
\big(1\oplus e^{i\theta},
\begin{bsmallmatrix}
a & b\\
b & d
\end{bsmallmatrix}\big)
$, $a,d,\widetilde{a}> 0$, $b\in \mathbb{C}^{*}$, $\widetilde{d}\in \mathbb{C}$.

\end{enumerate}

\item \label{011og01t0}
$ \big(\begin{bsmallmatrix}
0 & 1\\
1 & \omega 
\end{bsmallmatrix},
\widetilde{B}\big)\dashrightarrow 
\big(\begin{bsmallmatrix}
0 & 1\\
\tau & 0 
\end{bsmallmatrix},
\begin{bsmallmatrix}
a & b\\
b & d 
\end{bsmallmatrix}\big)
$, \quad $0< \tau <1$, $\omega\in \{0,i\}$

The following expressions are bounded by $\delta$ (Lemma \ref{lemapsi1} (\ref{lemapsi11}) (C\ref{r44}) for $\alpha=1$):
\begin{equation}\label{boundd}
\Rea(\overline{x}u), (1-\tau)\Ima(\overline{x}u),\Rea(\overline{y}v),(1-\tau)\Ima(\overline{y}v)-(-1)^{k}|\omega|, 1-\tau, \overline{x}v+\overline{u}y-(-1)^{k},
\end{equation}
where $k\in \mathbb{Z}$.
If in addition $\omega=i$, it then follows that
\begin{align}\label{ocena11xuyv}
&\delta_5=\tfrac{|xu|}{|yv|}=\tfrac{(1-\tau)|xu|}{(1-\tau)|yv|}  \leq\tfrac{\big|(1-\tau)\Rea(\overline{x}u)\big|+\big|(1-\tau)\Ima(\overline{x}u)\big|}{\big|(1-\tau)\Ima(\overline{y}v)\big|-\big|(1-\tau)\Rea(\overline{y}v)\big|}\leq \tfrac{\delta+\delta^{2}}{1-\delta-\delta^{2}},\\
&\delta|yv|\geq \big|(1-\tau)\Ima(\overline{y}v)\big|\geq 1-\delta,\nonumber\\
%
%
\label{ocenayv}
&(1+\delta)\tfrac{|v|}{|u|}\geq |\overline{u}y+\overline{x}v|\tfrac{|v|}{|u|}\geq |vy|-\tfrac{|xv|}{|uy|}|vy|= |vy|\big(1-\tfrac{|xu|}{|vy|}\,|\tfrac{v}{u}|^{2}\big),\\
&(1+\delta)\tfrac{|y|}{|x|}\geq |\overline{u}y+\overline{x}v|\tfrac{|y|}{|x|}\geq |vy|-\tfrac{|uy|}{|xv|}|vy|= |vy|\big(1-\tfrac{|xu|}{|vy|}\,|\tfrac{y}{x}|^{2}\big)\nonumber.
\end{align}
%

\begin{enumerate}[label=(\alph*),wide=0pt,itemindent=2em,itemsep=6pt]

\item \label{011og01t0a}
$B=a\oplus d$

Let $B=0\oplus 1$.
If $\widetilde{a}\neq 0$ (hence $\widetilde{b}=\widetilde{d}=0$, $\omega=i$) then
(\ref{eqBF1}) for $a=b=\widetilde{d}=0$, $d=1$ yields 
$(\frac{v}{u})^{2}=\frac{\epsilon_2}{\widetilde{a}+\epsilon_1}$, thus  
(\ref{ocena11xuyv}), (\ref{ocenayv}) give a contradiction for small $\epsilon,\delta $. 
Taking $c(s)=1$, $\tau(s)=1-s$, 
$P(s)=\frac{1}{\sqrt{\widetilde{d}+s}}\begin{bsmallmatrix}
1 & -i s^{-1} \\
0 & \widetilde{d}+s
\end{bsmallmatrix}$ 
proves 
$
\big(\begin{bsmallmatrix}
0 & 1\\
1 & i
\end{bsmallmatrix},
0\oplus \widetilde{d}\big)\to 
\big(\begin{bsmallmatrix}
0 & 1\\
\tau & 0
\end{bsmallmatrix},
0\oplus 1\big)
$, $\widetilde{d}\geq 0$.

\quad
Next, $B=1\oplus d$, $d\in \mathbb{C}$. 
If either $|\frac{x}{y}|\geq 1$ (or $|\frac{u}{v}|\geq 1$), then in case $\omega=i$ the second (the first) inequality of (\ref{ocenayv}) yields a contradiction. When $|\frac{x}{y}|,|\frac{u}{v}|\leq 1$ we multiply the second equation of (\ref{eqBF1}) for $b=0$, $a=1$ with $\frac{u}{v}$ and $\frac{x}{y}$, and simplify them: 
\[
\delta_5y^{2}+du^{2}=(\widetilde{b}+\epsilon_2)\tfrac{u}{v}, \quad 
x^{2}+\delta_5dv^{2}=(\widetilde{b}+\epsilon_2)\tfrac{x}{y} \quad \qquad (\delta_5\leq \tfrac{\delta+\delta^{2}}{1-\delta-\delta^{2}}).
\]
We add these equations and use (\ref{eqBF1}) for $b=0$, $a=1$ to get $\delta_5(\widetilde{d}+\epsilon_4)+(\widetilde{a}+\epsilon_1)=(\widetilde{b}+\epsilon_2)\tfrac{u}{v}+(\widetilde{b}+\epsilon_2)\tfrac{x}{y}$. Since $|\frac{x}{y}|,|\frac{u}{v}|\leq 1$, it fails for $\widetilde{a}\neq 0$, $\widetilde{b}=0$ and small $\epsilon,\delta$.
Finally, $c(s)=1$, $\tau(s)=1-s^{2}$, 
$P(s)=\frac{1}{\sqrt{\widetilde{b}}}e^{-i\frac{\pi}{4}}\begin{bsmallmatrix}
s^{2}e^{i\frac{\pi}{4}} & \widetilde{b} s^{-1} \\
s  & s^{-1}
\end{bsmallmatrix}$, 
$B(s)=1\oplus \widetilde{b}^{2}$ 
%
$
\big(\begin{bsmallmatrix}
0 & 1\\
1 & i
\end{bsmallmatrix},
\begin{bsmallmatrix}
0 & \widetilde{b}\\
\widetilde{b} & 0
\end{bsmallmatrix}\big)\to 
\big(\begin{bsmallmatrix}
0 & 1\\
\tau & 0
\end{bsmallmatrix},
1\oplus d\big)
$, 
while, $c(s)=1$, $\tau(s)=1-s^{3}$, 
$P(s)=\frac{1}{\sqrt{2}}e^{i\frac{\pi}{4}}\begin{bsmallmatrix}
s\widetilde{b}e^{-i\alpha(s)} & -is^{-1}e^{i\alpha(s)} \\
-is e^{i\alpha(s)} & (\widetilde{b}s)^{-1}e^{-i\alpha(s)}
\end{bsmallmatrix}$, 
$B(s)=1\oplus \widetilde{b}^{2}e^{4\alpha(s)+\beta(s)}$, $\sin (\alpha(s))=s^{3}$, $\sin(\frac{\beta(s)}{2})=-s^{2}$ 
gives 
$
\big(\begin{bsmallmatrix}
0 & 1\\
1 & 0
\end{bsmallmatrix},
\begin{bsmallmatrix}
0 & \widetilde{b}\\
\widetilde{b} & 1
\end{bsmallmatrix}\big)\to 
\big(\begin{bsmallmatrix}
0 & 1\\
\tau & 0
\end{bsmallmatrix},
1\oplus d\big)
$, $\widetilde{b} > 0$, $d\in \mathbb{C}$.

\item $B=\begin{bsmallmatrix}
0 & b\\
b & e^{i\varphi} 
\end{bsmallmatrix}$, \quad $0\leq \varphi<\pi$, $b>0$



Let $a=0$ and $\widetilde{B}=\widetilde{a}\oplus \widetilde{d}$.
Lemma \ref{lemapsi2} (D\ref{p3bd}) for $\widetilde{a}\neq 0$ implies
$\tfrac{v}{u}=\tfrac{i(-1)^{l}
\sqrt{\widetilde{a}\widetilde{d}}+\epsilon_2'}{\widetilde{a}+\epsilon_1}=i(-1)^{l}\sqrt{\frac{\widetilde{d}}{\widetilde{a}}}+\epsilon_2''$, \small 
$|\epsilon_2'|\leq \left\{\begin{array}{ll}
\tfrac{\epsilon (4|\max\{\widetilde{d},\widetilde{a}\}|+2+|\widetilde{a}\widetilde{d}|)}{|\widetilde{a}\widetilde{d}|}, & \det \widetilde{B}\neq 0\\
\sqrt{\epsilon (4|\max\{\widetilde{d},\widetilde{a}\}|+3)}, & \det \widetilde{B}= 0
\end{array}\right.
$
\normalsize, $|\epsilon_2''|\leq \frac{2}{\widetilde{a}}(|\epsilon_2'|+\epsilon\sqrt{\frac{\widetilde{d}}{\widetilde{a}}})$, $l\in \mathbb{Z}$, provided that $\epsilon\leq \frac{|\widetilde{a}|}{2}$. 
It contradicts (\ref{ocena11xuyv}), (\ref{ocenayv}) for $\omega=i$. If $\widetilde{a}=1$, $\widetilde{d}=|\widetilde{d}|e^{i\widetilde{\vartheta}}\neq 0$, $0<\widetilde{\vartheta}<\pi$ we apply (\ref{ocenah}) to deduce $\psi=\kappa-\eta-\frac{\widetilde{\vartheta}}{2}-\frac{\pi}{2}-l\pi$ with $|\sin \psi|\leq \frac{|\epsilon_2''|}{|\sqrt{\widetilde{d}}|}$. Hence
\begin{align*}
\overline{x}v+\overline{u}y=\overline{x}u\tfrac{v}{u}+y\overline{v}\tfrac{\overline{u}}{\overline{v}}
=-(-1)^{l}e^{i(\frac{\widetilde{\vartheta}}{2}+\psi)}\big(\Ima(\overline{x}u)|\tfrac{v}{u}|+\Ima(y\overline{v})|\tfrac{\overline{u}}{\overline{v}}|\big)
+\Rea(\overline{x}u)\tfrac{v}{u}+\Rea(y\overline{v})\tfrac{\overline{u}}{\overline{v}}.
%
%
\end{align*}
Using (\ref{boundd})) and $\big||\tfrac{v}{u}|-|\sqrt{\frac{\widetilde{d}}{\widetilde{a}}}|\big|\leq|\epsilon_2''|$, the above calculation and (\ref{ocenah}) gives
\vspace{-1mm}
\[
\psi'=k\pi-\big(\tfrac{\widetilde{\vartheta}}{2}+\psi+(l+1)\pi\big),\qquad |\sin \psi'|\leq  2\delta\big(1+|\sqrt{\tfrac{\widetilde{d}}{\widetilde{a}}}|+|\epsilon_2''|+(|\sqrt{\tfrac{\widetilde{d}}{\widetilde{a}}}|-|\epsilon_2''|)^{-1} \big),
\]
\vspace{-1mm}
which fails for small $\epsilon,\delta$ (recall $|\sin \psi|\leq \frac{|\epsilon_2''|}{|\sqrt{\widetilde{d}}|}$, $0<\widetilde{\vartheta}<\pi$).
%
%
%
%
%
%
%
Next, $c(s)=-1$, $P(s)=\begin{bsmallmatrix}
-\frac{2is}{3\widetilde{b}} & \frac{1}{s}\\
\frac{s}{3} & \frac{2i\widetilde{b}}{s}
\end{bsmallmatrix}$, 
$B(s)=\begin{bsmallmatrix}
0 & \widetilde{b} \\
\widetilde{b}  & i
\end{bsmallmatrix}$, $\tau(s)=1-\frac{s^{2}}{2\widetilde{b}}$ implies 
$
 \big(\begin{bsmallmatrix}
0 & 1\\
1 & i
\end{bsmallmatrix},
\begin{bsmallmatrix}
0 & \widetilde{b}\\
\widetilde{b} & 0
\end{bsmallmatrix}\big)\to 
\big(\begin{bsmallmatrix}
0 & 1\\
\tau & 0
\end{bsmallmatrix},
\begin{bsmallmatrix}
0 & b\\
b & e^{i\varphi}
\end{bsmallmatrix}\big)
$, $\widetilde{b}> 0$.







\quad
For $\tau(s)=\left\{
\begin{array}{ll}
1-s\sqrt{\widetilde{a}+s}, & \omega=i\\
1-s^{2}, & \omega=0
\end{array}\right.$,  
$P(s)=e^{i\frac{\pi}{4}}\begin{bsmallmatrix}
\sqrt{\widetilde{a}+s} & \frac{-i}{s}\\
s^{3}e^{-i\frac{\pi}{4}} & \frac{1}{\sqrt{a}+s}
\end{bsmallmatrix}$, 
$B(s)=\begin{bsmallmatrix}
-i & \frac{\sqrt{a}+s}{s} \\
\frac{\sqrt{\widetilde{a}+s}}{s} & -i(\widetilde{a}+s)(\widetilde{d}-\frac{1}{s^{2}})
\end{bsmallmatrix}$, $c(s)=-1$ we get 
$
\big(\begin{bsmallmatrix}
0 & 1\\
1 & \omega
\end{bsmallmatrix},\widetilde{a}\oplus \widetilde{d},
\big)\to 
\big(\begin{bsmallmatrix}
0 & 1\\
\tau & 0
\end{bsmallmatrix},
\begin{bsmallmatrix}
e^{i\varphi} & b\\
b & d
\end{bsmallmatrix}\big)
$, $d\in \mathbb{C}$, $\widetilde{a}\geq 0$. 

\item 
$B=\begin{bsmallmatrix}
0 & b \\
b & 0 
\end{bsmallmatrix}$, $b>0$

We multiply the first and the second equality of (D\ref{p3bd}) and (D\ref{p3ba}) of Lemma \ref{lemapsi2} for $\widetilde{a}\neq 0$ to get a contradiction with (\ref{ocena11xuyv}) for $\omega=i$ and small $\epsilon,\delta$.
Taking $c(s)=1$, $\tau(s)=1-s$, 
$P(s)=e^{-i\frac{\pi}{4}}\begin{bsmallmatrix}
se^{i\frac{\pi}{4}} & s^{-1} \\
s  & i
\end{bsmallmatrix}$, 
$B(s)=\frac{\widetilde{d}+s}{2}s\begin{bsmallmatrix}
0 & 1 \\
1 & 0
\end{bsmallmatrix}$ 
shows  
$
\big(\begin{bsmallmatrix}
0 & 1\\
1 & i
\end{bsmallmatrix},
0\oplus \widetilde{d}\big)\to 
\big(\begin{bsmallmatrix}
0 & 1\\
\tau & 0
\end{bsmallmatrix},
\begin{bsmallmatrix}
0 & b\\
b & 0
\end{bsmallmatrix}\big)
$.

\quad
For $\widetilde{a}=0$, $\widetilde{b}>0$ we have $b=\widetilde{b}+\delta'$, with $|\delta'|\leq \delta\widetilde{b}^{2}+\max\{\epsilon,\frac{\delta^{2}}{\nu^{2}}\}(4\max\{\widetilde{b},\widetilde{b}^{2},1\}+2)^{2}$ (see Lemma \ref{lemadet} (\ref{pathlema2ii} and (\ref{ocenakoren})); recall $1-\tau\leq \delta$).  
If $\omega=i$ (hence $\widetilde{d}=0$) the last equation of (\ref{eqadb}) for $\widetilde{d}=0$ contradicts the second estimate of (\ref{ocena11xuyv}). Next, let  $\omega=0$ (hence $\widetilde{d}=1$).
Using $2bvy=1+\epsilon_4$ (see (\ref{eqadb})) and $|\Rea(\overline{y}v)|\leq \delta$ (see (\ref{boundd})), we have $|\Ima (\overline{y}v)|\geq |yv|-|\Rea(\overline{y}v)|\geq \frac{1-\epsilon}{\overline{b}+|\delta'|}-\delta$.
%
%
Further Lemma \ref{lemapsi2} gives $2bvx=((-1)^{l+1}+1)\widetilde{b}+\epsilon_2'$, $2buy=((-1)^{l}+1)\widetilde{b}+\epsilon_2''$, $l\in \mathbb{Z}$, where $|\epsilon_2'|,|\epsilon_2''|\leq 
\frac{\epsilon(4\max\{1,\widetilde{b}\}+2+\widetilde{b}^{2})}{\widetilde{b}^{2}}$. So either $2bvx=2\widetilde{b}+\epsilon_2'$, $2buy=\epsilon_2''$ or $2buy=2\widetilde{b}+\epsilon_2''$, $2bvx=\epsilon_2'$. In the first case we also have $\overline{x}v=(-1)^{k}+\delta_2'$ with $|\delta_2'|\leq \delta+\frac{|\epsilon_2''|}{2(\widetilde{b}-|\delta'|)}$ (see (\ref{boundd})). 
We combine all facts:
\begin{equation*}
|\tfrac{y}{x}|^{2}=\tfrac{2bvy \,\overline{y}v}{2bvx\,\overline{x}v}=\tfrac{(1+\epsilon_4)(i\Ima (\overline{y}v)+\delta_0)}{(2\overline{b}+\epsilon_2')((-1)^{k}+\delta_2')}
\end{equation*}
For sufficiently small $\epsilon,\delta$ the right-hand (the left-hand) side is (not) real, a contradiction. 
The other case is treated similarly and yields a contradiction as well.
\end{enumerate}

\item \label{u1000}
$
(1\oplus 0,
\widetilde{B})\dashrightarrow
(1\oplus 0,
B)
$

If $B=\begin{bsmallmatrix}
0 & 1\\
1 & 0
\end{bsmallmatrix}$, 
$\widetilde{B}=
\widetilde{a}\oplus 1
$, 
$\widetilde{a}\geq 0$,
then \cite[Theorem 3.6, Case XI (a)]{ST} applies. (Taking $c(s)=1$, $P(s)=\begin{bsmallmatrix}
1 & s \\
\frac{\widetilde{a}}{2}  & 0
\end{bsmallmatrix}$ in (\ref{cPepsi}) proves $
{\ (1\oplus 0,
\widetilde{a}\oplus 0)\to 
\big(1\oplus 0,
\begin{bsmallmatrix}
0 & 1\\
1 & 0
\end{bsmallmatrix}\big)}
$.)

\quad
Next, Lemma \ref{lemapsi1} (\ref{lemapsi11}) with (C\ref{r11}) for $\alpha=1$ gives 
%
$\big||x|^2-1\big|\leq \delta$ and $|y|^{2}\leq \delta$, hence $|\frac{y}{x}|^{2}\leq \frac{\delta}{1-\delta}$.
%
%
%
%
%
%
When $B=a\oplus 0$ for $a\geq 0$,
then dividing the last two equalities of (\ref{eqBF1}) for $b=d=\widetilde{b}=0$, $\widetilde{d}=1$ yields  
$\frac{x}{y}=\frac{\epsilon_2}{1+\epsilon_4}$, which contradicts $|\frac{y}{x}|^{2}\leq \frac{\delta}{1-\delta}$ for small $\epsilon,\delta$.

\quad
Finally, 
$c(s)=1$, 
$P(s)=\begin{bsmallmatrix}
1 & 0 \\
\sqrt{\widetilde{a}-a} & s
\end{bsmallmatrix}$ in (\ref{cPepsi})
proves 
$(1\oplus 0,\widetilde{a}\oplus 0)\to(1\oplus 0,a\oplus 1)$, $a\geq 0$,
and $c(s)=1$, $P(s)=\begin{bsmallmatrix}
i & s^{3} \\
s^{-1}  & s
\end{bsmallmatrix}$, $B(s)=\frac{1}{s^{2}}\oplus 1$ implies $
\big(1\oplus 0,
\begin{bsmallmatrix}
0 & 1\\
1 & 0
\end{bsmallmatrix}\big)\to 
\big(1\oplus 0,
a\oplus 1\big)
$, $a>0$.


\item \label{g100110}
$
(1\oplus 0,
\widetilde{B})\to 
\big(\begin{bsmallmatrix}
0 & 1\\
1 & 0 
\end{bsmallmatrix},
B\big)
$ 


Lemma \ref{lemapsi1} (\ref{lemapsi11}) for (C\ref{r7}) for $\alpha=1$, $\beta=\omega=0$ yields
\begin{equation}\label{est144}
2\Rea(\overline{x}u)=(-1)^{k}+\delta_1,\,
2\Rea(\overline{y}v)=\delta_2, \,  \overline{x}v+\overline{u}y= \delta_4, \quad k\in \mathbb{Z}, |\delta_1|,|\delta_2|,|\delta_4|\leq \delta.
\end{equation}
Next, (\ref{yoxvou}) (compare (\ref{est144}) with (\ref{eqBF8})) is valid in this case as well. Since $|\det P|\leq \frac{\delta\sqrt{6}}{\nu}$ by Lemma \ref{lemadet} (\ref{lemadeta}), it follows from  (\ref{yoxvou}) that 
\begin{equation}\label{vuyxP1}
|\tfrac{v}{u}|,|\tfrac{y}{x}|\leq \delta\tfrac{\nu+\sqrt{6}}{\nu(1-\delta)}.
\end{equation}

\begin{enumerate}[label=(\alph*),wide=0pt,itemindent=2em,itemsep=6pt]

\item $B=1\oplus 0$

The bundle consists of one orbit, hence \cite[Theorem 3.6, Case XV (c)]{ST} applies.
(We take $c(s)=1$ and 
$P(s)=\begin{bsmallmatrix}
\sqrt{\widetilde{a}+s} & 0 \\
\frac{1}{2\sqrt{\widetilde{a}+s}}& s
\end{bsmallmatrix}$ to get 
$(1\oplus 0,
\widetilde{a}\oplus 0)\to
\big(\begin{bsmallmatrix}
0 & 1\\
1 & 0 
\end{bsmallmatrix},
1\oplus 0\big)
$ for $\widetilde{a}\geq 0$.)

\item \label{g100110a} 
$B=1\oplus d$, \quad $\Ima (d)>0$

For $\widetilde{B}=\widetilde{a} \oplus 1$ we have (\ref{eqBFadad}) with $\widetilde{b}=0$, $a=1$.  
By multiplying the second equation of (\ref{eqBFadad}) for $\widetilde{b}=0$ with $\delta_4:=\frac{v}{u}$, $\delta_5:=\frac{y}{x}$ and by simplifying it we obtain
%
\begin{equation}\label{advuyx}
ax^{2}\delta_4\delta_5+dv^{2}=\epsilon_2\delta_4, \qquad ay^{2}+dv^{2}\delta_4\delta_5=\epsilon_2\delta_5,
\end{equation}
respectively.
We add these equalities and using the first and the last equation of (\ref{eqBFadad}) we get the equality that fails for $\widetilde{d}\neq 0$ and $\epsilon,\delta$ small enough (recall (\ref{vuyxP1})):
\begin{equation}\label{vsotab}
\epsilon_2(\delta_4+\delta_5)=(ax^{2}+du^{2})\delta_4\delta_5+(ay^{2}+dv^{2})=(\widetilde{a}+\epsilon_1)\delta_4\delta_5+\widetilde{d}+\epsilon_4.
\end{equation}

\quad
Note that 
$\big(1\oplus 0,
\begin{bsmallmatrix}
0 & 1\\
1 & 0 
\end{bsmallmatrix}\big)\to
\big(\begin{bsmallmatrix}
0 & 1\\
1 & 0 
\end{bsmallmatrix},
1\oplus d\big)
$ will follow after we prove $\big(1\oplus 0,
\begin{bsmallmatrix}
0 & 1\\
1 & 0 
\end{bsmallmatrix}\big)\to
\big(1\oplus -1,\begin{bsmallmatrix}
0 & b\\
b & 0 
\end{bsmallmatrix}\big)
$ (see \ref{case10to1t} \ref{case10to1ta}).

\item 
$B=\begin{bsmallmatrix}
0 & b\\
b & 1
\end{bsmallmatrix}$, $b>0$



Let $\widetilde{B}=\widetilde{a}\oplus 1$, $\widetilde{a}\geq 0$.
From Lemma \ref{lemapsi2} (D\ref{p3bd}) for $\widetilde{b}=0$, $\widetilde{d}=1$ 
we get:
\begin{equation}\label{uvb0d1}
|u|\leq \tfrac{\sqrt{\widetilde{a}}+|\epsilon_2''|}{1-\epsilon}|v|,
\end{equation}
which clearly contradicts (\ref{vuyxP1}) for sufficiently small $\epsilon,\delta$.

\quad
For
$P(s)=\begin{bsmallmatrix}
-\frac{1}{2}s & s^{4}\\
s^{-1} & 2s 
\end{bsmallmatrix}$, $c(s)=1$, 
$B(s)=\begin{bsmallmatrix}
0 & s^{-2} \\
s^{-2}   & 1
\end{bsmallmatrix}$
we show 
$
\big(1\oplus 0,
\begin{bsmallmatrix}
0 & 1\\
1 & 0
\end{bsmallmatrix}\big)
\to 
\big(\begin{bsmallmatrix}
0 & 1\\
1 & 0
\end{bsmallmatrix},
\begin{bsmallmatrix}
0 & b\\
b & 1
\end{bsmallmatrix}\big)
$.


\end{enumerate}


\item \label{g100110tau}
$ (1\oplus 0,
\widetilde{B})\dashrightarrow 
\big(\begin{bsmallmatrix}
0 & 1\\
\tau & 0 
\end{bsmallmatrix},
\begin{bsmallmatrix}
a & b\\
b & d 
\end{bsmallmatrix}\big)
$, \quad $0\leq \tau <1$


From Lemma \ref{lemapsi1} (\ref{lemapsi11}) for (C\ref{r4}) with $\alpha=1$ we get
\begin{align}\label{ocena5psi1}
&\Rea(\overline{y}v)\leq \delta, \quad (1-\tau)\Ima (\overline{y}v)\leq \delta, \quad (1-\tau)|\overline{x}v|\leq \delta, \quad (1-\tau)|\overline{u}y|\leq \delta,\\
&\overline{x}v+\overline{u}y\leq \delta, \qquad \big|(1+\tau)\Rea(\overline{x}u)+i(1-\tau)\Ima(\overline{x}u)-\tfrac{1}{c}\big|\leq \delta.\nonumber
\end{align}
%
%
%

The last estimate yields either $\big|(1+\tau)\Rea(\overline{x}u)\big|\geq \frac{1-\delta}{2}$ or $\big|(1-\tau)\Ima(\overline{x}u)\big|\geq \frac{1-\delta}{2}$, thus
\begin{equation}\label{ocenaxuge} 
|\overline{x}u|\geq \tfrac{1-\delta}{4}.
\end{equation}

\begin{enumerate}[label=(\alph*),wide=0pt,itemindent=2em,itemsep=6pt]



\item \label{g100110taua}
$B=\begin{bsmallmatrix}
a & b\\
b & d 
\end{bsmallmatrix}$,
either $b> 0$ or $b=0$ and $ad=0$



First, let $\widetilde{B}=\widetilde{a}\oplus 1$, $\widetilde{a}\geq 0$; we have (\ref{uvb0d1}).
%
%
Using (\ref{ocena5psi1}), (\ref{ocenaxuge}) we thus get 
\[
\delta \tfrac{\sqrt{\widetilde{a}}+|\epsilon_2''|}{|1-\epsilon|}\geq (1-\tau)|\overline{x}v||\tfrac{u}{v}|=(1-\tau)|\overline{x}u|\geq \tfrac{1}{4}(1-\tau).
\]
Similarly, when $d=0$ then Lemma \ref{lemapsi2} (D\ref{p3ba}) for $\widetilde{b}=0$, $ \widetilde{d}=1$ and (\ref{ocena5psi1}), (\ref{ocenaxuge}) yield $|\frac{x}{y}|\leq \frac{\sqrt{\widetilde{a}}+|\epsilon_2'|}{1-\epsilon}$ and
$\tfrac{4\delta(\sqrt{\widetilde{a}}+|\epsilon_2'|)}{1-\epsilon}\geq 1-\tau$.
From Lemma \ref{lemadet} (\ref{lemadeta}) we obtain $\sqrt{\tau}|\det P|\leq \frac{\delta\sqrt{6}}{\nu}$. By combining the above statements with (\ref{xvuyre}), (\ref{xvuyre2}) we get $\Rea (\overline{x}u)\leq C\delta$, where a constant $C>0$ can be computed. Hence $(1-\tau) \Ima (\overline{x}u)\geq 1-\delta-C\delta$, and further 
\begin{equation}\label{yvfxu}
\frac{|yv|}{|xu|}=\frac{(1-\tau)|yv|}{(1-\tau)|xu|} \leq
\tfrac{\big|(1-\tau)\Ima(\overline{y}v)\big|+\big|(1-\tau)\Rea(\overline{y}v)\big|}{\big|(1-\tau)\Ima(\overline{x}u)\big|-\big|(1-\tau)\Rea(\overline{x}u)\big|}\leq \frac{2\delta}{1-\delta-2C\delta}.
\end{equation}
It is also easy to validate 
\begin{equation}\label{ocexvuy}
|\overline{x}v+\overline{u}y||\tfrac{u}{v}|\geq |ux|\big|1-\tfrac{|yv|}{|xu|}|\tfrac{u}{v}|^{2}\big|, \qquad
|\overline{x}v+\overline{u}y||\tfrac{x}{y}|\geq |ux|\big|1-\tfrac{|yv|}{|xu|}|\tfrac{x}{y}|^{2}\big|.
\end{equation}
We apply (\ref{ocena5psi1}) and the estimates on $|\frac{u}{v}|$, $|\frac{x}{y}|$, $|\frac{yv}{xu}|$ to (\ref{ocexvuy}) to get a contradiction for small $\epsilon,\delta$.
Next,  
$P(s)=\begin{bsmallmatrix}
-se^{i(\alpha(s)+\frac{\pi}{4})} & s^{3} \\
s^{-1}e^{i\frac{\pi}{4}} & 1
\end{bsmallmatrix}$, 
$B(s)=\begin{bsmallmatrix}
s^{-4}e^{-i\alpha(s)} & s^{-2} \\
s^{-2} & 1
\end{bsmallmatrix}$, 
$c(s)=-1$, $\tau(s)\to 0$, $\sin (\frac{\alpha(s)}{2})=\frac{\widetilde{a}s^{2}}{2}$
implies 
$
(1\oplus 0,
\widetilde{a}\oplus 1) 
\to 
\big(\begin{bsmallmatrix}
0 & 1\\
0 & 0
\end{bsmallmatrix},
\begin{bsmallmatrix}
\zeta^{*} & b\\
b & 1
\end{bsmallmatrix}\big)
$, $\zeta^{*}\in \mathbb{C}^{*}$.

\quad
Let $B=\begin{bsmallmatrix}
0 & b\\
b & 0 
\end{bsmallmatrix}$, $b>0$, $\widetilde{B}=\begin{bsmallmatrix}
0 & 1\\
1 & 0 
\end{bsmallmatrix}$. The first (the second) equation of (\ref{eqadb}) for $\widetilde{a}=0$ (for $\widetilde{b}=1$) combined with (\ref{ocenaxuge}) (with (\ref{ocena5psi1}) for $0\leq \tau\leq \frac{1}{2}$) yields $\epsilon\geq b|ux|\geq b\frac{1-\delta}{4}$ (and $1+\epsilon\geq b|vx+uy|\geq 4b\delta$), thus a contradiction for sufficiently small $\epsilon,\delta$ and $0\leq \tau\leq \frac{1}{2}$. If $1\geq \tau\geq \frac{1}{2}$ then Lemma \ref{lemadet} (\ref{lemadeta}),  (\ref{lemadetb}) leads to $|\det P|\leq \frac{2\sqrt{3}\delta}{\nu}$ and $b|\det P|\geq 1- 6\epsilon$, hence $\frac{8\sqrt{3}\epsilon \delta}{\nu(1-\delta)}\geq b \frac{2\sqrt{3}\delta}{\nu}\geq 1-6\epsilon$, which fails for small $\epsilon,\delta$.
Taking 
$P(s)=\begin{bsmallmatrix}
-s & s^{4} \\
s^{-1} & 2s
\end{bsmallmatrix}$, 
$B(s)=\begin{bsmallmatrix}
\zeta & \frac{1}{2}s^{-2} \\
\frac{1}{2}s^{-2} & 1
\end{bsmallmatrix}$ with $\frac{\zeta}{s^{2}}\to 0$ and $P(s)=\begin{bsmallmatrix}
s^{-1} & 2s \\
-s & s^{4}
\end{bsmallmatrix}$, 
$B(s)=\begin{bsmallmatrix}
1 & \frac{1}{2}s^{-2} \\
\frac{1}{2}s^{-2} & 0
\end{bsmallmatrix}$ (both with $c(s)=-1$, $\tau(s)\to 0$) in (\ref{cPepsi}) proves 
$
\big(1\oplus 0,
\begin{bsmallmatrix}
0 & 1\\
1 & 0
\end{bsmallmatrix}\big)\to 
\big(\begin{bsmallmatrix}
0 & 1\\
\tau & 0
\end{bsmallmatrix},
\begin{bsmallmatrix}
\zeta & b\\
b & e^{i\varphi}
\end{bsmallmatrix}\big)
$, $\zeta\in \mathbb{C}$ 
and $
\big(1\oplus 0,
\begin{bsmallmatrix}
0 & 1\\
1 & 0
\end{bsmallmatrix}\big)\to 
\big(\begin{bsmallmatrix}
0 & 1\\
\tau & 0
\end{bsmallmatrix},
\begin{bsmallmatrix}
e^{i\varphi} & b\\
b & 0
\end{bsmallmatrix}\big)
$ with
$b>0$, $0\leq\varphi<\pi$, respectively. 

\quad
Finally,
to see 
$
(1\oplus 0,
\widetilde{a}\oplus 0)\to 
\big(\begin{bsmallmatrix}
0 & 1 \\
\tau  & 0
\end{bsmallmatrix},
B\big)
$,
$0\leq \tau <1$, 
$\widetilde{a}\geq 0$, where $B$ is any of matrices 
$
\begin{bsmallmatrix}
a & b\\
b & e^{i\varphi}
\end{bsmallmatrix}
$, $a,b\geq 0$ and 
$\begin{bsmallmatrix}
e^{i\varphi} & b\\
b & d
\end{bsmallmatrix}
$, $d,b\geq 0$,
we take  
$P(s)=\begin{bsmallmatrix}
\frac{1}{\sqrt{\widetilde{a}+s}} & s \\
\sqrt{\widetilde{a}+s} & s
\end{bsmallmatrix}$, 
$B(s)=\begin{bsmallmatrix}
a(s) & b(s) \\
b(s) & 1
\end{bsmallmatrix}$ with $b(s)\to 0$, $\frac{a(s)}{\sqrt{s}}\to 0$  
or 
$P(s)=\begin{bsmallmatrix}
\sqrt{\widetilde{a}+s} & s\\
\frac{1}{\sqrt{\widetilde{a}+s}} & s \\
\end{bsmallmatrix}$, 
$B(s)=\begin{bsmallmatrix}
1 & b(s) \\
b(s) & d(s)
\end{bsmallmatrix}$ with $b(s)\to 0$, $\frac{d(s)}{\sqrt{s}}\to 0$ in (\ref{cPepsi}) ($c(s)=1$, $\tau(s)\to 0$ in both cases).
To prove 
$
(1\oplus 0,
\widetilde{a}\oplus 0)\to 
\big(\begin{bsmallmatrix}
0 & 1 \\
\tau  & 0
\end{bsmallmatrix},
\begin{bsmallmatrix}
0 & b \\
b & 0
\end{bsmallmatrix}\big)
$, 
$b> 0$, 
we put $P(s)=\begin{bsmallmatrix}
1 & s \\
1  & 0
\end{bsmallmatrix}$, $B(s)=\frac{\widetilde{a}+s}{2}\begin{bsmallmatrix}
0 & 1 \\
1  & 0
\end{bsmallmatrix}$, $c(s)=1$, $\tau(s)\to 0$ in (\ref{cPepsi}).
%

\item \label{g100110taub}
$B=a\oplus d$, \quad $a,d\neq 0$

For 
$c(s)=-i$, 
$P(s)=\begin{bsmallmatrix}
s & s^{3} \\
is^{-1} & s^{2}
\end{bsmallmatrix}$, 
$B(s)=\frac{1}{s^{4}}\oplus 1$ we get
$\big(1\oplus,\begin{bsmallmatrix}
0 & 1\\
1 & 0 
\end{bsmallmatrix}\big)\to 
\big(
\begin{bsmallmatrix}
0 & 1\\
0 & 0 
\end{bsmallmatrix},
a\oplus 1\big)$.
%

\quad
If $\tau\leq \frac{1}{2}$ then we have $|xv|,|uy|\leq 2\delta$, thus using (\ref{ocenaxuge} ) we get $|\frac{v}{u}|=|\frac{vx}{ux}|\leq 8\delta$ and $|\frac{y}{x}|=|\frac{uy}{ux}|\leq 8\delta$. On the other hand for $\tau \geq \frac{1}{2}$ we get $|\det P|\leq \frac{2\sqrt{3}\delta}{\nu}$ (Lemma \ref{lemadet} (\ref{lemadeta})), therefore (\ref{xvuyre}), (\ref{xvuyre2}), (\ref{ocena5psi1}) imply $|\Rea (\overline{x}u)||\frac{v}{u}|,|\Rea (\overline{x}u)||\frac{v}{u}|\leq 2\sqrt{3}\delta+\delta$. If $|\Rea (\overline{x}u)|\leq \sqrt{2\sqrt{3}\delta+\delta}$, then $(1-\tau)|\Ima (\overline{x}u)|\geq 1-\sqrt{(2\sqrt{3}+1)\delta}$ and similarly as in (\ref{yvfxu}) we obtain $|\frac{v}{u}|,|\frac{y}{x}|\leq \frac{2\delta}{1-2\sqrt{(2\sqrt{3}+1)\delta}}$.
If $\widetilde{B}=\widetilde{a}\oplus \widetilde{d}$ with $\widetilde{d}\neq 0$, then in any case we proceed mutatis mutandis as in \ref{g100110} \ref{g100110a} to get a contradiction for small $\epsilon,\delta$.

\end{enumerate}


\item \label{case10to1t}
$(1\oplus 0,\widetilde{B})\dashrightarrow
(1\oplus e^{i\theta},B)$, \;$0\leq \theta \leq \pi$

From Lemma \ref{lemapsi1} (\ref{lemapsi11}) with (C\ref{r32}) for $\alpha=1$ and $0< \theta<\pi$ we have 
%
\begin{equation}\label{est0101t}
\big||x|^2+e^{i\theta}|u|^2-c^{-1}\big|\leq \delta,\, \big||y|^2+e^{i\theta}|v|^2\big|\leq \delta, \, \sin(\theta) |\overline{u}v|\leq \delta, \, |\overline{x}y+\cos (\theta) \overline{u}v|\leq \delta.
\end{equation}
Further Lemma \ref{lemapsi1} (\ref{lemapsi11}) with (C\ref{r9}) for $\alpha=1$, $\sigma=-1$ yields that 
\begin{equation}\label{est0101n}
\big||x|^2-|u|^{2}\big|=1 +\delta_1 ,\quad \big||y|^2-|v|^{2}\big|=\delta_4,\quad  |\overline{x}y-\overline{u}v|= \delta_2,\quad |\delta_1|,|\delta_2|,|\delta_4|\leq \delta,
\end{equation}
while from (C\ref{r9}) for $\alpha=1$, $\sigma=1$ we deduce
\begin{equation}\label{est0101p}
|x|^2+|u|^{2}=1+ \delta_1 ,\qquad |y|^{2},|v|^{2}\leq \delta,\quad |\delta_1|\leq \delta.
\end{equation}
\begin{enumerate}[label=(\alph*),wide=0pt,itemindent=2em,itemsep=6pt]

\item \label{case10to1ta} $\widetilde{B}=\begin{bsmallmatrix}
0 & 1\\
1 & 0
\end{bsmallmatrix}$\\

%

\vspace{-2mm}

%
%
%
%
Taking 
$c(s)=1$, $P(s)=\frac{1}{\sqrt{2}}\begin{bsmallmatrix}
1 & s\sqrt{2} \\
i  & -is\sqrt{2}
\end{bsmallmatrix}$, 
$B(s)=\frac{\sqrt{2}}{2s}I_2
$ gives $
 \big(1\oplus 0,
\begin{bsmallmatrix}
0 & 1\\
1 & 0
\end{bsmallmatrix}\big)\to 
(I_2,
aI_2
)
$, 
$
a> 0$.
%
If $B=dI_2$ and $\theta=\pi$, then Lemma \ref{lemadet} (\ref{lemadeta}), (\ref{lemadetb}) gives $|\det P|\leq \frac{\delta\sqrt{6}}{\nu}$ and $d|\det P|\geq 1-6\epsilon$.
The first equation of (\ref{eqBF1}) for $a=d$,$b=\widetilde{a}=0$ yields $\epsilon\geq \big|d(x^2+u^{2})\big|\geq |d|\big||x|^2-|u|^{2}\big|\geq |d|(1-\delta)$ (see (\ref{est0101n})). Thus $\frac{\epsilon\delta\sqrt{6}}{\nu} \geq (1-\delta)(1-6\epsilon)$, which fails for $\epsilon,\delta\leq \frac{1}{12}$.


\item \label{case10to1tbb}
$\widetilde{B}=\widetilde{a}\oplus 0$, $\widetilde{a}\geq 0$

We take $c(s)=e^{-i\theta}$, $P(s)=\begin{bsmallmatrix}
s & s \\
1  & s
\end{bsmallmatrix}$, $B(s)=\begin{bsmallmatrix}
a(s) & b(s) \\
b(s)  & d(s)
\end{bsmallmatrix}$ with $d(s)\to \widetilde{a}$, $sa(s),b(s)\to 0$ 
to prove a path 
$
 (1\oplus 0,
\widetilde{a}\oplus 0)\to 
\big(1\oplus e^{i\theta},
\begin{bsmallmatrix}
a & b \\
b  & d
\end{bsmallmatrix}\big)
$ for $d> 0$, $b\geq 0$, $0\leq \theta \leq \pi$. 

\item \label{case10to1tc}
$\widetilde{B}=\widetilde{a}\oplus 1$, $\widetilde{a}\geq 0$

\begin{enumerate}[label=(\roman*),wide=0pt,itemindent=4em,itemsep=3pt]

\item \label{case10to1tci}
$B=\begin{bsmallmatrix}
a & b\\
b & d
\end{bsmallmatrix}$, $|a|+|d|\neq 0$, $a\neq d$

For $c(s)=1$, $P(s)=\begin{bsmallmatrix}
1 & s \\
0  & s
\end{bsmallmatrix}$, 
$B(s)=\begin{bsmallmatrix}
a(s) & b(s) \\
b(s)  & s^{-2}
\end{bsmallmatrix}$, $a(s)\to \widetilde{a}$, $sb(s)\to 0$ and
$c(s)=e^{-i\theta}$, $P(s)=\begin{bsmallmatrix}
0 & s \\
1  & s
\end{bsmallmatrix}$, 
$B(s)=\begin{bsmallmatrix}
s^{-2} & b(s) \\
b(s)  & d(s)
\end{bsmallmatrix}$ with $d(s)\to \widetilde{a}$, $sb(s)\to 1$,
we get 
$
(1\oplus 0,
\widetilde{a}\oplus 1)\to 
\big(1\oplus e^{i\theta},
\begin{bsmallmatrix}
a & b\\
b & d
\end{bsmallmatrix}\big)
$ for $b\geq 0$, $d>0$ and $b\geq 0$, $a>0$, respectively.
Next, $c(s)=1$, $P(s)=\begin{bsmallmatrix}
-i & 0 \\
is  & \frac{s}{\sqrt{\widetilde{a}+s}}
\end{bsmallmatrix}$, 
$B(s)=(\widetilde{a}+s)\begin{bsmallmatrix}
0 & s^{-1} \\
s^{-1}  & s^{-2}
\end{bsmallmatrix}$ and
$c(s)=e^{-i\theta}$, $P(s)=\begin{bsmallmatrix}
is  & \frac{s}{\sqrt{\widetilde{a}+s}}\\
-i & 0 
\end{bsmallmatrix}$, 
$B(s)=(\widetilde{a}+s)\begin{bsmallmatrix}
s^{-2} & s^{-1} \\
s^{-1}  & 0
\end{bsmallmatrix}$ in (\ref{cPepsi})
imply  
$
(1\oplus 0,
\widetilde{a}\oplus 1)\to 
\big(1\oplus e^{i\theta},
\begin{bsmallmatrix}
0 & b\\
b & d
\end{bsmallmatrix}\big)
$ and 
$
(1\oplus 0,
\widetilde{a}\oplus 1)\to 
\big(1\oplus e^{i\theta},
\begin{bsmallmatrix}
a & b\\
b & 0
\end{bsmallmatrix}\big)
$ for $b,a,d > 0$.


\item $B=\begin{bsmallmatrix}
0 & b\\
b & 0
\end{bsmallmatrix}$, $b> 0$ \quad ($0<\theta\leq \pi$)

%
%
The second estimate of (\ref{est0101t}) gives $(\sin \theta) |v|^{2}\leq \delta$, thus either $ |v|^{2}\leq \sqrt{\delta}$ or $\sin \theta \leq \sqrt{\delta}$ (or both). If $ |v|^{2}\leq \sqrt{\delta}$, then the second estimate of (\ref{est0101t}) (or (\ref{est0101n})) implies $|y|^{2}\leq \delta+\sqrt{\delta}$.
Since we have (\ref{eqabd}) for $\widetilde{d}=1$, we further get $|u|,|x|\leq (\delta+\sqrt{\delta})  \frac{|\sqrt{\widetilde{a}}|+\max\{|\epsilon_2'|,|\epsilon_2''|\}}{(1-\epsilon)}$, which contradicts the first estimate of (\ref{est0101t}) and (\ref{est0101n}).

\quad
Let now $v,y\neq 0$ and $\sin\theta\leq \sqrt{\delta}$. If $\theta\in (0,\frac{\pi}{4})$, then $(1-\cos \theta)|v|^{2}=2(\sin^{2}\frac{\theta}{2})|v|^{2}\leq 2(\sin ^{2}\theta)|v|^{2}$, hence the second estimate of (\ref{est0101t}) yields 
\[
\delta\geq \big||y|^{2}+\cos\theta |v|^{2}\big|\geq \big||y|^{2}+ |v|^{2}\big|-(1-\cos \theta)|v|^{2}\geq \big||y|^{2}+|v|^{2}\big|-2\delta.
\]
Hence $|y|^{2},|v|^{2}\leq 3 \delta$ and it gives a contradiction again.
If $\theta\in (\frac{3\pi}{4},\pi]$, then $|\cos \frac{\theta}{2}|=|\sin\frac{\pi-\theta}{2}|\leq |\sin (\pi-\theta)|$, and by combining it with the first equation in (\ref{eqabd}) for $\widetilde{d}=1$ and the third estimate of (\ref{est0101t}) we get $(\cos \frac{\theta}{2})|u|^{2}\leq  (\sin \theta)|uv||\frac{u}{v}|\leq \delta\frac{\sqrt{\widetilde{a}}+|\epsilon_2'|}{1-\epsilon}$. 
Since $|x|^{2}+e^{i\theta}|u|^{2}=|x|^{2}- |u|^{2}+2(\cos\frac{\theta}{2})|u|^{2}e^{i\frac{\theta}{2}}$, the first estimate of (\ref{est0101t}) yields 
\begin{equation}\label{xucd5}
|x|^{2}-|u|^{2}=c^{-1}+\delta_5, \qquad  |\delta_5|\leq \delta+2\delta\tfrac{\sqrt{\widetilde{a}}+|\epsilon_2'|}{1-\epsilon} .
\end{equation}
%
%
Next, (\ref{eqabd}) for $\widetilde{d}=1$ yields $|\frac{x}{y}|,|\frac{u}{v}|\leq \frac{|\sqrt{\widetilde{a}}|+\max\{|\epsilon_2'|,|\epsilon_2''|\}}{(1-\epsilon)} $.
%
%
%
%
%
From the first estimate of (\ref{est0101t}) (or (\ref{est0101n})) we deduce either $|x|^{2}\geq \frac{1-\delta}{2}$ or $|u|^{2}\geq \frac{1-\delta}{2}$, and the second estimate of (\ref{est0101t}) (or (\ref{est0101n})) gives 
$|y|,|v|\geq \frac{(1-\epsilon)(1-\delta)}{2(|\sqrt{\widetilde{a}}|+\max\{|\epsilon_2'|,|\epsilon_2''|\})}-\sqrt{\delta}$.
%
%
To conclude we use the (\ref{ocenaxysuv}) with (\ref{est0101t}), (\ref{est0101n}) and (\ref{xucd5}) to obtain an inequality that fails for small $\epsilon,\delta$:
\vspace{-1mm}
\begin{equation*}
\delta 
\geq\big|\overline{x}y+(\cos \theta)\overline{u} v\big|\geq \big||\overline{x}y|-|\overline{u} v|\big|-|uv||1-\cos \theta|
\geq \tfrac{1-|\delta_5|}{2\frac{|u|}{|v|}+\frac{\sqrt{1+|\delta_5|}}{|y|}}-\big(\tfrac{|u|}{|v |}+\tfrac{\sqrt{1+|\delta_5|}}{|v|}\big)\delta-2\delta.
\end{equation*}

\vspace{-2mm}

\item $B=aI_2$, $a>0$ \qquad (hence $A=1\oplus \sigma$, $\sigma=e^{i\theta}\in \{1, -1\}$) 

The first equation of (\ref{eqBFadad}) for $a=d$ and (\ref{est0101n}) yield 
\begin{equation}\label{axau}
\tfrac{\epsilon+|\widetilde{a}|}{a}\geq |x^2+u^{2}|\geq \big||x|^2-|u|^{2}\big|.
\end{equation}
If $\sigma=1$, then the last equation of (\ref{eqBFadad}) for $a=d$ and the last estimates in (\ref{est0101p}) imply $\frac{1-\epsilon}{a}\leq y^{2}+v^{2}\leq 2\delta$. Hence (\ref{axau}) gives $\big||x|^2-|u|^{2}\big|\leq \delta_0:=\frac{2\delta(|\widetilde{a}|+\epsilon)}{1-\epsilon}$. The first equation of (\ref{eqBFadad}) further yields that $|x|,|u|\geq \frac{1-\delta}{2}-\delta_0$ with $\frac{|v|}{|u|},\frac{|y|}{|x|}\leq \frac{2\delta}{\frac{1-\delta}{2}-2\delta_0}$.
If $\widetilde{B}=\widetilde{a}\oplus 1$, we proceed mutatis mutandis as in \ref{g100110} \ref{g100110a} to get a contradiction.

\quad
Let $\sigma=-1$.
By Lemma \ref{lemadet} (\ref{lemadeta}), (\ref{lemadetb}) we have $a \frac{\delta\sqrt{6}}{\nu}\geq a|\det P|=|\sqrt{\widetilde{a}}+\delta'|$ with $\delta'\leq \epsilon\frac{4\widetilde{a}+2}{\widetilde{a}}$ if $\widetilde{a}\neq 0$ (or $\delta'\leq \epsilon\sqrt{4\widetilde{a}+2}$ if $\widetilde{a}= 0$). If $\widetilde{a}\neq 0$, we combine it with the first equality of (\ref{est0101n}) and (\ref{axau}), to obtain
$\frac{\delta\sqrt{6}(\epsilon+|\widetilde{a}|)}{\nu(|\sqrt{\widetilde{a}}|-|\delta'|)}\geq\big||x|^2-|u|^{2}\big|\geq(1-\delta)$, which fails for small $\epsilon,\delta$.
Next, if $\widetilde{a}=0$ then (\ref{axau}) and (\ref{est0101n}) imply $a\leq \frac{\epsilon}{1-\delta}$. 
Using the second equation of (\ref{exu2v2}) and (\ref{exuv}) we deduce
$\frac{u}{v}=\frac{ae^{2i\varphi}\delta_2-\epsilon_4}{ae^{2i\varphi}\delta_4-1-\epsilon_4}$, while the last equation of (\ref{eqBFadad}) for $a=d$, $\widetilde{d}=1$ and the second inequality of (\ref{est0101n}) give $2|v^{2}|\geq \frac{1-\epsilon}{2a}-\delta$.
Applying this and (\ref{est0101n}) to (\ref{ocenaxysuv}) leads to an inequality that fails for small $\epsilon,\delta$.

\end{enumerate}

\end{enumerate}

\item \label{b1010}
$
(1\oplus 0,
\widetilde{B})\dashrightarrow
\big(\begin{bsmallmatrix}
0 & 1\\
1 & i 
\end{bsmallmatrix},B\big)
$, 

From Lemma \ref{lemapsi1} (\ref{lemapsi11}) for (C\ref{r10}) with $\alpha=1$, $c^{-1}=e^{i\Gamma}$ we deduce
%
%
%
\begin{equation}\label{est16}
|\overline{x}v+\overline{u}y|\leq \delta, \quad  |v|^{2},|\overline{u}v|\leq \delta, \quad \big|2\Rea(\overline{y}v)\big|\leq \delta , \quad 
\big|2\Rea(\overline{x}u)+i|u|^{2}-e^{i\Gamma}\big|\leq \delta. 
\end{equation}
%
%
%
%
%
%

\begin{enumerate}[label=(\alph*),wide=0pt,itemindent=2em,itemsep=6pt]

\item \label{b1010a}
$
B=
\begin{bsmallmatrix}
0 & b\\
b & 0 
\end{bsmallmatrix}
$, $b>0$

If $\widetilde{B}=\widetilde{a}\oplus 1$ we again have (\ref{eqabd}) for $\widetilde{d}=1$,
and by combining it with $|v|^{2}\leq \delta$ (see (\ref{est16})), we get $|u|\leq \frac{(\sqrt{\widetilde{a}}+|\epsilon_2''|)\sqrt{\delta}}{1-\epsilon}$ with $\epsilon_2''$ is as in (\ref{eqabd}). 
%
%
The last estimate of (\ref{est16}) then yields
%
$|2\Rea (\overline{x}u)|
\geq 1+\delta+\tfrac{\delta(\sqrt{\widetilde{a}}+|\epsilon_2'|)^{2}}{(1-\epsilon)^{2}}$.
%
%
By applying this, $ \frac{\delta\sqrt{6}}{\nu}\geq |\det P|$ (Lemma \ref{lemadet} (\ref{lemadeta})) and the first estimate of (\ref{est16}) to (\ref{xvuyre2}) we get
$|\frac{v}{u}|(1-\delta-\frac{\delta(\sqrt{\widetilde{a}}+|\epsilon_2''|)^{2}}{(1-\epsilon)^{2}})\leq \delta-\frac{\delta\sqrt{6}}{\nu}$, which contradicts (\ref{eqabd})
for small $\epsilon,\delta$.

\quad
Taking 
$B(s)=\frac{1}{s}\begin{bsmallmatrix}
0 & 1\\
1 & 0 
\end{bsmallmatrix}$, $P(s)=
\begin{bsmallmatrix}
s^{2} & s \\
1 & s
\end{bsmallmatrix}$, $c(s)=-i$ 
gives
$\big(1\oplus 0,
\begin{bsmallmatrix}
0 & 1\\
1 & 0 
\end{bsmallmatrix}\big)\to
\big(\begin{bsmallmatrix}
0 & 1\\
1 & i 
\end{bsmallmatrix},\begin{bsmallmatrix}
0 & b\\
b & 0 
\end{bsmallmatrix}\big)$. 
%
%


\item \label{b1010b}
$B=a\oplus d$, \quad $a\geq 0$, $d\in \mathbb{C}$

%


First, $B(s)=\frac{1}{s^{2}}\oplus \widetilde{a}$, $P(s)=
\begin{bsmallmatrix}
s^{2} & s \\
1 & s^{2}
\end{bsmallmatrix}$, $c(s)=-i$ 
give
$(1\oplus 0,
\widetilde{a}\oplus 1)\to
\big(\begin{bsmallmatrix}
0 & 1\\
1 & i 
\end{bsmallmatrix},a\oplus d\big)$, $a>0$.
%
%
%
%

\quad
If $\widetilde{B}=\widetilde{a}\oplus 1$, $\widetilde{a}\geq 0$ and $a=0$, then the last two equations of (\ref{eqBFadad}) for $\widetilde{d}=1$, $a=\widetilde{b}=0$ give $(1+\epsilon_4)u=\epsilon_2v$ with $|\frac{u}{v}|\leq \frac{\epsilon}{1-\epsilon}$.
Thus $|u|^{2}=|\frac{u}{v}||uv|\leq \frac{\epsilon^{2}\delta}{(1-\epsilon)^{2}}$ and $|2\Rea (\overline{u}x)|\geq 1-\delta-\frac{\delta\epsilon^{2}}{(1-\epsilon)^{2}}$. 
By applying this, $ \frac{\delta\sqrt{6}}{\nu}\geq |\det P|$ (Lemma \ref{lemadet} (\ref{lemadeta})) and (\ref{est16}) to (\ref{xvuyre2}) we get  $|\frac{v}{u}|(1-\delta-\frac{\delta\epsilon^{2}}{(1-\epsilon)^{2}})\leq \delta-\frac{\delta\sqrt{6}}{\nu}$; it contradicts $|\frac{u}{v}|\leq \frac{\epsilon}{1-\epsilon}$ for small $\epsilon,\delta$.
%
%
%
%


\end{enumerate}

\vspace{3mm}

\quad
So far we have proved Theorem \ref{izrek} (\ref{izrek3}), (\ref{izrek5}). In particular, we obtained a path from $(1\oplus 0,\widetilde{a}\oplus 0)$ with $\widetilde{a}>0$ to all bundles, except to $(0_2,B)$ for $B\in \mathbb{C}_S^{2\times 2}$ and $(A,0_2)$ for $0_2\neq A\in \mathbb{C}^{2\times 2}$. Further, Theorem \ref{izrek} (\ref{izrek1}), (\ref{1izrek1}) can be concluded for all cases except maybe for $(0_2,1\oplus 0)$.

\item \label{p0010}
$
(0_2,
1\oplus \sigma)\dashrightarrow (A,B)$

\begin{enumerate}[label=(\alph*),wide=0pt,itemindent=2em,itemsep=6pt]

\item \label{p0011}
$\sigma=1$ \quad ($\widetilde{B}=I_2$)


%


We prove $(0_2,I_2)\to
\big(A,\begin{bsmallmatrix}
a & b\\
b & d
\end{bsmallmatrix}\big)
$, $b>0$, $A\in \mathbb{C}^{2\times 2}$ by taking 
$P(\epsilon)=\frac{s}{\sqrt{2}}e^{i\frac{\pi}{4}}\begin{bsmallmatrix}
1 & -i \\
-i & 1
\end{bsmallmatrix}$, $c(s)=1$, $B(s)=\begin{bsmallmatrix}
a(s) & s^{-2}\\
s^{-2} & d(s)
\end{bsmallmatrix}$, $a(s),d(s)\leq \frac{1}{s}$.
Next, $P(s)=\frac{1}{\sqrt{2}}\begin{bsmallmatrix}
s & s\\
1 & -1
\end{bsmallmatrix}$, $B(s)=\frac{1}{s^{2}}\oplus 1$, $c(s)=1$ give  
$(0_2,I_2)\to (A,a\oplus 1)$ with $a>0$ and either $A=\begin{bsmallmatrix}
0 & 1\\
0 & 0
\end{bsmallmatrix}$ or $A=1\oplus 0$, while $P(s)=\frac{1}{\sqrt{2}}\begin{bsmallmatrix}
s & s\\
s & -s
\end{bsmallmatrix}$, $B(s)=\frac{1}{s^{2}}\oplus (\frac{1}{s^2}+ \frac{d-a}{s})$, $c(s)=1$ yield   
$(0_2,I_2)\to(1\oplus \sigma,a\oplus d)$, $d\geq a>0$.

\item $\sigma=0$ \quad ($\widetilde{B}=1\oplus 0$)

\quad
To prove $
(0_2,
1\oplus 0)\to  
(1\oplus 0,
a \oplus 0)
$ for $a> 0$
we take 
$B(s)=\frac{1}{s^{2}}\oplus 0$, 
$P(\epsilon)=sI_2$, $c(s)=1$ in (\ref{cPepsi}). From what we proved so far this implies $
(0_2,
1\oplus 0)\to  
(A,
B)
$ for all $B\neq 0_2$.

\end{enumerate}

\end{enumerate}


This completes the proof of the theorem.
\end{proof}

\vspace{-2mm}

\section*{Data Availability}
Data sharing not applicable to this article as no datasets were generated or analysed during the current study.

\vspace{-2mm}

\section*{Funding}
The research was supported by Slovenian Research Agency (grant no. P1-0291 and no. J1-3005).

\end{document}